\documentclass[11pt]{amsart}
\usepackage[utf8]{inputenc}
\usepackage{graphicx,amssymb,amsmath,amsthm,wrapfig}
\usepackage{caption}
\usepackage{subcaption}
\usepackage{indentfirst}
\usepackage{cancel}
\usepackage{lipsum}
\usepackage{epstopdf}
\usepackage{epsfig}
\usepackage[normalem]{ulem}
\usepackage{soul}
\usepackage{comment}
\usepackage{enumerate,color}   
\usepackage[overload]{empheq}

\newcommand{\veps}{\varepsilon}
\newcommand{\ve}{\varepsilon}
 \newcommand{\XX}{\mathfrak{X}}
   
\newcommand{\wzero}{w_0}
\newcommand{\zzero}{z_0}

\newcommand{\W}{\overline{W}}

\newcommand{\Zt}{\tilde Z}
\newcommand{\wt}{\widetilde}

\numberwithin{equation}{section}
\usepackage{soul}
\everymath{\displaystyle}

\makeatletter

\renewcommand\subsubsection{\@secnumfont}{\bfseries}%
\renewcommand\subsubsection{\@startsection{subsubsection}{3}
  \z@{.5\linespacing\@plus.7\linespacing}{-.5em}%
  {\normalfont\bfseries}}
  
  \makeatother

\theoremstyle{plain}
\newtheorem{theorem}{Theorem}[section]
\newtheorem{lemma}[theorem]{Lemma}
\newtheorem{proposition}[theorem]{Proposition}

\theoremstyle{definition} 
\theoremstyle{remark}

\newtheorem{remark}{Remark}

\newtheorem*{defi*}{Definition}

\usepackage{etoolbox}
\makeatletter
\let\alignts@preamble\align@preamble
\patchcmd{\alignts@preamble}{\displaystyle}{\textstyle}{}{}
\patchcmd{\alignts@preamble}{\displaystyle}{\textstyle}{}{}

\def\alignts{\let\align@preamble\alignts@preamble\start@align\@ne\st@rredfalse\m@ne}

\makeatother

\allowdisplaybreaks
\title[Singularity for the  Euler Poisson system]{Structure of singularities for the Euler-Poisson system of ion dynamics}

\author[J. Bae]{Junsik Bae}
\address[JB]{Department of Mathematical Sciences, Ulsan National Institute of Science and Technology, Ulsan, 44919, Korea}
\email{junsikbae@unist.ac.kr}

\author[Y. Kim]{Yunjoo Kim}
\address[YK]{Department of Mathematical Sciences, Ulsan National Institute of Science and Technology, Ulsan, 44919, Korea}
\email{gomuli3@unist.ac.kr}

\author[B. Kwon]{Bongsuk Kwon}
\address[BK]{Department of Mathematical Sciences, Ulsan National Institute of Science and Technology, Ulsan, 44919, Korea}
\email{bkwon@unist.ac.kr}

\date{\today}

\subjclass{Primary: 	35Q35,  35Q53   Secondary:  	35Q31, 76B25}

\begin{document}
	
	\maketitle 
	
	\begin{abstract}
		We study the formation of singularity for the isothermal Euler-Poisson system arising from plasma physics. 
		Contrast to the previous studies yielding only limited information on the blow-up solutions, for instance, sufficient conditions for the blow-up and the temporal blow-up rate along the characteristic curve, we rather give a constructive proof of singularity formation from smooth initial data. More specifically, employing the stable blow-up profile of the Burgers equation in  the self-similar variables,  we establish the global stability estimate in the self-similar time, which yields the asymptotic behavior of blow-up solutions near the singularity point. Our analysis indicates that the smooth solution to the Euler-Poisson system can develop a cusp-type singularity; it exhibits $C^1$ blow-up in a finite time, while it belongs to $C^{1/3}$ at the blow-up time,
  provided that smooth initial data  are sufficiently  close to the blow-up profile in some weighted $C^4$-topology. We also present a similar result for the isentropic case, and discuss noteworthy differences in the analysis.
		%
		%

		\noindent{\it Keywords}:
		Euler-Poisson system; Singularity formation 
	\end{abstract}

	\section{Introduction}
	We consider the isothermal Euler-Poisson system  in a non-dimensional form:
	\begin{subequations}\label{EP}
		\begin{align}
			& \rho_t +  \left(\rho u\right)_x = 0, \label{EP_1} 
\\ 
			& \rho( u_t  + u u_x ) +K  \rho_x 
			= -  \rho \phi_x, \label{EP_2} 
			\\
			& - \phi_{xx} = \rho - e^\phi, \label{EP_3} 
		\end{align}
	\end{subequations} 
	where $\rho>0$, $u$ and $\phi$ are the unknown functions of $\left(x,t\right) \in \mathbb{R}\times [-\veps,\infty)$ representing the ion density, the fluid velocity for ions and  the electric potential, respectively, and $K> 0$ is a constant. 
	The Euler-Poisson equations under consideration govern the dynamics of ions with a given Boltzmann electron background, including a pressure term. There are a number of variants of Euler-Poisson equations, but this one is most common and significant in the study of the dynamics of ions in plasmas, see \cite{Ch}. We refer to \cite{GGPS} for the rigorous derivation of \eqref{EP} from the two-species Euler-Poisson system.
	
	The Euler-Poisson system \eqref{EP} demonstrates rich dynamics due to the nonlocal effect of electric force, and it has been extensively  studied for the purpose of mathematically justifying various plasma phenomena. %
	These include plasma sheath formation \cite{S, NOS}, quasi-neutral limit problems \cite{GHR}, and plasma solitary waves \cite{BK2, BK}. 
	Especially in \cite{BK},  the asymptotic behavior and stability of solitary waves to \eqref{EP} are investigated. In fact, the global existence of smooth solutions is the first step to study nonlinear stability. 
Regarding global existence, however, no satisfactory theory is available up to date. 
No global existence of smooth solution is known for \eqref{EP}, and only a few results on the finite time blow-up are known. 

In the previous study, \cite{BCK},  the authors prove that smooth solutions to \eqref{EP} leave $C^1$  class in a finite time when the gradients of the Riemann functions are initially large. Since the proof is based on  a  contradiction argument, only limited information on the blow-up solution is obtained such as some sufficient conditions leading $C^1$ blow-up and the temporal blow-up rate.  On the other hand, in \cite{BCK}, the authors find that the derivatives of density and velocity blow up, while their $L^\infty$ norms stay bounded up to the singularity time. In fact, numerical experiments in \cite{BCK} demonstrate that \eqref{EP} may exhibit a cusp-like singularity in a finite time, and it develops into a shock-like one. 

The main objective of the present work is to investigate the exact structure of singularities for \eqref{EP} at the blow-up time, when it occurs. 
To this end, we study the asymptotic behavior of blow-up profile near the singularity point.
		Borrowing the ideas of \cite{BSV, BSV2} used to show the singularity formation for the multi-D Euler equations, we make use of the blow-up profile for the Burgers equation in self-similar variables. To control the derivatives of the electric potential $\phi$, we adopt the idea from \cite{BCK} which utilizes the conserved (Hamiltonian) energy of the Euler-Poisson system \eqref{EP}.  We establish the stability estimates in some topology to show that the blow-up solutions behave   similarly to those of the Burgers equation near the singular point.
				More precisely, 
		our result indicates that the smooth solution to the Euler-Poisson system  exhibits $C^1$ blows up in a finite time, while it belongs to $C^{1/3}$ at the blow-up time, provided that smooth initial data are sufficiently close to the blow-up profile in some weighted $C^4$-topology.  This specifically implies that the solutions we construct here form a singularity of cusp-type, rather than shock-type (jump discontinuities). The question of whether the solutions we construct indeed evolve into shock waves \textit{after the blow-up time} remains open, and is a challenging problem.

	There have been extensive studies  of shock formation for the Euler equations, particularly in the case of one spatial dimension, for which  use is made of the method of characteristics. Employing the Riemann invariants, Lax \cite{L} showed that finite-time singularity may form for general $2\times2$ genuinely nonlinear hyperbolic systems, 
and  John \cite{J} proved 
for $n\times n$ genuinely nonlinear
hyperbolic systems, for which Liu \cite{TL} later generalized the result further.  See \cite{D} for a
more extensive bibliography of 1D hyperbolic systems. 
We also refer to \cite{CGIM, CGM} concerning singularity formation for   two dimensional unsteady
Prandtl’s system and 
the Burgers equation with transverse viscosity. It is also worth mentioning that the works of \cite{OP, Y} study the blow-up profiles of the Burgers equation with nonlocal perturbations.

\subsection{Riemann functions}
	We begin with introducing the \textit{Riemann functions} $w$ and $z$, 
	\begin{equation}\label{w-z}
			w:=u+\sqrt{K}\log \rho,\quad z:=u-\sqrt{K}\log \rho, 
		\end{equation} 
		 associated to the wave speeds $\lambda_{+}$ and $ \lambda_{-}$ of the hyperbolic part of \eqref{EP}, 
	\[ \lambda_{+}:=\frac{w+z}{2}+\sqrt{K}, \quad \lambda_{-}:=\frac{w+z}{2}-\sqrt{K},
	\] 
	respectively. 	The original unknown functions $\rho$ and $u$ are recovered via the relations
	\begin{equation}\label{w-z1}  
	 u = \frac12(w+z), \quad \rho -1 = e^{\frac{w-z}{2\sqrt{K}}} -1.
	\end{equation}
	By setting $\tilde{t}=t/2$, 
	and abusing $t$ for $\tilde{t}$, we see that 
	\eqref{EP} is equivalent to  
	\begin{subequations}\label{iso_EP}
		\begin{align}
			&w_t+ (w+   z+2\sqrt{K}     ) w_x=-2\phi_x, \label{iso_w}
			\\
			&z_t+ (z+ w-2\sqrt{K} ) z_x=-2\phi_x, \qquad\quad (x,t)\in\mathbb{R}\times[-\veps,\infty), \label{iso_z}
			\\
			&-\phi_{xx}=e^{\frac{w-z}{2\sqrt{K}}}-e^{\phi}. \label{iso_phi}
		\end{align} 
	\end{subequations}
	In order to study the exact regularity of blow-up profiles for \eqref{iso_EP}, we  consider  
	 the Burgers equation: 
	\begin{equation}\label{1D-Burgers}
		\partial_t \tilde{w}+\tilde{w}\partial_x \tilde{w} =0.
	\end{equation}

\subsection{Self-similar Burgers profiles}
	Introducing the self-similar variables
	\begin{equation*}\label{change_var-1}
		y\left(x,t\right)= \frac{x}{\left( -t\right)^{3/2}}, \quad s(t)=-\log\left(  - t \right),
	\end{equation*}
	and a new unknown function $\mathcal{W}(y,s)$ defined by  
	\begin{equation*}\label{WZPhi-1}
		\tilde{w}(x,t) = e^{-s/2} \mathcal{W}(y,s), 
	\end{equation*}
	one finds from \eqref{1D-Burgers} that $\mathcal{W}$ satisfies
	\begin{equation}\label{W-eq-0-1}
		\left( \partial_s-\frac{1}{2} \right)\mathcal{W}+\left( \frac{3}{2} y+\mathcal{W} \right)\mathcal{W}_y =0. 
	\end{equation}
	It is known, see \cite{CSW, EF} for instance, that \eqref{W-eq-0-1} admits the \emph{stable}\footnote{The linearized operator around $\overline{W}$ associated with \eqref{W-eq-0-1} has four non-negative eigenvalues $3/2$, $1$, $1/2$, $0$. These eigenvalues are related to the spatial translation, the time translation, the Galilean invariance, and the space-amplitude scaling invariance, respectively. }  steady solution $\overline{W}$,  
%
	i.e., 
	$\overline{W}$ is a solution to \begin{equation}\label{Burgers_SS}
		-\frac12 \overline{W}(y) + \frac32 y \overline{W}'(y) + \W (y) \W'(y) =0.
	\end{equation}
Multiplying \eqref{Burgers_SS} by $\overline{W}^{-4}$
 and integrating the resultant  with respect to $y$,  we find that
\begin{equation}\label{Burgers_SS'}
	y+\overline{W} + c \overline{W}^3 =0,
\end{equation}
where    $c>0$ is a constant of integration. 
 In fact, $\overline{W}$ is a smooth solution to \eqref{Burgers_SS'} for some $c>0$ if and only if it is  a 
smooth solution to \eqref{Burgers_SS}.
We remark that $\{\overline{W}_c: c>0\}$ forms a one-parameter family of smooth solutions to \eqref{Burgers_SS}. 
Differentiating \eqref{Burgers_SS'} with respect to $y$ successfully and evaluating the resultants at  $y=0$, we find that 
$\overline{W}_c'(0)=-1, \overline{W}_c''(0)=0,  \overline{W}_c'''(0)=6c>0,  \overline{W}_c^{(4)}(0)=0$, and so on. 
Moreover, the asymptotic behavior of $\overline{W}_c$ for large $|y|$ can be obtained. Since  $\overline{W}_c(y)\to\mp \infty$     as  $y\to \pm\infty$ from \eqref{Burgers_SS'}, we see that $y^{-1/3} \overline{W}_c(y) \to \mp c^{-1/3}$ as $y\to \pm \infty$. 

Owing to  the space-amplitude scaling invariance,   we set $c=1$ in what follows, without loss of generality.
Then we have
\begin{equation}\label{wx-exp-inf}
y^{-1/3} \overline{W}(y) \to  -1 \quad \text{ and } 	\quad |y^{2/3} \overline{W}'(y)|\rightarrow \frac{1}{3}
	\quad 
	\text{as } |y| \rightarrow \infty 
\end{equation}
and 
 \begin{equation}\label{Wbar-zero}
	\overline{W}'(0)=-1, \quad \overline{W}''(0)=0, \quad \overline{W}'''(0)=6, \quad \overline{W}^{(4)}(0)=0. 
\end{equation}
%
%


Although one can obtain an explicit form of $\overline{W}$ by solving \eqref{Burgers_SS}, it is not useful in our analysis due to its complicated form. We rather use the quantitative properties of $\overline{W}$ which can be obtained from \eqref{Burgers_SS'} itself. We derive the properties of $\overline{W}$ in subsection~\ref{inequalities}.	\subsection{Main result}\label{Initial_subs} 
	We first describe the initial conditions that lead to the finite-time blow up of $\partial_x w$ while $\partial_x z$ remains bounded.
	For simplicity, we let the initial time  $t=-\veps$, where $\veps>0$ is to be chosen  sufficiently small, and we consider the initial value problem for \eqref{iso_EP}  with the initial data 
	\begin{equation}\label{in-H-C} (w_0, z_0)(x) = (w,z)(-\veps,x) \in 
 H^5(\mathbb{R}) \subset 
 C^4(\mathbb{R}).
 \end{equation} 
	 
		We shall focus on the case in which $-\partial_x w_0$ is sufficiently large in a way that $\partial_{x}w_0$ attains its \emph{non-degenerate} global minimum at $x=0$. More specifically, we assume  that 
		\begin{equation}\label{init_w_3}
			 \partial_x \wzero (0)  =-\veps^{-1}, \quad \partial_x^2 \wzero (0)=0, \quad \partial_x^3w_0(0)=6\veps^{-4}.
		\end{equation}
We also assume that  the  derivatives of $w_0(x)$ and $z_0(x)$ satisfy 
	\begin{equation}\label{init_24}
		\begin{split}
			& \| \partial_x \wzero  \|_{L^{\infty}}\leq \veps^{-1},  \quad \|\partial_x^2\wzero \|_{L^{\infty}}\leq \veps^{-5/2},  \quad \|\partial_x^3\wzero \|_{L^{\infty}}\leq 7\veps^{-4}, 
			\\
			& 
			\|\partial_x^4  \wzero  \|_{L^{\infty}}\leq \veps^{-11/2}, \quad \| \zzero \|_{C^4}\leq 1/4. 
		\end{split}
	\end{equation} 
		In addition, we	assume that 	 
	\begin{equation}\label{w-z-L-inf}
			\| \wzero-\zzero \|_{L^{\infty}} \leq \frac{\sqrt{K}}{2}.
	\end{equation}
We remark that the condition \eqref{init_w_3} can be relaxed, which is discussed in subsection~\ref{scaling-IC}. 

To describe the exact blow-up regularity, we further impose the conditions on the asymptotic behavior of $\partial_x \wzero$ near $x=0$ and $x=\infty$ as follows:
	\begin{equation}
		\left|\varepsilon(\partial_x \wzero)\left( x\right)-\overline{W}'\left( \frac{x}{\varepsilon^{3/2}}\right)\right|\leq \min\left\{\frac{(\frac{x}{\varepsilon^{3/2 }})^2}{40(1+(\frac{x}{\varepsilon^{3/2}})^2)}, \frac{1}{24(8+(\frac{x}{\veps^{3/2 }})^{2/3 })} \right\} \label{4.3a}
	\end{equation}
	for all $x\in\mathbb{R}$, where $\overline{W}$ is the stable   self-similar blow-up profile to the Burgers equation solving \eqref{Burgers_SS'} with $c=1$. 
	Lastly we  assume that  
		\begin{equation}\label{tildeP_init}
			 \sup_{x\in \mathbb{R} }(1+x^{2/3}) \left|\rho_0(x) - 1\right| = \sup_{x\in \mathbb{R} }(1+x^{2/3}) \left|e^{\frac{\wzero(x)-\zzero(x)}{2\sqrt{K}}}-1\right|\leq \frac{1}{28}.
		\end{equation}

Throughout the paper,  $C^\beta(\Omega)$ denotes the H\"older space with exponent $\beta$, equipped  with  the associated H\"older norm defined  as 
\begin{equation}\label{holder}
[w]_{C^{\beta}(\Omega) } := \sup_{x,y\in \Omega, x\ne y} \frac{|w(x) - w(y) | }{|x-y|^{\beta}},
\end{equation} 
and $H^k(\Omega)$ denotes the $L^2$-based standard Sobolev space of order $k$. For notational simplicity, we let 
\begin{equation}\label{P-+}
	P_{-}:=\inf_{x\in\mathbb{R}}(w_0(x)-z_0(x)), \quad P_{+}:=\sup_{x\in\mathbb{R}}(w_0(x)-z_0(x)).
\end{equation}

Now we state our main theorem.


	\begin{theorem}\label{mainthm1}
There is a constant $\veps_0=\veps_0(\|(\rho_0 -1, u_0)\|_{L^2}, \inf_{x\in\mathbb{R}} \rho_0, \sup_{x\in\mathbb{R}} \rho_0)>0$ such that for each $\veps\in(0,\veps_0)$, if the initial data   $(\rho_0-1,u_0)\in H^5(\mathbb{R})$ satisfies  
\eqref{in-H-C}--\eqref{tildeP_init}, 
then there is a unique smooth solution $(\rho,u)\in C\left([-\veps,T_\ast); C^4(\mathbb{R})\right)$ to \eqref{EP}, where  the maximal existence time $T_\ast >-\ve$ is finite, and $T_\ast = O(\veps)$.
Furthermore, it holds that 
\begin{enumerate}[(i)]
\item
$ \sup_{t<T_*} \left[ (\rho, u) (\cdot, t) \right]_{C^\beta}<\infty \; \text{ for } \beta \leq 1/3$;
 \item 
 $\lim_{t\nearrow T_\ast} \left[ \rho (\cdot, t) \right]_{C^\beta} =\infty$ and $\lim_{t\nearrow T_\ast} \left[ u (\cdot, t) \right]_{C^\beta} =\infty$  for $\beta>1/3$;
\item   
 for $\beta > 1/3$, the temporal blow-up rate is obtained as
\ \begin{equation*} \left[ \rho (\cdot, t) \right]_{C^\beta},  \left[ u(\cdot, t) \right]_{C^\beta} \sim (T_*-t)^{-\frac{3\beta-1}{2}} 
\end{equation*} 
for all $t$ sufficiently close to $T_*$; \footnote{Here $A(t)\sim B(t)$ means that there is a constant $\theta_0>0$ such that $\theta_0^{-1} B(t) \le A(t) \le \theta_0 B(t)$ for $t$ sufficiently close to $T_*$.}
 \item $\inf_{x\in \mathbb{R},\; t<T_*} \rho(x,t) \ge \rho_*$  and $\sup_{t<T_\ast} \| (  \rho, u)(\cdot, t) \|_{L^\infty}\le M_*$ for some $ \rho_*, M_*>0$. 
 \end{enumerate}
%
	\end{theorem}
	The proof of Theorem~\ref{mainthm1}  is given in the end of Section~\ref{C13_subsec}. 
	We remark that Theorem~\ref{mainthm1} specifically implies that a cusp-type  singularity can develop  from smooth initial data; the smooth solution exhibits $C^1$ blow-up in a finite time, while it still belongs to $C^{1/3}$ at the blow-up time.
	In \cite{BCK}, when the singularity occurs, the density and velocity are known to be bounded, while their derivatives blow up. No further detailed information can be obtained due to its limited ODE approach. However, our result gives more detailed geometric information on the singularity indicating that the singularity exhibits $C^{1/3}$ regularity. Furthermore, the assertion $(ii)$ indicates that the gradients of $\rho$ and  $u$ blow up simultaneously, which cannot be drawn by the approach employed in \cite{BCK}. The temporal blow-up rate is also obtained as in the assertion $(iii)$.

	A similar blow-up result can be established for the isentropic  Euler-Poisson system \eqref{EP-Isen} with the pressure law $P_\gamma (\rho) =  \rho^\gamma/\gamma,  \gamma>1$. 
	Under an appropriate set of the assumptions similar to those  in Theorem~\ref{mainthm1},  
	we obtain the same result as in Theorem~\ref{mainthm1}, i.e., 
	there is a unique smooth solution $(\rho,u)\in C\left([-\veps,T_\ast); C^4(\mathbb{R})\right)$ to \eqref{EP-Isen} whose the maximal existence time $T_\ast$ is finite and 
	  $T_*=O(\veps)$, and  $C^{1/3}$ norm of the solution is bounded. 	A more detailed statement is given in Theorem~\ref{main-isen}   in Section~\ref{Isen}.  The proof is quite similar to that of Theorem~\ref{mainthm1}, except some technical points arising from the different forms of the Riemann functions. We briefly outline the proof in Section~\ref{Isen}, with emphasis on significant differences.

 Outline of the paper: 
In Section~\ref{C13_subsec}, we establish the global stability estimates in the self-similar time and prove Theorem~\ref{mainthm1}. Section~\ref{sec3-boot} and Section~\ref{sec 4} are devoted to close the  bootstrap assumptions, of Proposition~\ref{mainprop} and Proposition~\ref{decay_lem}, respectively. In Section~\ref{Isen}, we discuss a blow-up result for the isentropic case, where we also give a sketch of the proof. In Section~\ref{appendix}, we present various preliminaries and some standard facts that are used in the course of our analysis. 

	\section{Stability estimates in the self-similar variables}\label{C13_subsec} 
	In this section, we establish the global stability estimates in the self-similar time and prove Theorem~\ref{mainthm1}. 
	 First, we rewrite \eqref{EP} in the self-similar variables incorporated with dynamic modulation functions. 
	\subsection{Self-similar variables and modulations}\label{modulation_sec}
	
	We define three dynamic modulation functions $\tau,\kappa, \xi : [-\veps, \infty)\rightarrow \mathbb{R}$ satisfying 
a system of ODEs:
\begin{subequations}\label{modulation-new}
			\begin{align}
\dot \tau & = ( \tau(t) - t) \partial_x z(\xi(t) , t) - 2 ( \tau(t) - t)^2 \partial_x^2 \phi (\xi(t) , t), 
\\
\dot{\kappa}&= \frac{-2( \tau(t) - t)^{-1} \partial_x^3 \phi  (\xi(t) , t)+ ( \tau(t) - t)^{-2} \partial_x^2 z (\xi(t) , t) }{\partial_x^3 w (\xi(t) , t)} - 2 \partial_x \phi (\xi(t) , t), 
\\
\dot{\xi}&= z(\xi(t) , t) + 2\sqrt{K} + \frac{ 2 \partial_x^3 \phi(\xi(t) , t) - (\tau(t) - t)^{-1} \partial_x^2 z(\xi(t) , t) }{ \partial_x^3 w(\xi(t) , t) } + \kappa(t)
\end{align}
\end{subequations}
	with the initial conditions 
	\begin{equation}\label{Modul_init}
		\tau(-\veps)=0, \quad \kappa_0:=\kappa(-\veps)=w_0(0), \quad \xi(-\veps)=0.
	\end{equation}
 Note that the initial value problem \eqref{EP} with \eqref{in-H-C}--\eqref{tildeP_init} admits a unique solution $(w,z, \phi)\in C^4([-\veps,T)\times \mathbb{R})$ for some $T>0$.
This means that all the functions $w,z$ and $\phi$ 
 and their derivatives appearing on the right-hand side of \eqref{modulation-new} are at least $C^1$. By a standard ODE theory, the initial value problem for \eqref{modulation-new} with the initial data \eqref{Modul_init} admits a unique local $C^1$ solution $(\tau,\kappa,\xi)$. 
	Here, we define $T_*$ as the first time such that $\tau(t)=t$,  
	i.e., 
	\begin{equation}\label{T-star}
	T_* := \inf\{ t \in [-\ve,\infty) : \tau(t) = t\}.
	\end{equation}
	Later we will show that $T_*$ is finite and is the blow-up time of $\partial_x w$ in the proof of Theorem~\ref{mainthm1}.
	
	Introducing the   self-similar variables with modulations  
	\begin{equation}\label{change_var}
		y(x,t)= \frac{x-\xi(t)}{(\tau(t)-t)^{3/2}}, \quad s(t)=-\log\left( \tau(t) - t \right),
	\end{equation} 
	we define new functions $(W,Z,\Phi)$ as 
	\begin{equation}\label{WZPhi}
		w(x,t) = e^{-s/2} W(y,s) + \kappa(t),  \quad z(x,t)=Z(y,s), \quad \phi(x,t) = \Phi(y,s).
	\end{equation}
	From  \eqref{Modul_init} and \eqref{change_var}, we see that  
	\begin{equation*}
		s_0 := s(-\ve) =-\log\veps. 
	\end{equation*}
	Using the Ansatz \eqref{WZPhi} for  the system \eqref{iso_EP} and using  
\begin{equation}\label{dx-dy}
	\frac{\partial y}{\partial t} = -\dot{\xi}e^{3s/2}+\frac{3}{2}ye^s(1-\dot{\tau}), \qquad \frac{\partial y}{\partial x} = e^{3s/2}, \qquad \frac{\partial s}{\partial t}=(1-\dot{\tau})e^s, 
\end{equation} 
	which follow from  \eqref{change_var}, 
	we obtain the equations  for $W,Z$ and $\Phi$:
	\begin{subequations}\label{EP2}
		\begin{align}
			& \partial_s W-\frac{1}{2}W+ U^W W_y=-\frac{2\Phi_y e^s}{1-\dot{\tau}}-\frac{\dot{\kappa}e^{-s/2}}{1-\dot{\tau}}, \label{EP2_1} 
			\\ 
			& \partial_s Z+ U^Z Z_y=-\frac{2\Phi_y e^{s/2}}{1-\dot{\tau}}, \label{EP2_2} 
			\\
			& -\Phi_{yy}e^{3s} = e^{\frac{e^{-s/2}W+\kappa-Z}{2\sqrt{K}}} - e^{\Phi}, \label{EP2_3} 
		\end{align}
	\end{subequations}  
where  
	\begin{equation}\label{UW}
	\begin{split}
		&U^W:=\frac{W}{1-\dot{\tau}}+\frac{3}{2}y + \sigma,  \quad   U^Z:=  U^W   -\frac{4\sqrt{K}e^{s/2}}{1-\dot{\tau}},  \\
		& \sigma: = \frac{e^{s/2}}{1-\dot{\tau}} (\kappa-\dot{\xi} + Z + 2\sqrt{K} ).
\end{split}
	\end{equation} 
	
	We let $P:=\frac{e^{-\frac{s}{2}}W+\kappa-Z}{2\sqrt{K}}(=\log \rho)$ for simplicity of notation. Subtracting \eqref{EP2_2} from \eqref{EP2_1},  and using \eqref{dx-dy}, we obtain  
		\begin{equation}\label{Peq}
			\partial_s P+U^Z P_y=-\frac{2W_y}{1-\dot{\tau}}.
		\end{equation} 
			Applying $\partial_y^k$ to \eqref{EP2_1}, for  $k=1,2,3,4$, we have 
	\begin{subequations}\label{EP2_2D}
		\begin{align}
			& \left( \partial_s + 1 +\frac{W_y}{1-\dot{\tau}}+\frac{e^{s/2}}{1-\dot{\tau}}Z_y \right)W_y + U^W W_{yy}=-\frac{2e^s}{1-\dot{\tau}}\Phi_{yy},\label{EP2_2D1} 
			\\ 
			& \left( \partial_s + \frac{5}{2} +\frac{3W_y}{1-\dot{\tau}}+\frac{2e^{s/2}}{1-\dot{\tau}}Z_y \right)W_{yy} + U^W  \partial_y^3W=-\frac{2e^s}{1-\dot{\tau}}\partial_y^3\Phi-\frac{e^{s/2}}{1-\dot{\tau}}Z_{yy}W_y,\label{EP2_2D2} 
			\\
   \begin{split} 
			& \left( \partial_s + 4 +\frac{4W_y}{1-\dot{\tau}}+\frac{3e^{s/2}}{1-\dot{\tau}}Z_y \right)\partial_y^3W + U^W  \partial_y^4W
   \\ & \qquad\qquad  
   =-\frac{2e^s}{1-\dot{\tau}}\partial_y^4\Phi-\frac{e^{s/2}}{1-\dot{\tau}}(\partial_y^3ZW_y+3Z_{yy}W_{yy})-\frac{3W_{yy}^2}{1-\dot{\tau}},\label{EP2_2D3}
   \end{split} 
			\\
			\begin{split}
			& \left( \partial_s + \frac{11}{2} +\frac{5W_y}{1-\dot{\tau}}+\frac{4e^{s/2}}{1-\dot{\tau}}Z_y \right)\partial_y^4W + U^W \partial_y^5W
			\\
			&\qquad \qquad    =  -\frac{2e^s}{1-\dot{\tau}}\partial_y^5\Phi-\frac{10W_{yy}\partial_y^3W}{1-\dot{\tau}}-\frac{e^{s/2}}{1-\dot{\tau}}(\partial_y^4ZW_y+4\partial_y^3ZW_{yy}+6Z_{yy}\partial_y^3W).\label{EP2_2D4}
			\end{split}
		\end{align}
	\end{subequations}
Similarly, applying $\partial_y^k$ to \eqref{EP2_2}, for $k=1,2$, we have 
	\begin{subequations}\label{EP2_3D}
		\begin{align}
			& \left( \partial_s + \frac{3}{2} +\frac{W_y}{1-\dot{\tau}}+\frac{e^{s/2}}{1-\dot{\tau}}Z_y \right)Z_y + U^Z Z_{yy}=-\frac{2e^{s/2}}{1-\dot{\tau}}\Phi_{yy},\label{EP2_3D1}
			\\ 
			& \left( \partial_s + 3 +\frac{2W_y}{1-\dot{\tau}}+\frac{3e^{s/2}}{1-\dot{\tau}}Z_y \right)Z_{yy} + U^Z  \partial_y^3Z=-\frac{2e^{s/2}}{1-\dot{\tau}}\partial_y^3\Phi-\frac{W_{yy}Z_y}{1-\dot{\tau}}.\label{EP2_3D2}   
		\end{align}
	\end{subequations}

We claim that the fact that $(\tau,\xi,\kappa)(t)$ satisfies the ODE system \eqref{modulation-new} is equivalent to 
		\begin{equation}\label{constraint}
			W(0,s) = 0, \quad  W_y(0,s) = -1, \quad W_{yy}(0,s)=0,
		\end{equation}
under the initial conditions \eqref{init_w_3} and \eqref{Modul_init}.  First of all, in view of \eqref{change_var} and \eqref{WZPhi}, the system \eqref{modulation-new} is equivalent to  the following system of ODEs:   
		\begin{subequations}\label{modulation}
			\begin{align}
				\dot{\tau} &= e^{s/2}Z_y(0,s)-2e^s\Phi_{yy}(0,s),
				 \label{mod1} \\
				\dot{\kappa}&=\frac{-2e^{3s/2}\partial_y^3\Phi(0,s)+e^sZ_{yy}(0,s) }{\partial_y^3W(0,s)}-2\Phi_y(0,s)e^{3s/2}, 
				\label{mod2} \\ 
				\dot{\xi}&=Z(0,s)+2\sqrt{K}+\frac{2e^{s/2}\partial_y^3\Phi(0,s)-Z_{yy}(0,s)}{\partial_y^3W(0,s)}+\kappa. \label{mod3}
			\end{align}
		\end{subequations}
It is straightforward to check that $(W(0,s), W_y(0,s), W_{yy}(0,s)) \equiv (0, -1, 0)$ is a solution to a set of equations \eqref{EP2_1}, \eqref{EP2_2D1} and \eqref{EP2_2D2} evaluated at $y=0$, as long as \eqref{modulation} is satisfied. On the other hand, \eqref{init_w_3} is equivalent to \eqref{constraint} at $s=s_0$. Now, the claim follows from the uniqueness.  

		\subsection{Bootstrap argument}\label{boots_sec}
In this section, we establish the global pointwise estimates  by a   bootstrap argument. We introduce the bootstrap assumptions for given  any  time interval $[s_0,\sigma_1]$. 
First, we assume that   for all $s\in [s_0,\sigma_1]$,  
\begin{equation}\label{tau}
	|\dot{\tau}|\leq 2\veps,
\end{equation}
and
	\begin{equation}\label{Boot_3}			 
				\|Z_y(\cdot, s)\|_{L^{\infty}} \le  e^{-3s/2}.  		 
		\end{equation}

		We also assume that the derivatives of $W$ are bounded as follows: for  $s\in[s_0,\sigma_1]$ and $y\in\mathbb{R}$,  
		\begin{subequations}\label{Boot_2}
			\begin{align}
     			& |W_y(y,s)-\overline W'(y)|  \leq \frac{y^2}{10(1+y^2)},\label{Boot_1} 
     			\\
				&|W_{yy}(y,s)| \leq \frac{15|y|}{(1+y^2)^{1/2}}, \label{EP2_1D3}
				\\
				&|\partial_y ^3 W(0,s)-6| \leq 1,  \label{EP2_1D2} 
				\\
				&\|\partial_y ^3 W(\cdot, s) \|_{L^{\infty}} \leq  M_3:=M^{5/6} , \label{EP2_1D5}
				\\
				&\|\partial_y ^4 W(\cdot, s) \|_{L^{\infty}} \leq M , \label{EP2_1D4}
			\end{align}
		\end{subequations}
 		where $M>0$ is a sufficiently large  constant.
 
 The bootstrap assumption \eqref{Boot_2} implies that $W_y$ is close to $\overline{W}'$ near $y=0$ in $C^2$ norm, ensuring the constraints \eqref{constraint} at $y=0$, and that $W$ stays uniformly bounded in $C^4$ norm. 
  In fact, the asymptotic behaviors of $W$ near $y=0$ plays an important role in resolving the degeneracy of the damping term around $y=0$ in the transport-type equations.
We close the bootstrap assumptions \eqref{tau}--\eqref{Boot_2} in the following proposition.  
\begin{proposition}\label{mainprop}
There is a sufficiently small constant $\veps_0=\veps_0(M,\|(w_0,z_0)\|_{L^2},P_\pm)>0$ such that for each $\veps\in (0,\veps_0)$, the following statements hold. Consider the smooth solution $(w,z)$ to \eqref{iso_EP} satisfying \eqref{in-H-C}--\eqref{tildeP_init}. Suppose that \eqref{tau} holds and the associated solution $(W,Z)$ in the self-similar variables satisfies  \eqref{Boot_3} and \eqref{Boot_2}. Then, it holds that for all $s \in [s_0,\sigma_1]$ and $y\in\mathbb{R}$,
			\begin{subequations}\label{Est1}
				\begin{align}
					&|\dot{\tau}|\leq \frac{3}{2}e^{-s},\label{0}
					\\
					&\|Z_y (\cdot, s) \|_{L^{\infty}} \le  \frac{7}{8}e^{-3s/2}, \label{7}
					\\
					&|W_y(y,s)-\overline W'(y)| \leq \frac{y^2}{20(1+y^2)}, \label{1}
					\\ 
					&|W_{yy}(y,s)| \leq \frac{14|y|}{(1+y^2)^{1/2}},   \label{3}
					\\
					&|\partial_y ^3 W(0,s)-6| \leq C\veps^{1/3},  \label{4}
					\\
					&\|\partial_y ^3 W (\cdot, s) \|_{L^{\infty}} \leq \frac{M^{5/6}}{2} ,\label{5}
					\\
					&\|\partial_y ^4 W (\cdot, s) \|_{L^{\infty}} \leq \frac{M }{2} \label{6}
				\end{align}
			\end{subequations}
			where $C>0$ is a constant.
		\end{proposition}

 The proof of Proposition~\ref{mainprop} is rather lengthy, so we decompose it into several lemmas presented in Section~\ref{sec3-boot}.  
The desired estimates are established in  Lemma~\ref{lem-3.7}
--Lemma~\ref{Wy4_lem}.

 Together with the estimates for the higher order derivatives of $Z$ (see Lemma~\ref{difZn}), we will see that Proposition~\ref{mainprop} implies that $C^4$ norm of $(W,Z)$ is uniformly bounded for  all $s\ge s_0$ by a standard  continuation argument. In particular, in view of 	\eqref{T-star}--\eqref{WZPhi}, the estimate \eqref{Boot_3} implies that $\|z_x(\cdot,t)\|_{L^\infty}$ is uniformly bounded in $t\in[-\ve, T_\ast]$.

Here we collect several simple inequalities that are frequently used in the course of our analysis.  Using \eqref{Boot_1} and \eqref{y-w-y2}, we have the uniform bound of $W_y$, 
\begin{equation} \label{Uy1}
| W_y | \le | \overline{W}' | + | W_y - \overline{W}' | \le  \frac{1}{1+\frac{3y^2}{(3y^2+1)^{2/3}}}+\frac{y^2}{10(1+y^2)} \leq 1.
\end{equation} 
By the fundamental theorem of calculus, it follows from $W(0,s)=\overline{W}(0) =  0$,  $|W_y| \leq 1$, $|\overline{W}'|\leq 1$, \eqref{Boot_1} and \eqref{y-w-y2} that
\begin{equation}\label{Wbd1}
|W| \leq |y|, \qquad |\overline{W}| \leq |y|, \qquad |W - \overline{W}| \leq \frac{|y|^3}{30}.
\end{equation}

In order to  study the exact  regularity of $w$ at the blow up time $t=T_\ast$, it is crucial to obtain the decaying property of $W $ as $|y| \to +\infty$. 
We suppose as a part of the bootstrap assumptions that  
\begin{subequations}\label{boot1} 
		\begin{align} 
			&(y^{2/3}+8)|W_y(y,s)-\overline{W}'(y)| \le 1, \label{Utildey_M}
			\\
			&\sup_{ y\in \mathbb{R} } \; (1+ e^{-s} y^{2/3})\left|e^{\frac{e^{-s/2}W+\kappa-Z}{2\sqrt{K}}}-1\right|\leq \frac{1}{14}. \label{Boot_L}
		\end{align}
	\end{subequations}
Note that the estimate \eqref{Boot_L} for $e^{\frac{e^{-s/2}W+\kappa-Z}{2\sqrt{K}}}-1 ( = \rho-1)$ is required to control the electric force term $\Phi$.
\begin{proposition}\label{decay_lem}
There is a sufficiently small constant $\veps_0=\veps_0(\|(w_0,z_0)\|_{L^2},P_\pm)>0$  such that for each $\veps\in(0,\veps_0)$, if  the same assumptions as in  Proposition~\ref{mainprop} and   \eqref{boot1} hold, then  it holds that  
	\begin{subequations}\label{Est2}
		\begin{align}
			&(y^{2/3}+8)|W_y(y,s)-\overline{W}'(y)| \le \frac{24}{25},\label{Utildey_M'}
			\\
			&\sup_{ y\in \mathbb{R} } \; (1+ e^{-s}y^{2/3})\left|e^{\frac{e^{-s/2}W+\kappa-Z}{2\sqrt{K}}}-1\right|\leq \frac{3}{56}.\label{pi}
		\end{align}
	\end{subequations}
\end{proposition}  
The proof of Proposition~\ref{decay_lem} is given in Section~\ref{sec 4}.
  
  			\begin{remark}\label{init_rmk}{(Initial values of $W,Z,\Phi$)}
		The set of initial conditions \eqref{init_w_3}--\eqref{tildeP_init} for \eqref{iso_EP} implies that the new unknowns $(W,Z, \Phi)$  in  the self-similar variables $(y,s)$  satisfy 
		\begin{subequations}\label{EP-W-IC}
			\begin{align}
				&\|Z_y(\cdot, s_0)\|_{L^{\infty}} \le  \frac{\veps^{3/2}}{4}, \label{Zyinit} 
				\\
				& |W_y(y,s_0)-\overline W'(y)| \leq \frac{y^2}{40(1+y^2)},   \label{W-y2}
				\\
				&|W_{yy}(y, s_0)| \leq 1, \label{1D3}
				\\
				&|\partial_y ^3 W(0, s_0)-6| = 0,  \label{1D2} 
				\\
				&\|\partial_y ^3 W(\cdot, s_0) \|_{L^{\infty}} \leq 7, \label{1D5}
				\\
				&\|\partial_y ^4 W( \cdot, s_0) \|_{L^{\infty}} \leq 1, \label{1D4}
				\\
				&(y^{2/3}+8)|W_y(y,s_0)-\overline{W}'(y)| \le \frac{1}{24},\label{Utildey_init}
					\\
				&\sup_{ y\in \mathbb{R} } \; ( 1+ \veps y^{2/3} ) \left|e^{\frac{e^{-s/2}W(y,s_0)+\kappa_0-Z(y,s_0)}{2\sqrt{K}}}-1\right|\leq \frac{1}{28}.\label{pi_init}
			\end{align}
		\end{subequations}
By the Taylor expansion with \eqref{1D2} and \eqref{1D4}, 
			\begin{equation*}
				|W_{yy}(y,s_0)|\leq|y||\partial_y^3W(0,s_0)|+\frac{y^2}{2}\|\partial_y^4W(\cdot,s_0)\|_{L^{\infty}}\leq 6|y|+\frac{y^2}{2}. 
			\end{equation*}
			Combining this with \eqref{1D3}, we obtain
			\begin{equation}\label{1D3'}
				|W_{yy}(y,s_0)|\leq \min\left\{ 6 |y|+\frac{y^2}{2},1\right\}\leq \frac{7|y|}{(1+y^2)^{1/2}}
			\end{equation}
			for all $y \in \mathbb{R}$.

			Also, from \eqref{modulation-new} and \eqref{Modul_init}, we get
			\begin{equation*}
				|\dot{\tau}(-\veps)|\leq\veps |\partial_xz(0,-\veps)|+2\veps^2|\partial_x^2\phi(0,-\veps)|.
			\end{equation*}
   Since
   $$|\partial_x^2\phi(0,-\veps)|= |\rho(0,-\ve) - e^{\phi(0,-\ve)}| \le | \rho(0, -\ve) -1 | + |e^{\phi(0,-\ve)} -1 | <c_0,$$ 
   where $c_0$ is a positive constant depending only on $\| \rho_0 \|_{L^\infty}$ (see Lemma~\ref{LemmaAppen}),   
   we obtain
			\begin{equation}\label{dottau_init}
				|\dot{\tau}(-\veps)|\leq \frac{\veps}{4}+2\veps^2 c_0
    \leq \veps
			\end{equation} 
			for any $0<\veps<\veps_0$, where $\veps_0=\veps_0(\|\rho_0\|_{L^{\infty}})$ is sufficiently small.
  			%
			Hence, \eqref{EP-W-IC}--\eqref{dottau_init} imply that the initial data satisfies the bootstrap assumptions \eqref{tau}--\eqref{boot1}. 
  \end{remark}
  We present a local existence theorem for the initial value problem \eqref{iso_EP} with the initial data \eqref{in-H-C}. 
  \begin{lemma}[Local existence]\label{criterion}
  Let $(w_0, z_0)\in H^5(\mathbb{R})$. Then there is $T\in(-\ve, \infty)$ such that the initial value problem \eqref{iso_EP}  admits a unique solution $(w,z)\in C([-\ve, T); H^5(\mathbb{R}))$. 
As long as the solution exists, the energy conservation holds:
	\begin{equation}\label{H}
		H(t):=\int_{\mathbb{R}}\frac{1}{2}\rho u^2+\mathcal{P}(\rho)+\frac{1}{2}|\phi_x|^2+(\phi-1)e^{\phi}+1\,dx = H(-\veps), 
	\end{equation} 
	where $\mathcal{P}(\rho):=K(\rho\log\rho -\rho+1)$.   Suppose further that \[\lim_{t\nearrow T} \| (w_x, z_x) (t) \|_{L^\infty(\mathbb{R})} <\infty.\] Then  the solution can be continuously extended beyond $t=T$, i.e., there exists a unique solution \[(w,z) \in C([-\ve, T+\delta); H^5(\mathbb{R}))\] for some $\delta>0$.  
  \end{lemma}
  
From \eqref{w-z} and \eqref{w-z1}, we see that  $(\rho-1, u)\in H^5(\mathbb{R})$ if and only if $(w,z)\in H^5(\mathbb{R})$. 
The local well-posedness of \eqref{iso_EP} with the initial data $(w_0, z_0)\in H^5(\mathbb{R})$, which is \eqref{in-H-C}, is equivalent to that of \eqref{EP} with the initial data $(\rho_0-1, u_0)\in H^5(\mathbb{R})$. 
Then, Lemma~\ref{criterion} can be obtained from the local well-posedness for \eqref{EP} with  the initial data $(\rho_0-1, u_0)\in H^5(\mathbb{R})$ that can be proved by a standard energy method, similarly as a standard local existence theory for hyperbolic system. We refer the readers to \cite{LLS} for more details. 
  Note that by the Sobolev embedding, we have $H^5(\mathbb{R}) \hookrightarrow C^4(\mathbb{R})$, which implies $(w,z)\in C^4([-\ve, T) \times \mathbb{R})$. 
		%

		From Remark~\ref{init_rmk} and Lemma~\ref{criterion}, there is a time-interval $[s_0, \sigma_1]$ for some $\sigma_1>s_0$, in which the solution $(W,Z)$ satisfies the bootstrap assumptions \eqref{Boot_2}--\eqref{boot1}.
\subsection{Global continuation and Proof of Theorem~\ref{mainthm1}}\label{global-conti}
We claim that the desired estimates in Proposition~\ref{mainprop}--\ref{decay_lem} hold globally, i.e., for all $s\ge s_0$. 

Let us define  the vector $\{V_i(s)\}_{1\leq i\leq 9}$ as follows:
		\begin{equation*}
		\begin{split}
			V_1 & :=|\dot{\tau}|, \quad
			 V_2  : = e^{3s/2} \| Z_y (\cdot, s) \|_{L^\infty}, \quad 
			  V_3 := \sup_{y\in\mathbb{R}} \left( \frac{1+y^2}{y^2} |W_y(y,s)-\overline{W}'(y)| \right) ,
			  \\
			  V_4 &:= \sup_{y\in\mathbb{R}} \left( \frac{(1+y^2)^{1/2}}{|y|} | W_{yy}(y,s) | \right), \quad
			 V_5  := | \partial_y^3 W(0,s) -6 |,  \quad V_6  := \| \partial_y^3 W (\cdot, s) \|_{L^\infty}, 
			 \\
			 V_7 &:= \| \partial_y^4 W (\cdot, s) \|_{L^\infty}, \quad    
			V_8 := \sup_{y\in\mathbb{R}} \left( (y^{2/3} + 8 ) | W_y(y,s) - \overline W'(y) | \right), 
			\\
			V_9  &:=  \sup_{y\in\mathbb{R}} \left( (1+ e^{-s} y^{2/3})\left|e^{\frac{e^{-s/2}W+\kappa-Z}{2\sqrt{K}}}-1\right| \right).
			\end{split}
		\end{equation*}
		Similarly,
		let $\{K_i\}_{1\leq i\leq 9}$ be such that 
		\begin{equation*}
			\begin{split}
					K_1:=2\veps, \quad 
				K_2:=1,\quad  K_3:= \frac{1}{10}, \ \ K_4:= 15, \ \ K_5:=1 , 
				\\
				\ \  K_6:= M^{5/6}, \ \ K_7 := M, \ \ K_8:= 1, \ \ K_9:= \frac{1}{14}. 
			\end{split}
		\end{equation*}
		For $\beta=(\beta_1, ...,\beta_9)\in(0,1)\times \cdots \times(0,1)$, we define the solution space $\mathfrak{X}_\beta(s)$ by 
		\begin{equation}\label{VK_boots}
			\mathfrak{X}_\beta(s) :=\{ (W,Z) \in C^4([s_0, s]\times \mathbb{R} ) : V_i(s')\leq \beta_i K_i, \; \forall s'\in[s_0,s], \; 1\leq i\leq 9  \}.
		\end{equation}
For simplicity,   let us  denote $\XX_\beta(s)$ with $\beta_i=1$ for $i=1, ..., 9$ by $\XX_1(s)$.  

   We see from  Remark~\ref{init_rmk} that
		\begin{equation}\label{VK_init}
			V_i(s_0)\leq \alpha_iK_i \quad \text{for some }  \alpha_i\in(0,1), \quad (1\leq i\leq 9),
		\end{equation}
	i.e., $(W,Z)\in\XX_\alpha(s_0)$	for some $ \alpha_i\in(0,1)$, $1\leq i\leq 9$.
		Proposition~\ref{mainprop} and Proposition~\ref{decay_lem} imply that 
if $(W,Z)\in\XX_1(\sigma_1)\cap \XX_\alpha(s_0)$, 
then $(W,Z)\in\XX_\beta(\sigma_1)$
		for some $\beta_i\in(\alpha_i,1)$, $1\leq i\leq 9$.


Now we will use a standard continuation argument together with local existence theory,   Lemma~\ref{criterion}, to show that the solution $(W, Z)$ exists globally and  $(W, Z)\in\XX_\beta(\infty)$. 
To see this, we define \begin{equation}\label{D-sig}
S:= \sup \{ s\ge s_0 : 
(W, Z) \in \mathfrak{X}_\beta(s)\}.
\end{equation} 
We claim that $S=\infty$. Suppose to the contrary that $S<\infty$. Then, there is a number $T=T(S)$ such that $S=-\log(\tau(T) - T)$ since $\dot{s}(t) = (1-\dot{\tau}(t))/(\tau(t)-t)>0$ for all $t\in[-\ve, T)$.  
Note that  \eqref{1} and \eqref{Utildey_M'} together with \eqref{y-w-y2} imply that 
\[ \sup_{s\in[s_0, S)} \| W_y(\cdot,s) \|_{L^\infty} \le 1. \]
This together with the fact that $w_x = e^{s} W_y$ implies $
 \|w_x(\cdot, t)\|_{L^\infty}\le e^{S}<\infty$ for all $t\le T$. Similarly by \eqref{7}, we have  $\|z_x\|_{L^\infty}\le 1$.
Thanks to Lemma~\ref{criterion}, 
 there is a unique extended solution $(w,z)\in C^4([-\ve, T+\delta))$ for some $\delta>0$, which immediately implies that 
$(W,Z)$ can be extended so that $(W,Z)\in \mathfrak{X}_1(S+\delta')$ for some $\delta'>0$. Then by Proposition~\ref{mainprop} and Proposition~\ref{decay_lem}, 
we see that $(W,Z)\in \mathfrak{X}_\beta(S+\delta')$. This contradicts to the definition of $S$, which concludes that  $S=\infty$.

\begin{remark}\label{T*_lem}
Using \eqref{dx-dy} and \eqref{tau}, we have 
	\[
	|\tau(t)|\leq \int^{t}_{-\veps}|\dot{\tau}(t')|\,dt' \leq 4\veps \int^{s}_{-\log\veps} e^{-s'}\,ds' \leq  4\veps^{2}
	\] 
for all $t>-\veps$ (or for all $s>s_0$).
\end{remark}

	Now we are ready to  prove Theorem~\ref{mainthm1}.

\begin{proof}[Proof of Theorem~\ref{mainthm1}]
 We split the proof into several steps.

Step 1 : We show that the solution $(w,z)$ exists up to $t=T_*$, i.e., $(w,z)\in C^4([-\ve, T_*) \times \mathbb{R})$, and that $|T_\ast|=O(\veps^2)$.   By a standard continuation argument, we have shown that $(W,Z)$ exists globally in $s$ and $(W,Z)\in \mathfrak{X}_1(\infty)$. This specifically implies that  the solution $(w,z)$
 exists on $[-\ve, T_*)$, i.e., $(w,z) \in C^4([-\ve, T_*) \times \mathbb{R})$,  where $T_\ast$ is the first time satisfying $\tau(T_\ast)=T_\ast$ (see \eqref{T-star} and \eqref{change_var}). Indeed, thanks to  \eqref{0}, which implies that $|\dot{\tau}(t)|<1/2$ for sufficiently small $\ve>0$, there exists a number $t_0<\infty$ such that $\tau(t_0)=t_0$ since $\tau(-\veps)=0$. This proves the existence of $T_\ast<\infty$. The fact that $|T_*| = O( \ve^2)$ directly follows from Remark~\ref{T*_lem}. Moreover, in Step 3, we will show that $\| w_x(\cdot, t) \|_{L^\infty} \nearrow \infty$ as $t\nearrow T_*$. This implies that $T_*$ is the life span of $C^1$ solution. 
 \\

Step 2: We show that $w$ and 
$z$ exhibit $C^{1/3}$ regularity at the blow-up time $T_\ast$.  
By a standard continuation argument, we have shown that  \eqref{Utildey_M'} holds globally in $s$, i.e., 
$$(y^{2/3}+8)|W_y(y,s)-\overline{W}'(y)| \le \frac{24}{25} \quad  \text{ for all } s\ge s_0, \ \  y\in \mathbb{R}.$$
From this and \eqref{y-w-y}, 
we have 
%
%
%
%
%
\begin{equation}\label{Wydec_fin_2}
|W_y (y,s) | < \frac{1}{y^{2/3}+8}+|\overline{W}'(y) |\leq \frac{C}{(1+y^2)^{1/3}}
\end{equation} 
 for some $C>0$.
Hence, we see that 
 for all $y \neq y'$, 
\begin{equation}\label{not-zero}
\begin{split} 
 {\frac{|W(y,s)-W(y',s)|}{{|y-y'|}^{1/3}}} & =  \frac{1}{{|y-y'|}^{1/3}} {\left|\int_{y'}^y W_y(\tilde y,s) d\tilde y \right|} 
 \\
 & \le   \frac{C}{{|y-y'|}^{1/3}}  \left| \int_{y'}^{y}{(1+\tilde y^2)^{-1/3}\,d\tilde y} \right| \lesssim 1.
 \end{split}
\end{equation} 

Consider any two points $x\neq x'\in \mathbb{R}$.  From \eqref{not-zero} and the change of variables,  \eqref{change_var} and \eqref{WZPhi}, we have
\begin{equation*}\label{Eq_holder}
\begin{split}
	\frac{|w(x,t)-w(x',t)|}{{|x-x'|}^{1/3}}
	& =\frac{|W(y,s)-W(y',s)|}{{|y-y'|}^{1/3}} \lesssim 1,
	\end{split}
\end{equation*} 
which proves the assertion (i), i.e.,   $w \in L^{\infty}([-\varepsilon,T_*];C^{1/3}(\mathbb{R}))$.  
\vspace{1\baselineskip}

Step 3: 
%
We claim that when $\beta > 1/3$, $C^\beta$ H\"older norm of $w$ blows up as $t\rightarrow T_*$. Again by  \eqref{change_var} and \eqref{WZPhi}, we note that 
\begin{equation}\label{main3_1}
	\frac{|w(x,t)-w(x',t)|}{|x-x'|^\beta}=e^{(\frac{3}{2}\beta-\frac{1}{2})s}\frac{|W(y,s)-W(y',s)|}{|y-y'|^{\beta}}. 
\end{equation} 
By the mean value theorem, we have
\begin{equation}\label{main3_2}
	\frac{|W(y,s)-W(0,s)|}{|y|^{\beta}}=|W_y(\overline{y},s)||y|^{1-\beta} 
\end{equation}
 for some $\overline{y}$ between $0$ and $y$. On the other hand, from \eqref{constraint} and \eqref{3}, it holds that
\begin{equation}\label{main3_3}
|W_y(y,s)| \geq 1/2
\end{equation}
for all $y$ sufficiently close to $0$ and all $s \geq s_0$. Now we let $y'=0$ and fix $y\ne 0$ sufficiently close to $0$. Combining \eqref{main3_1}--\eqref{main3_3}, we have  
\begin{equation}\label{LB-blow}
	[w(t,\cdot)]_{C^{\beta}}\geq \frac12 |y|^{1-\beta} e^{(\frac{3}{2}\beta-\frac{1}{2})s} \ge c_0 e^{(\frac{3}{2}\beta-\frac{1}{2})s}> 0.
\end{equation}
Since $\beta>1/3$, 
the right hand side becomes unbounded as $s\to \infty$, i.e., $t\to T_*$. 
On the other hand, 
using the fact 
\begin{equation*} 
|W(y,s)-W(y',s)| \le \int_{y'}^y |W_y (\tilde y,s) | d\tilde y \le C \int_{y'}^y  (1+\tilde y^2)^{(\beta-1)/2} d\tilde y \leq C |y-y'|^\beta
\end{equation*} 
where the second inequality holds thanks to  \eqref{Wydec_fin_2}, i.e., $|W_y (y,s) |  \leq C(1+y^2)^{-1/3} \le C (1+y^2)^{(\beta-1)/2}$ for any $\beta>1/3$, 
we have from \eqref{main3_1} that 
\begin{equation} \label{UB-blow}
\begin{split}
	\frac{|w(x,t)-w(x',t)|}{|x-x'|^\beta} & = e^{(\frac{3}{2}\beta-\frac{1}{2})s} \frac{|W(y,s)-W(y',s)|}{|y-y'|^{\beta}} \le C_0 e^{(\frac{3}{2}\beta-\frac{1}{2})s}
	\end{split} 
\end{equation} 
for some $C_0>0$. 
Note  from \eqref{tau} that 
\begin{equation}
\frac{1}{2}(T_\ast-t) \leq \tau(t) - t = \int_t^{T_\ast} (1-\dot{\tau})\,dt' \leq \frac{3}{2}(T_\ast-t).
\end{equation}
This together with \eqref{LB-blow} and \eqref{UB-blow} yields the desired blow-up rate as 
\[\left[ w(\cdot, t) \right]_{C^\beta} \sim (T_*-t)^{-\frac{3\beta-1}{2}}\] 
which together with $\left[ z(\cdot, t) \right]_{C^\beta} \lesssim 1$ for all $\beta>1/3$ and   \eqref{w-z}, in turn implies the assertion (ii) and  (iii). 
In particular, the blow up location of $w_x$ is  $x=\xi(T_*)$, and the blow up rate of $\| \partial_x w(\cdot, t) \|_{L^\infty}$ is $(t-T_*)^{-1}$.
\vspace{1\baselineskip}

Step 4: We shall show that the density is far from the vacuum state. From \eqref{iso_EP}, we  have 
	\[ (w-z)_t + ( w+z+ 2\sqrt{K} ) (w-z)_x + 4 \sqrt{K} z_x =0.\]
	By integrating the equation along the characteristic curve, we have 
	\[ \| ( w-z )(\cdot, t)\|_{L^\infty} \le \| w_0 - z_0 \|_{L^\infty} + 4 \sqrt{K} (\ve + T_\ast) \|z_x\|_{L^{\infty}}\le \frac{\sqrt{K}}{2} + O(\ve)\le \sqrt{K},\]
	where we have used \eqref{w-z-L-inf}, 
 the fact that $\sup_{t<T_\ast} \| z_x (\cdot, t) \|_{L^{\infty}}\le 1$, which follows from \eqref{Boot_3} for all $s\ge s_0$, and that $T_\ast=O(\ve)$. 
	This immediately yields the desired lower  bound for $\rho$ in the assertion (iv) since  
		$$\sup_{t< T_\ast} \|\rho(\cdot, t)-1\|_{L^{\infty}}= \sup_{t < T_\ast} \|e^{\frac{w-z}{2\sqrt{K}}}-1\|_{L^{\infty}}\leq \min\{e^{1/2}-1, 1-e^{-1/2} \} <2/5.$$
		Similarly, integrating \eqref{iso_w} and \eqref{iso_z} along the characteristic curves, respectively, and using \eqref{Phi_1_M} in Lemma~\ref{phi_lem}, we have 
		\[ \sup_{t< T_\ast}  \| (w, z)(\cdot, t) \|_{L^\infty} \le C\] for some $C>0$. This together with \eqref{w-z1} yields the desired upper bounds for  $\rho$ and $u$. We are done with   the proof of the assertion (iv). 
		\vspace{1\baselineskip}		
We finish the proof of Theorem~\ref{mainthm1}.
\end{proof}

\section{Closure of Bootstrap I}
\label{sec3-boot}
In this section, we first establish some preliminary estimates and prove Proposition~\ref{mainprop}.
\subsection{Preliminary estimates}		
In the following lemma, we show that  $|\rho(x,t)-1|<1$ as long as the bootstrap assumptions \eqref{tau} and \eqref{Boot_3} hold.  To see this,  
we let  $T_1$ be the number such that $\sigma_1 = -\log(\tau(T_1) - T_1)$ where  $\sigma_1>s_0$ is the time up to which the bootstrap assumptions hold. $T_1
$ is well-defined and $T_1<T_*$.

\begin{lemma}\label{rho-1_prop}
	Let $(\rho, u)$ be a smooth solution to \eqref{EP}   with  the initial data $(\rho_0,u_0)(x)$ satisfying \eqref{w-z-L-inf}.
 Assume also that  \eqref{tau} and \eqref{Boot_3} hold. 
Then it holds that for all $t\in[-\ve, T_1]$, 
	\begin{equation}\label{wz3M1}
		\|\rho(\cdot,t)-1\|_{L^{\infty}} < \frac{2}{5}.
	\end{equation}
\end{lemma}
\begin{proof}
	From \eqref{iso_EP}, we  have 
	\[ (w-z)_t + ( w+z+ 2\sqrt{K} ) (w-z)_x + 4 \sqrt{K} z_x =0.\]
	By integrating along the characteristic curve  over $[-\ve, T_1]$, one has 
	\[ \| w-z\|_{L^\infty} \le \| w_0 - z_0 \|_{L^\infty} + 4 \sqrt{K} (\ve + T_1) \|z_x\|_{L^{\infty}}.\]
We note that, from Remark~\ref{T*_lem}, $\ve+T_1<1/8$ for sufficiently small $\ve>0$ and that, by  \eqref{Boot_3}, $\| z_x (\cdot, t) \|_{L^\infty}\le 1$ for $t\le T_1$. Hence,  using \eqref{w-z-L-inf}, we have   
	$$\|w-z\|_{L^{\infty}}\leq \sqrt{K},$$
which in turn, implies 
	$$\|\rho-1\|_{L^{\infty}}=\|e^{\frac{w-z}{2\sqrt{K}}}-1\|_{L^{\infty}}\leq \min\{e^{1/2}-1, 1-e^{-1/2} \} <2/5$$
	for all $t\in [-\veps,T_1)$.
\end{proof}

For the sake of convenience, we state an inequality (Lemma A.2 in \cite{BCK}). For the proof, we refer to \cite{BCK} and the references therein, for instance \cite{LLS}.

	\begin{lemma}[\cite{BCK}]
\label{LemmaAppen}
Assume that $\rho-1 \in L^\infty(\mathbb{R}) \cap L^2(\mathbb{R})$ satisfying $\inf_{x \in \mathbb{R}}\rho>0$ and $\lim_{|x|\to  \infty } \rho = 1$.  Let $\phi$ be the solution to the Poisson equation \eqref{EP_3}. Then, it holds that   
\begin{equation}\label{PhiH1Bd2}
\int_{\mathbb{R}} |\phi_x|^2 + (\phi-1)e^\phi +1  \,dx \leq \frac{1}{\theta_0} \int_{\mathbb{R}} |\rho-1|^2\,dx,
\end{equation}
where $\theta_0= \frac{1-  \underline \rho }{-\log  \underline \rho}$ if $\underline \rho:=\inf_{x\in\mathbb{R}}\rho \in (0,1)$, and $\theta_0=1$  if $\underline{\rho} \geq 1$. 
 
\end{lemma}

Together with \eqref{H}, \eqref{PhiH1Bd2} implies that the energy $H(t)$ is controlled by the size of $(w,z)$: 
\begin{equation}\label{H-L2}
 H(t)  \le  C  \| (\rho-1,u) \|_{L^2}^2   \le C  \| (w,z) \|_{L^2}^2   , 
\end{equation}
where $C>0$ depends only on $\inf_{x\in\mathbb{R}}\rho$ and $\sup_{x\in\mathbb{R}}\rho$.

The following lemma is proved in \cite[p.~8-9]{BCK}. Nevertheless, we restate it and give a brief proof in terms of $w$ and $z$ variables. Thanks to Lemma~\ref{rho-1_prop}, we see that \eqref{wz3M1} holds true as long as the bootstrap condition \eqref{tau} and \eqref{Boot_3} hold.
\begin{lemma}
\label{phi_lem} 
 Let $(\rho, u)$ be a smooth solution to \eqref{EP}   with  the initial data $(\rho_0,u_0)(x)$ satisfying \eqref{w-z-L-inf}. Then, it holds that 
\begin{equation}\label{Phi_0_M-1}
			| \phi(x,t) | \leq C_1,
\end{equation}
 and as long as $  \| \rho(\cdot,t) -1\|_{L^\infty(\mathbb{R})} <1$,
\begin{equation}\label{Phi_1_M}
			|\phi_x(x,t)| \leq C_2 
	\end{equation}
	for some 
	  $C_1>0$ and $C_2>0$. Here, $C_1$ and $C_2$ only depend on $\| (\rho_0-1,u_0) \|_{L^2}$, $\inf_{x\in\mathbb{R}}\rho_0$, and   $\sup_{x\in\mathbb{R}}\rho_0 $. 
\end{lemma}
\begin{proof}
Using the energy conservation \eqref{H}, it is shown in \cite[p.~8]{BCK} that for the smooth solutions to \eqref{EP}, it holds that 
\begin{equation}\label{V-phi}
V_-^{-1}\left( H(-\veps) \right) \leq \phi(x,t) \leq V_+^{-1}\left( H(-\veps) \right),
\end{equation}
where $V_+^{-1}:[0,+\infty) \to [0,+\infty)$ and $V_-^{-1}:[0,+\infty) \to (-\infty,0]$ are the inverse functions of 
\begin{equation*}
V_+(z):= \displaystyle{ \int_0^z \sqrt{2U(\tau)}\,d\tau } \; \text{ for } z \geq 0,  \quad V_-(z):= \displaystyle{ \int_z^0 \sqrt{2U(\tau)}\,d\tau } \; \text{ for }  z \leq  0,
\end{equation*}
respectively, and $U(\tau):=(\tau-1)e^\tau + 1$ is nonnegative for all $\tau\in \mathbb{R}$. Combining \eqref{H-L2} and   \eqref{V-phi}, we have 
 \begin{equation*}\label{phi-c1} 
 |   \phi(x,t) | \le \max\{ V_+^{-1} (C \| (\rho_0-1,u_0)\|^2_{L^2(\mathbb{R})}) , V_-^{-1} (C \| (\rho_0-1,u_0)\|^2_{L^2(\mathbb{R})})\}=:C_1.
 \end{equation*}
 This proves \eqref{Phi_0_M-1}. 

Multiplying the Poisson equation \eqref{EP_3} by $-\phi_x$, and then integrating in $x$, we have 
\begin{equation}
\begin{split}\label{phi-x-2}
\frac{\phi_x^2}{2} 
& = \int_{-\infty}^x(\rho-1)(-\phi_x)\,dx + \int_{-\infty}^x(e^\phi-1)\phi_x\,dx \\
& \leq \frac{1}{2}  \int_{\mathbb{R}} | \rho-1 |^2 dx + \frac{1}{2} \int_{\mathbb{R}} | \phi_x|^2 dx + e^\phi - \phi -1. \\ 
%
%
%
%
%
%
\end{split}
\end{equation}
Here, we notice that as long as $|\rho(x,t)-1|<1$ for all $(x,t) \in \mathbb{R} \times [0,T]$, it holds that 
\begin{equation*}\label{phi-x-3}
\int_{-\infty}^\infty \frac{1}{4}|\rho-1|^2\,dx \leq \int_{-\infty}^\infty \rho \ln \rho - \rho +1 \,dx   \leq  \frac{H(t)}{K}. 
\end{equation*}
By the Taylor expansion using \eqref{Phi_0_M-1},
\[ 0\le e^\phi - 1 - \phi \lesssim   C_1^2.\]
Hence, by the energy conservation \eqref{H} and \eqref{H-L2}, we obtain from \eqref{phi-x-2} that 
\[ |\phi_x | \le C (  \| (\rho_0-1, u_0 ) \|_{L^2(\mathbb{R})}, \inf_{x\in\mathbb{R}}\rho_0, \sup_{x\in\mathbb{R}}\rho_0)  =: C_2.\] We finish the proof.

\end{proof}


Now using the uniform bounds for $\phi$ and $\phi_x$ in Lemma~\ref{phi_lem}, 
we obtain the uniform bounds for $\partial_y^i \Phi$, $i=0,1,2,3,4$.  
\begin{lemma}\label{phi_y1y2_lem} Let the same assumptions as in Lemma~\ref{rho-1_prop} hold. Then, there is a constant $C>0$ such that for all $s \in [s_0, \sigma_1]$, 
		\begin{equation}\label{Phiy2}
\|\Phi\|_{L^{\infty}} + e^{3s/2}\|\Phi_y\|_{L^{\infty}} + e^{3s}\|\Phi_{yy}\|_{L^{\infty}} \leq C.
		\end{equation} 
If we further assume that \eqref{Boot_3} and \eqref{Boot_2} hold, then for all $s \geq s_0$, we have 
\begin{equation}\label{Phiy34}
e^{7s/2} \|\partial_y^3\Phi\|_{L^{\infty}}   \leq C.
\end{equation}
In \eqref{Phiy2} and \eqref{Phiy34}, the constant $C >0$ depends on $\|(w_0,z_0)\|_{L^2}$, $P_{\pm}$ where $P_{\pm}$ is defined in \eqref{P-+}.			
		\end{lemma}
		\begin{proof}
  Note that $\phi=\Phi$,  $\phi_x=\Phi_y e^{3s/2}$ and $\phi_{xx}=\Phi_{yy }e^{3s}$ from \eqref{WZPhi} and \eqref{dx-dy}.
 We also see from \eqref{EP_3}, \eqref{wz3M1} and \eqref{Phi_0_M-1} that \[ | \phi_{xx} (x,t) | = | \rho(x,t) - e^{\phi(x,t)}| \le | \rho (x,t) | + | e^{\phi(x,t)}   | \le C \]
for some $C>0$. 

Then by this together with  Lemma~\ref{phi_lem}, we obtain \eqref{Phiy2}. 
To show \eqref{Phiy34}, we differentiate  \eqref{EP2_3} with respect to  $y$ to obtain 
\begin{equation*}
-\partial_y^3\Phi e^{3s}  =\frac{1}{2\sqrt{K}}(e^{-s/2}W_y-Z_y)e^{\frac{e^{-s/2}W+\kappa-Z}{2\sqrt{K}}}-\Phi_ye^{\Phi}.
\end{equation*}
Recalling the relation \eqref{w-z}, yielding  $w-z = 2\sqrt{K} \log \rho$, and  using  \eqref{WZPhi} and  \eqref{wz3M1}, 
we have 
\begin{equation}\label{Phi3y-e3s} 
\left| e^{\frac{e^{-s/2}W+\kappa-Z}{2\sqrt{K}}} \right| = | \rho | \le \frac75. 
\end{equation}
Using this, \eqref{Boot_3}, \eqref{Uy1} and \eqref{Phiy2}, we have
\begin{equation*}
\begin{split}
					|\partial_y^3\Phi e^{3s}| 
					& \leq \frac{7}{10 \sqrt{K}}(e^{-s/2}|W_y|+|Z_y|) + Ce^{-3s/2} \\
					& \leq \frac{7}{10\sqrt{K}}(e^{-s/2}+e^{-3s/2}) + Ce^{-3s/2} \le C e^{-s/2}.
\end{split}
\end{equation*}
This gives the desired bound for $\partial_y^3\Phi$ in \eqref{Phiy34}. 
 We are done. 
		\end{proof}
		\begin{lemma}\label{difZn}
			If the same  assumptions as in Proposition~\ref{mainprop} hold, then there exists a constant $C>0$ satisfying 
			\begin{equation}\label{Zyy_56}
			\| Z_{yy}\|_{L^{\infty}} \leq e^{-5s/6},
			\end{equation}
			\begin{equation}\label{Phi4y}
\|\partial_y^4\Phi\|_{L^{\infty}} \leq Ce^{-7s/2}
			\end{equation}
   
   \begin{equation}\label{difZn_1}
				\|\partial_y^3 Z\|_{L^{\infty}} \leq C e^{-5s/6},  
			\end{equation}
			\begin{equation}\label{Phiy5}
				\|\partial_y^5\Phi \|_{L^{\infty}}\leq C e^{-7s/2},
			\end{equation}
			\begin{equation}\label{difZn_2}
			\|\partial_y^4 Z\|_{L^{\infty}} \leq C e^{-5s/6}
		\end{equation}
			for all $s\in[s_0,\sigma_1]$.
		\end{lemma}
		
		\begin{proof} We first prove \eqref{Zyy_56}. 
Recalling \eqref{EP2_3D2}, $Z_{yy}$ satisfies
\begin{equation}\label{Zyyeq}
	\partial_s Z_{yy}+D^Z_2 Z_{yy}+U^Z \partial_y^3Z = F^Z_2,
\end{equation}
where $U^Z$ is defined in \eqref{UW}, and
\begin{subequations}
	\begin{align*}
		D^Z_2(y,s) &:=3+\frac{2W_y}{1-\dot{\tau}}+\frac{3e^{s/2}Z_y}{1-\dot{\tau}},\\
		F^Z_2(y,s) &:= -\frac{2e^{s/2}\partial_y^3\Phi}{1-\dot{\tau}}-\frac{W_{yy}Z_y}{1-\dot{\tau}}.
	\end{align*}
\end{subequations}
Thanks to \eqref{tau}, \eqref{Boot_3} and \eqref{Uy1}, we see that
\begin{equation}\label{DZ2}
\begin{split}
	D^Z_{2}
	& \geq 3-(1+4\veps)(2+3e^{-s})  \geq \frac{5}{6}
\end{split}
\end{equation}
for all sufficiently small $\veps>0$.  Also, by \eqref{Boot_3}, \eqref{EP2_1D3} and \eqref{Phiy34}, we have
\begin{equation}\label{FZ2}
	|F^Z_2| \leq C(e^{-3s}+e^{-3s/2})\leq Ce^{-3s/2}.
\end{equation} 
Let  $\psi$ be the characteristic curve satisfying $\partial_s \psi=U^Z$ with $\psi(y,s_0)=y$. 
Integrating $Z_{yy}$ along $\psi$ gives
\begin{equation*}
	Z_{yy}(\psi(y,s),s)=Z_{yy}(y,s_0)e^{-\int^s_{s_0}(D^Z_2\circ\psi)\,ds'}+\int^s_{s_0}e^{-\int^s_{s'}(D^Z_2\circ\psi)\,ds''}(F^Z_2\circ\psi)\,ds'. 
\end{equation*}
We see, from $Z_{yy}(y,s_0) = e^{-3s_0} z_{xx}(x,-\ve)$ and \eqref{init_24} that
$|Z_{yy}(y,s_0)| \le e^{-3s_0}/4= \ve^3/4$. 
Using this, \eqref{DZ2} and \eqref{FZ2}, we get
\begin{equation}\label{Zyy-pf}
\begin{split}
	\|Z_{yy}\|_{L^{\infty}}
	& \leq \|Z_{yy}(y,s_0)\|_{L^{\infty}}e^{-5(s-s_0)/6} + C \int_{s_0}^s e^{-5(s-s')/6} e^{-3s'/2}\,ds' \\
	& \leq \frac{\veps^{13/6}}{4}e^{-5s/6}+ C\veps^{3/2}e^{-5s/6} \\
	& \leq e^{-5s/6}
\end{split}
\end{equation} 
for all sufficiently small $\veps>0$.
Here we have used the fact that $\psi(\cdot, s):\mathbb{R} \to \mathbb{R}$ is invertible. 
This completes the proof of \eqref{Zyy_56}.

Now we show \eqref{Phi4y}. 
Differentiating \eqref{EP2_3} with respect to $y$ twice, we have
\begin{equation*}
-\partial_y^4\Phi e^{3s} =\left(\frac{1}{2\sqrt{K}}(e^{-s/2}W_{yy}-Z_{yy})+\frac{1}{4K}(e^{-s}W_y-Z_y)^2\right)e^{\frac{e^{-s/2}W+\kappa-Z}{2\sqrt{K}}}-(\Phi_{yy}+\Phi_y^2)e^{\Phi}. 
\end{equation*}
We use \eqref{Boot_3}, \eqref{EP2_1D3}, \eqref{Uy1},  \eqref{Phiy2}, \eqref{Phi3y-e3s} and \eqref{Zyy-pf}, to  obtain the desired bound for $\partial_y^4\Phi$ in  \eqref{Phi4y}.

		    Next, we  prove \eqref{difZn_1}.  
			Taking $\partial_y$ of  \eqref{EP2_3D2}, we have the equation for  $\partial^3_yZ$,
			\begin{equation}\label{Zy3eq}
				\partial_s \partial_y^3Z +D^Z_3 \partial_y^3Z +U^Z \partial_y^4Z =F^Z_3,
			\end{equation}
			where $U^Z$ is defined in \eqref{UW}, and
			\begin{subequations}
				\begin{align*}
					D^Z_3(y,s) &:= \frac{9}{2}+\frac{3W_y}{1-\dot{\tau}}+\frac{4e^{s/2}Z_y}{1-\dot{\tau}}, 
					\\
					F^Z_3(y,s) &:= -\frac{2e^{s/2}\partial_y^4\Phi}{1-\dot{\tau}}-\frac{3e^{s/2}Z_{yy}^2}{1-\dot{\tau}}-\frac{1}{1-\dot{\tau}}(\partial_y^3WZ_y+3Z_{yy}W_{yy}).
				\end{align*}
			\end{subequations}
		   Similarly as \eqref{DZ2}, we have
			\begin{equation}\label{DZ3}
				D^Z_3 \geq \frac{9}{2}-(3 + O(\veps))\geq 1
			\end{equation}
			for sufficiently small $\veps>0$.  Also, we see that
			\begin{equation}\label{FZ3}
				\begin{split}
					|F^Z_3| & \leq 2 \left|2e^{s/2}\partial_y^4\Phi+3e^{s/2}Z_{yy}^2+(\partial_y^3WZ_y+3Z_{yy}W_{yy})\right|
					\\
					&\leq 2\left(Ce^{-3s }+3e^{-7s/6}+M_3e^{-3s/2}+45e^{-5s/6}\right)
					\\
					&\leq Ce^{-5s/6},
				\end{split}
			\end{equation}
			where we have used \eqref{tau} in the first inequality, and used \eqref{Boot_3}, \eqref{Boot_2},  \eqref{Phiy34}, \eqref{Zyy_56}, \eqref{Phi4y}, 
       in  the second inequality. Upon  integrating \eqref{Zy3eq} along $\psi$ as above, we see that  \eqref{DZ3} and \eqref{FZ3} yield  
			\begin{equation*}
				\begin{split}
					\|\partial_y^3Z\|_{L^{\infty}}
					&\leq \|\partial_y^3Z_0\|_{L^{\infty}}e^{-(s-s_0)}+ C\int_{s_0}^s e^{-(s-s')} e^{-5s'/6}\,ds' \\
					& \leq \frac{1}{4}\veps^{7/2}e^{-s}+ Ce^{-5s/6} \\
					& \leq C e^{-5s/6},
				\end{split}
			\end{equation*}
			where we have used \eqref{init_24}, i.e.,  $\veps^{-9/2}\|\partial_y^3Z(\cdot,s_0) \|_{L^{\infty}} = \|\partial_x^3z(\cdot,-\veps)\|_{L^{\infty}}\leq 1/4$. Thus, \eqref{difZn_1} holds for sufficiently small $\veps>0$.

Next we shall prove \eqref{Phiy5}. By taking $\partial_y^3$ of \eqref{EP2_3}, 
we have 
	\begin{equation*}
	\begin{split}
	-\partial_y^5\Phi e^{3s} 
	& =\frac{1}{2\sqrt{K}}(e^{-s/2}\partial_y^3W-\partial_y^3Z)\rho  	+\frac{3}{4K}(e^{-s}W_y-Z_y)(e^{-s}W_{yy}-Z_{yy})\rho  \\
	& \quad  +\frac{1}{8K\sqrt{K}}(e^{-s}W_y-Z_y)^3\rho -(\partial_y^3\Phi+3\Phi_y\Phi_{yy}+\Phi_y^3)e^{\Phi}.
	\end{split}
	\end{equation*}
Note that $\rho = 	e^{\frac{e^{-s/2}W+\kappa-Z}{2\sqrt{K}}} \leq 2	$ by \eqref{wz3M1}.  			
			We estimate each term. By \eqref{EP2_1D5} and \eqref{difZn_1}, we have
			\begin{equation*}
				|e^{-s/2}\partial_y^3W-\partial_y^3Z| \leq (M_3e^{-s/2}+Ce^{-5s/6}).
			\end{equation*} 
By  \eqref{Boot_3},  \eqref{EP2_1D3} and \eqref{Uy1}, we get 	
			\begin{multline*}
				|e^{-s}W_y-Z_y| |e^{-s}W_{yy}-Z_{yy}|
				  \leq 
				 (e^{-s}+e^{-3s/2})(15e^{-s}+e^{-5s/6}),
			\end{multline*} 
			\begin{equation*}
				|e^{-s}W_y-Z_y|^3 \leq (e^{-s}+e^{-3s/2})^3.
			\end{equation*}
By Lemma~\ref{phi_y1y2_lem}, we have 
\begin{equation*}
				|\partial_y^3\Phi+3\Phi_y\Phi_{yy}+\Phi_y^3|e^{\Phi}\leq C(e^{-7s/2}+3e^{-9s/2}+e^{-9s/2})e^{C}.
			\end{equation*}
		Combining the above estimates, we obtain \eqref{Phiy5}.

Next we show  \eqref{difZn_2}.  Taking $\partial_y^2$ of \eqref{EP2_3D2}, we have 
			\begin{equation}\label{Zy4eq}
				\partial_s \partial_y^4Z +D^Z_4 \partial_y^4Z +U^Z\partial_y^5Z = F^Z_4
			\end{equation}
			where $U^Z$ is defined in \eqref{UW}, and
			\begin{subequations}
				\begin{align*}
					D^Z_4 &:=6 +\frac{4W_y}{1-\dot{\tau}}+\frac{5e^{s/2}Z_y}{1-\dot{\tau}},
					\\
					F^Z_4 &:=-\frac{2e^{s/2}\partial_y^5\Phi}{1-\dot{\tau}}-\frac{10e^{s/2}Z_{yy}\partial_y^3Z}{1-\dot{\tau}}-\frac{1}{1-\dot{\tau}}(\partial_y^4WZ_y+4\partial_y^3WZ_{yy}+6W_{yy}\partial_y^3Z).
				\end{align*}
			\end{subequations}
In the same way as \eqref{DZ2}, we see that
			\begin{equation}\label{DZ4}
				D^Z_4 \geq 6-(4 + O(\veps))\geq 1
			\end{equation}
as $\veps \to 0$.  By \eqref{tau}, \eqref{Boot_3}, \eqref{Boot_2},   \eqref{Zyy_56}, \eqref{difZn_1} and \eqref{Phiy5}, we obtain for a constant $C>0$ that
		\begin{equation}\label{FZ4}
			\begin{split}
				|F^Z_4| & \leq 2 \left| 2e^{s/2}\partial_y^5\Phi+10e^{s/2}Z_{yy}\partial_y^3Z+\partial_y^4WZ_y+4\partial_y^3WZ_{yy}+6W_{yy}\partial_y^3Z \right|   \\ 
	&\leq C \left(e^{-3s}+ e^{-7s/6} +e^{-3s/2}+e^{-5s/6}+ e^{-5s/6}\right)
				\\
				&\leq C e^{-5s/6}.
			\end{split}
		\end{equation}
After integrating \eqref{Zy4eq} as above, it is straightforward to see that \eqref{difZn_2} follows by combining \eqref{init_24}, \eqref{DZ4} and \eqref{FZ4}.
			\end{proof}

\begin{lemma}
	For each $\veps\in(0,\veps_0)$, if the same  assumptions as in Proposition~\ref{mainprop} hold for all $s\in[s_0, \sigma_1]$, then it holds that 
\begin{equation}\label{temp_2}
		 |\sigma(y,s)| \leq 2|y|e^{-s} + 2e^{-s/3},
			\end{equation}
			\begin{equation}|\kappa|\leq |\kappa_0|+1 , \quad |\xi|\leq C\veps, \label{xibound}
			\end{equation} 
	\begin{equation}\label{UW_far}
		\inf_{\{|y|\geq 1, s\in[s_0,\sigma_1]\}} U^W(y,s) \frac{y}{|y|}  \geq \frac{1}{8}
	\end{equation}
	 for some $C>0$, where $U^W$ and $\sigma$ are defined in \eqref{UW}.
\end{lemma}		

\begin{proof}
Using \eqref{EP2_1D2}, \eqref{Phiy34} and \eqref{Zyy_56}, we first observe that
\begin{equation}\label{ax2} 
\left|\frac{-2e^{s/2}\partial_y^3\Phi(0,s)+Z_{yy}(0,s)}{\partial_y^3W(0,s)}\right| \leq e^{-5s/6}.
\end{equation}
Applying \eqref{ax2} to \eqref{mod3}, we have
			\begin{equation}\label{p1} 
			\begin{split}
				|\kappa-\dot{\xi} + Z(0,s)+2\sqrt{K}| \leq e^{-5s/6}.
				\end{split}
			\end{equation}
On the other hand, using the fundamental theorem of calculus,  \eqref{Boot_3} implies that
\begin{equation}\label{p2}
|Z(y,s)-Z(0,s)| \leq |y|e^{-3s/2}.
\end{equation}		
Using \eqref{tau} and combining \eqref{p1} and \eqref{p2}, we obtain \eqref{temp_2}:
\[
|\sigma(y,s)| \leq 2e^{s/2} ( |\kappa-\dot{\xi} + Z(0,s)+2\sqrt{K}| + |Z(y,s)-Z(0,s)| ) \leq 2 e^{-s/3} + 2|y|e^{-s}.
\]

To prove \eqref{xibound}, we observe that from \eqref{mod2}, 
			\begin{equation*}
			\begin{split}
				|\dot{\kappa}|
				&  \leq  e^s\left|\frac{-2e^{s/2}\partial_y^3\Phi(0,s)+Z_{yy}(0,s)}{\partial_y^3W(0,s)}\right| + |2\Phi_y(0,s)e^{3s/2}| 
				 \leq e^{s/6} + C,
			\end{split}
			\end{equation*}
			where we have used \eqref{Phiy2} and \eqref{ax2}.  The bound for $|\dot{\kappa}|$ yields that for sufficiently small $\veps>0$, 
			\begin{equation}\label{kap_bound}
				\begin{split}
					|\kappa| 
					& \leq |w_0(0)|+\int^t_{-\veps}|\dot{\kappa}|dt'    = |\kappa_0|+\int^s_{s_0} |\dot{\kappa}| \frac{e^{-s'}}{1-\dot{\tau}} ds'   
					\\
					&\leq |\kappa_0|+2\int^s_{s_0} (e^{s'/6}+C) e^{-s'} ds'   
					\\
					&\leq |\kappa_0|+1.
				\end{split}
			\end{equation}  
	For the bound of $\xi$, we take integration along the characteristic curves to \eqref{iso_z} and use \eqref{init_24}, \eqref{Phi_1_M} and Remark~\ref{T*_lem} to obtain 
	\begin{equation*}
		\begin{split}
			\|z(\cdot,t)\|_{L^{\infty}}&\leq \|z_0\|_{L^{\infty}}+2\int^t_{t_0}\|\phi_x(\cdot,t')\|_{L^{\infty}}\,dt'
			\\
			&\leq \|z_0\|_{L^{\infty}}+2C_2(T_*+\veps)\leq 1.
		\end{split}
	\end{equation*}
	Applying this and  \eqref{ax2}, \eqref{kap_bound} 
	to \eqref{mod3}, we have that for all sufficiently small $\veps>0$,
	\begin{equation*}
		\begin{split}
			|\dot{\xi}|
			& \leq  |Z(0,s) | + 2\sqrt{K} + \left|\frac{2e^{s/2}\partial_y^3\Phi(0,s)-Z_{yy}(0,s)}{\partial_y^3W(0,s)} \right| +|\kappa |  
			\leq C(K,\kappa_0),
		\end{split}
	\end{equation*} 
	from which we obtain the bound for $|\xi|$ in \eqref{xibound} in the same manner as \eqref{kap_bound}. 
	This finishes the proof of \eqref{xibound}.

To show \eqref{UW_far}, using \eqref{tau}, \eqref{Wbd1} and \eqref{temp_2}, one can check that   
	\begin{equation*}
		\begin{split}
			U^W&\geq -\frac{y}{1-\dot{\tau}}+\frac{3y}{2}-(2ye^{-s}+e^{-s/3})
			\ge  \frac{y}{4}-\veps^{1/3}\quad \text{ for }y\geq 0,
		\end{split}
	\end{equation*}
	\begin{equation*}
		\begin{split}
			U^W &\leq -\frac{y}{1-\dot{\tau}}+\frac{3y}{2}+(-2ye^{-s}+e^{-s/3})
			\leq \frac{y}{4}+\veps^{1/3}\quad \text{ for }y<0, \\
		\end{split}
	\end{equation*} 
from which \eqref{UW_far} directly follows. This finishes the proof of Lemma.
\end{proof}	 
Now we are ready to prove Proposition~\ref{mainprop}.  
\subsection{Proof of Proposition~\ref{mainprop} }\label{ch3.4}
In this section, we prove each assertion in Proposition~\ref{mainprop}. 
\subsubsection{The uniform bound of $\dot{\tau}$}
First, we close \eqref{tau}.
	\begin{lemma}\label{lem-3.7}
		If the same assumptions as in Proposition~\ref{mainprop} hold, then we have 
		\begin{equation}\label{dottau} 
			|\dot{\tau}| \le \frac{3}{2}e^{-s}
		\end{equation}
		for $s\in [s_0,\sigma_1]$.	
	\end{lemma}
	\begin{proof}
		From \eqref{mod1}, together with  \eqref{Boot_3} and \eqref{Phiy2}, 
		we have
		\begin{equation*}
			|\dot{\tau}|\leq |e^{s/2}Z_y(0,s)|+2e^s|\Phi_{yy}(0,s)|\leq  e^{-s}+2Ce^{-2s} \le (1+ 2C\ve) e^{-s} \leq \frac{3}{2}e^{-s} 
		\end{equation*}
  for sufficiently small $\ve>0$. 
		This proves \eqref{dottau}.
	\end{proof}

\subsubsection{The first derivative of $Z$}  
We close the bootstrap assumption \eqref{Boot_3}.
	\begin{lemma}\label{Zder_lem}
		If the same assumptions as in Proposition~\ref{mainprop} hold, then we have 
			\begin{equation*} 
			\|Z_y (\cdot, s) \|_{L^{\infty}} \le \frac{7}{8}e^{-3s/2}
			\end{equation*}
			for $s\in [s_0,\sigma_1]$.			 
		\end{lemma}

		\begin{proof} 
		Setting $\wt{Z}=e^{3s/2}Z$, we have from \eqref{EP2_3D1} that
		\begin{equation}\label{tildeZy_eq}
			\partial_s\wt{Z}_y+U^Z\wt{Z}_{yy} + D_{\tilde Z}\wt{Z}_y =F_{\tilde Z},
		\end{equation}
 where $U^Z$ is defined in \eqref{UW} and
 \begin{subequations}
 	\begin{align*}
 		&D_{\tilde Z}(y,s):=\frac{W_y}{1-\dot{\tau}},
 		\qquad F_{\tilde Z}(y,s):= -\frac{2e^{2s}\Phi_{yy}}{1-\dot{\tau}}-\frac{e^{-s}\wt{Z}_y^2}{1-\dot{\tau}}.
 	\end{align*}
 \end{subequations}
We let $\psi=\psi(y,s)$  be the solution to $\partial_s\psi=U^Z(\psi,s)$ with $\psi(y,s_0)=y$. By integrating \eqref{tildeZy_eq} along $\psi$, we obtain 
		\begin{equation}\label{Zt}
			\wt{Z}_y(\psi(y,s),s)=\wt{Z}_y(y,s_0)e^{-\int^s_{s_0} ( D_{\tilde Z}\circ \psi )(s') \,ds'}+ \int^s_{s_0}e^{-\int^s_{s'} ( D_{\tilde Z}\circ \psi )(s'') \,ds''}(F_{\tilde Z}\circ \psi)(s')\,ds'.
		\end{equation} 
		Here, we have from \eqref{Boot_3} and \eqref{Phiy2}  that
		\begin{equation}\label{FZt}
			\|F_{\Zt}\|_{L^{\infty}}\leq Ce^{-s }.
		\end{equation}

On the other hand, by integrating \eqref{Peq} along $\psi$, we get
		\begin{equation}\label{*}
			P(\psi(y,s),s)-P(y,s_0)=-\int^s_{s_0}\frac{2 }{1-\dot{\tau}}W_y(\psi(y,s'),s')\,ds'.
		\end{equation}
Thanks to \eqref{wz3M1}, $P=\log\rho$ satisfies $|P|\leq 3/5$, so the left hand side terms satisfy
$$\sup_{s\in [s_0, \sigma_1] } \|P(\psi(y,s),s)-P(y,s_0)\|_{L^{\infty}_y}
\leq \frac65.$$
 Using this for \eqref{*}, we get
		\begin{equation*}
			\left|\int^s_{s_0}\frac{W_y(\psi(y,s'),s')}{1-\dot{\tau}}\,ds' \right|\leq \frac35 
		\end{equation*}
		for all $y\in \mathbb{R}$ and for all $s\in[ s_0, \sigma_1]$. 
		From this inequality, we obtain 
		\begin{equation}\label{DZt_s}
			-\int^s_{s_0} ( D_{\Zt}\circ \psi )(s') \,ds' \leq 
   \left|\int^s_{s_0}\frac{W_y(\psi(y,s'),s')}{1-\dot{\tau}}\,ds' \right|\leq \frac35 ,
		\end{equation}
		for all $s\geq s_0$.
		We can also simply obtain from this inequality that
		\begin{equation}\label{DZt_s'}
			-\int^s_{s'}D_{\Zt}\circ\psi \,ds' = -\int^s_{s_0}D_{\Zt}\circ\psi \,ds' + \int^{s'}_{s_0} D_{\Zt}\circ \psi\,ds' \leq \frac65 ,
		\end{equation}
		for all $s,s'\geq s_0$.
		Using \eqref{FZt}, \eqref{DZt_s} and \eqref{DZt_s'} for \eqref{Zt}, we get
		\begin{equation*}
			\begin{split}
				|\wt{Z}_y(\psi(s),s)|&\leq |\wt{Z}_y(y,s_0)|e^{3/5}+\int^s_{s_0} Ce^{6/5}e^{-s'}\,ds'
				\\
				&\leq \frac{1}{4}e^{3/5}+C\veps<\frac{7}{8}.
			\end{split}
		\end{equation*}
		Here we used \eqref{init_24} for the second inequality, and the last inequality holds by sufficiently small $\veps$.
		This completes the proof of Lemma~\ref{Zder_lem}.
\end{proof}

\subsubsection{The third derivative of $W$ at $y=0$}
In Lemma~\ref{W30_lem}, we close the bootstrap assumption \eqref{EP2_1D2}.
\begin{lemma}\label{W30_lem}
	For each $\veps\in(0,\veps_0)$, if the same assumptions as in Proposition~\ref{mainprop} hold, then we have 
	\begin{equation}\label{str_U3}
		 |\partial_y^3W(0,s)-6|\leq C\veps^{1/3}, 
	\end{equation}
	  for some $C>0$ and $s \in [s_0,\sigma_1]$.
\end{lemma}

\begin{proof}
Evaluating \eqref{EP2_2D3} at $y=0$ and using  \eqref{constraint}, we have 
\begin{equation}\label{Eq_3rd1}
\begin{split} 
	 \partial_s \partial_y^3W^0 & = - \left( 4 - \frac{4}{1-\dot{\tau}}  +\frac{3e^{s/2}Z_y^0}{1-\dot{\tau}} \right) \partial_y^3W^0 
  \\
  & \qquad \qquad \qquad \qquad \qquad 
  -\sigma(0,s)\partial_y^4W^0 
  -\frac{2e^s\partial_y^4\Phi^0}{1-\dot{\tau}}+\frac{e^{s/2}\partial_y^3Z^0}{1-\dot{\tau}},
  \end{split}
\end{equation}
where $\partial_y^k f^0 := \partial_y^k f (0,s)$ for $f=W,Z$ or $\Phi$. 
Using \eqref{Boot_3}, \eqref{Boot_2}, \eqref{temp_2} and \eqref{dottau},  the first two terms of the right-hand side of \eqref{Eq_3rd1} are bounded as follows:
\begin{equation*}
\left|\left( 4 - \frac{4}{1-\dot{\tau}}  +\frac{3e^{s/2}Z_y^0}{1-\dot{\tau}} \right) \partial_y^3W^0    + \sigma(0,s)\partial_y^4W^0 \right| \leq C\left( \frac{|\dot{\tau}|}{|1-\dot{\tau}|}   + e^{-s} \right) + M e^{-s/3} \leq C e^{-s/3}.
\end{equation*} 
For the  remaining terms, using \eqref{tau}, \eqref{Phi4y} and \eqref{difZn_1}, we have 
\begin{equation*}
\left|\frac{2e^s\partial_y^4\Phi^0}{1-\dot{\tau}}- \frac{e^{s/2}\partial_y^3Z^0}{1-\dot{\tau}} \right|  \leq C e^{-s/3}.
\end{equation*}
Combining the above estimates, we obtain that 
\begin{equation*} 
	|\partial_y^3W^0(s)-\partial_y^3W^0(s_0)|\leq \int_{s_0}^{s}{|\partial_{s'}\partial_y^3W^0(s')|\,ds'}
 \leq C \int_{s_0}^{s}e^{-s'/3}\,ds'
 \leq C\veps^{1/3}.
\end{equation*}
Noting that $\partial_y^3W^0(s_0) = 6$ (see \eqref{1D2}), we conclude the proof.
\end{proof}

\subsubsection{The first derivative of $W$ near $y=0$}
We close the bootstrap assumption \eqref{Boot_1}. Define  
\begin{equation} \label{wtW}
\wt{W}(y,s) := W(y,s) - \overline{W}(y),
\end{equation} 
where $\overline{W}$ is the blow-up profile of Burgers' equation, which solves \eqref{Burgers_SS}.
Differentiating \eqref{Burgers_SS} in $y$ and using \eqref{EP2_2D1}, we have 
\begin{multline}
 \label{Eq_diff}
		  \left( \partial_s + 1 + \frac{\wt{W}_y + 2\overline{W}'}{1-\dot{\tau}} + \frac{e^{s/2}Z_y}{1-\dot{\tau}} \right) \wt{W}_y +U^W\wt{W}_{yy}   
		\\
  =
		  - \left( \frac{\wt {W} + \dot{\tau}\overline{W}}{1-\dot{\tau}}  + \sigma \right) \overline{W}'' - \left(\frac{\dot{\tau}\overline{W}'}{1-\dot{\tau}}+\frac{e^{s/2}Z_y}{1-\dot{\tau}}\right)\overline{W}' - \frac{2e^s}{1-\dot{\tau}}\Phi_{yy},
\end{multline}
where $U^W$ is defined in \eqref{UW}. Note that the damping term in \eqref{Eq_diff}, 
\[ D^{\widetilde{W}} := 1 + \frac{\wt{W}_y + 2\overline{W}'}{1-\dot{\tau}} + \frac{e^{s/2}Z_y}{1-\dot{\tau}} \] 
is \emph{degenerate} around $y=0$. More precisely, it is strictly negative around $y=0$. To capture the precise behavior of  $\wt {W}_y$ around $y=0$, we introduce a weight function and set 
\begin{equation}\label{Def-V}
V(y,s):=\frac{1+y^2}{y^2}\wt{W}_y(y,s).
\end{equation}
The order of the weight is chosen in view of the constraint \eqref{constraint} at $y=0$, and it does not affect the damping near $y=\infty$.

In terms of $V(y,s)$, one can rewrite \eqref{Eq_diff} as 
	\begin{equation}
	 \label{Eq_V}
		 \partial_s V + D^V  V +U^WV_y  =
		F^V + \int_{\mathbb{R}} K^V(y,s;y') V(y',s) dy',
	\end{equation}
where
\begin{subequations}
	\begin{align}
		& D^V(y,s)   :=  \left(  1 + \frac{\wt{W}_y + 2\overline{W}'}{1-\dot{\tau}} + \frac{e^{s/2}Z_y}{1-\dot{\tau}} \right) + \frac{2}{y(1+y^2)}\left(\frac{\overline{W}+\wt{W}}{1-\dot{\tau}}+\frac{3}{2}y + \sigma  \right), \label{DV0}\\
	& F^V(y,s)  :=  - \frac{1+y^2}{y^2} \left( \left(  \sigma + \frac{\dot{\tau}\overline{W}}{1-\dot{\tau}}   \right) \overline{W}'' 
		 + \left(\frac{\dot{\tau}\overline{W}'}{1-\dot{\tau}}+\frac{e^{s/2}Z_y}{1-\dot{\tau}}\right)\overline{W}' + \frac{2e^s}{1-\dot{\tau}}\Phi_{yy} \right), \label{FV-0} \\
		& K^V(y,s;y')   := - \mathbb{I}_{[0,y]}(y') \frac{\overline{W}''(y)}{1-\dot{\tau}} \frac{1+y^2}{y^2} \frac{(y')^2}{1+(y')^2},
	\end{align}
\end{subequations}
where $\mathbb{I}_{[0,y]}$ denotes an indicator function, i.e., $\mathbb{I}_{[0,y]}(y') = 1$ for $y'\in[0,y]$ and it is zero otherwise. 
An important point here is that a part of $D^V$ has a non-negative lower bound. More precisely, it holds that 
\begin{equation}\label{WbLb}
	1+2\overline{W}'+\frac{2}{y(1+y^2)}\left(\frac{3y}{2}+\overline W\right)\geq \frac{y^2}{5(1+y^2)}+\frac{16y^2}{5(1+8y^2)}
\end{equation}
for all $y\in \mathbb{R}$. 
This is verified in \eqref{0605_m_3}.
We will show that the remaining terms
are   absorbed into the lower bound on the region $|y| \geq l$ for some small $l>0$. To analyze the region around $y=0$, we use the Taylor approximation. 
\begin{lemma}\label{Wy1_lem}
		If the same assumptions as in Proposition~\ref{mainprop} hold, then we have 
	\begin{equation*}
		|W_y(y,s)-\overline{W}'(y)|\leq \frac{y^2}{20(1+y^2)} \qquad \text{ for all } y \in \mathbb{R}, \quad s\in[s_0,\sigma_1]. 
	\end{equation*}
\end{lemma}

\begin{proof}
In view of the definition of $V(y,s)$, i.e., \eqref{Def-V}, it suffices to show \[\|V(\cdot, s)\|_{L^\infty(\mathbb{R})} \leq \frac{1}{20}.\] 
We first estimate $V(y,s)$ near $y=0$. By expanding $\wt{W}_y$, we have 
\begin{equation*}
	\wt{W}_y(y,s)= \frac{y^2}{2} \left( \partial_y^3 W(0,s) - 6 \right) +\frac{y^3}{6} \left(\partial_y^4 W(y',s) - \partial_y^4\overline{W}(y') \right) \quad \text{ for some } y'\in(-|y|, |y|)
\end{equation*}  
since $\wt{W}_y=\wt{W}_{yy} = 0$ at $y=0$ and $\partial_y^3\overline{W}(0) = 6$. Multiplying the above equation by  $(1+y^2)/y^2$ and using \eqref{EP2_1D4}, we get 
\[
|V(y,s)| \leq \frac{1+y^2}{2}|\partial_y^3 W(0,s) - 6| + \frac{1+y^2}{6}|y| \left( M  + \|\partial_y^4\overline{W}\|_{L^\infty} \right).
\]
Now, we first choose  small $l=l(M )>0$ such that 
\[\frac{1+l^2}{6}|l| \left( M  + \|\partial_y^4\overline{W}\|_{L^\infty} \right) < \frac{1 }{80}, \qquad |y|\leq l,\] 
and then choose $\ve_0>0$ such that 
\[  \frac{1+l^2}{2}|\partial_y^3 W(0,s) - 6| < \frac{1}{80}.\]
Here we have used Lemma~\ref{W30_lem}, i.e., $|\partial_y^3 W(0,s) - 6| \le C \ve^{1/3}$, with a choice of sufficiently small $\veps_0$.
Thus we have 
\begin{equation}\label{V-l}
|V(y,s)| 
< \frac{1}{40}, \qquad  |y|\leq l. 
\end{equation}

Next, we consider the region $|y| \geq l$.
%
Note that $D^V$ defined in \eqref{DV0} can be regrouped and 
 written as  
\begin{multline*}
D^V  = 	 \left(1+2\overline{W}'+\frac{2}{y(1+y^2)}\left(\frac{3y}{2}+\overline{W}\right)\right)
\\
+ \left( \frac{e^{s/2}Z_y}{1-\dot{\tau}}+\frac{\wt{W}_y+2\dot{\tau}\overline{W}'}{1-\dot{\tau}}+\frac{2}{y(1+y^2)}\left(\frac{\wt{W}}{1-\dot{\tau}}+\frac{\dot{\tau}\overline{W}}{1-\dot{\tau}} + \sigma \right)\right).
\end{multline*}
Each term on the right hand side can be estimated as follow. From \eqref{Boot_3}, we have 
\begin{equation}\label{V-l1}
\left|\frac{e^{s/2}Z_y}{1-\dot{\tau}}\right|\leq 2e^{-s}  \leq 2\veps.
\end{equation}
Using \eqref{tau}, \eqref{Boot_1} and the fact that $|\overline{W}'|\leq 1$, we see that
\begin{equation}\label{V-l2}
	\left|\frac{\wt{W}_y+2\dot{\tau}\overline{W}'}{1-\dot{\tau}}\right|\leq (1+4\veps)\left(\frac{y^2}{10(1+y^2)}+4\veps\right)  = \frac{y^2}{10(1+y^2)} + O(\veps).
\end{equation}
Using \eqref{tau} and \eqref{Wbd1}, we get
\begin{equation} \label{V-l3}
\left|\frac{2(\wt{W} + \dot{\tau} \overline{W})}{y(1+y^2)(1-\dot{\tau})}  \right|  \leq (1+4\veps)\left(\frac{y^2}{15(1+y^2)} + 8\veps\right) = \frac{y^2}{15(1+y^2)} + O(\veps).
\end{equation}
On the other hand, \eqref{temp_2} implies
\begin{equation}\label{2.86}
	\begin{split}
		\left|\frac{2\sigma}{y(1+y^2)}\right|&  \leq 4e^{-s}+\frac{2}{|y|}e^{-s/3} \leq \frac{4}{l}e^{-s/3}
	\end{split}
\end{equation}
for all $|y|\geq l$.  
Combining \eqref{V-l1}--\eqref{2.86} together with \eqref{WbLb}, for sufficiently small $\ve>0$,  we obtain 
	\begin{equation}\label{DV}
		\begin{split}
			D^V(y,s) &\geq \left( \frac{y^2}{5(1+y^2)}+\frac{16y^2}{5(1+8y^2)} \right) -\frac{y^2}{10(1+y^2)} - \frac{y^2}{15(1+y^2)} -  \frac{4}{l}\veps^{1/3} - O(\veps)
			\\
			&\geq \frac{3y^2}{1+8y^2} + \frac{y^2}{30(1+y^2)}  
		\end{split}
	\end{equation}
	for all $|y|\geq l$. Since $ \frac{3y^2}{1+8y^2}$ and $\frac{y^2}{30(1+y^2)}$ are increasing on $y>0$, we have
	\begin{equation}\label{DVL}
		D^V (y,s)  \ge \frac{3l^2}{1+8l^2}+\frac{l^2}{30(1+l^2)}=: \lambda_D^V, \qquad  |y|\ge l.
	\end{equation} 

\textit{Estimate of $K^V$}: Now we  estimate the term in \eqref{Eq_V} involving $K^V$. Using \eqref{tau} and \eqref{0605_m_4}, we see that  
	\begin{equation}\label{kernel}
		\begin{split}
			\int_{\mathbb{R}}{|K^V(y,s;y')|\,dy'} &\leq \frac{|\overline{W}''(y)|}{| 1-\dot{\tau} | } \frac{ 1+y^2}{y^2} \int^{|y|}_0{\frac{(y')^2}{1+(y')^2 }\,dy'}
			\\
			&\leq \frac{1}{\lambda} \left(\frac{3y^2}{1+8y^2} + \frac{y^2}{30(1+y^2)} \right)
		\end{split}
	\end{equation}
for some $\lambda>1$.
From \eqref{DV} and \eqref{kernel}, we deduce that 
\begin{equation}\label{K-DV}
	\int_{\mathbb{R}}|K^V(y,s;y')|\,dy'  \le \frac{1}{\lambda} D^V(y,s), \qquad |y|\ge l.  
\end{equation}

%
%

\textit{Estimate of $F^V$}: Now we shall estimate   the forcing term, $F^V$ defined in \eqref{FV-0}:
\[F^V(y,s)  =  - \frac{1+y^2}{y^2} \left( \left(  \sigma + \frac{\dot{\tau}\overline{W}}{1-\dot{\tau}}   \right) \overline{W}'' 
		 + \left(\frac{\dot{\tau}\overline{W}'}{1-\dot{\tau}}+\frac{e^{s/2}Z_y}{1-\dot{\tau}}\right)\overline{W}' + \frac{2e^s\Phi_{yy}}{1-\dot{\tau}} \right). \]
   We note that $\frac{1+y^2}{y^2}$ decreases on $y>0$.  
     By   \eqref{Boot_3}, \eqref{Phiy2}, \eqref{temp_2} and \eqref{dottau},   
\begin{equation}\label{FV}
\begin{split}
|F^V(y,s)| 
& \leq  \frac{1+l^2}{l^2} \left(\left| \sigma \overline{W}'' \right| +   \left| \frac{\dot{\tau}}{1-\dot{\tau}}\overline{W}\overline{W}''\right|  + \left|\frac{\dot{\tau}\overline{W}'^2}{1-\dot{\tau}} \right| + \left|\frac{e^{s/2}Z_y\overline{W}'}{1-\dot{\tau}} \right| + \left| \frac{2e^{s}\Phi_{yy}}{1-\dot{\tau}}\right| \right) \\
& \leq \frac{1+l^2}{l^2} \left(e^{-s/3}(|y| +1)|\overline{W}''| + O(\veps) \right) \le C \ve^{1/3}  =: F_0^V, \qquad |y| \geq l.
\end{split}
\end{equation} 
 In the above inequality, using \eqref{W<y} together with \eqref{dottau},   we estimated the second term in the right hand side as  
 \[ \left| \frac{\dot{\tau}}{1-\dot{\tau}}\overline{W}\overline{W}''\right| \le C e^{-s} |y| | \overline W''|. \] 
For the last inequality in \eqref{FV}, we used  \eqref{Wyybar_bdd} implying $| \overline W'' | \le C ( 1+ |y|^{-5/3})$.  By choosing sufficiently small $\ve>0$, one can make $\|F^V(\cdot,s)\|_{L^{\infty}(|y|\geq l)}$ relatively small compared to the lower bound of $D^V$ in \eqref{DVL}. 
On the other hand, from  \eqref{Utildey_M} and \eqref{W-y2},  
we have 
\begin{equation}\label{V-l-0} 
|V(y,s_0)|\leq \frac{1}{40}, \qquad \limsup_{|y|\rightarrow \infty}|V(y,s)|=0.
\end{equation} 
For sufficiently small $\ve_0>0$, we apply   Theorem~\ref{max_2} with \eqref{V-l}, \eqref{DV} and \eqref{K-DV}--\eqref{V-l-0}, 
to conclude that 
$$\|V(\cdot,s)\|_{L^\infty(\mathbb{R})}\leq \frac{1}{20}.$$
This completes the proof. 
\end{proof}

\subsubsection{The second derivative of $W$}

Lemma~\ref{Wy2_lem} gives a stronger estimate than \eqref{EP2_1D3}. The strategy is similar to Lemma~\ref{W30_lem}.
\begin{lemma}\label{Wy2_lem}
	If the same assumptions as in Proposition~\ref{mainprop} hold, then we have 
	\begin{equation*}
		|W_{yy}(y,s)|\leq \frac{14|y|}{(1+y^2)^{1/2}}, \qquad   y \in \mathbb{R}, \quad s\in[s_0,\sigma_1].  
	\end{equation*}
\end{lemma}

\begin{proof}
We first estimate $W_{yy}$ near $y=0$. Expanding $W_{yy}$ at $y=0$, and then using  \eqref{EP2_1D4} and \eqref{str_U3}, we get
\begin{equation}\label{W_yy1}
\begin{split}
|W_{yy}(y,s)|
& \leq |y||\partial_y^3W(0,s)|+\frac{y^2}{2}|\partial_y^4W(y',s)| \quad (\text{for some }y'\in(-|y|, |y|))
\\
& \leq |y|(6 + C\veps^{1/3} + \frac{M }{2}|y|)
\\
&
\le \frac{7|y|}{(1+y^2)^{1/2}}, \qquad |y| \leq \frac{1}{M}.
\end{split}
\end{equation}
 Here  the last inequality holds since $|y|\le 1/M$ and $M\ge 10$, and $0<\ve<\ve_0$, where $\ve_0$ is chosen 
sufficiently small.
%

In what follows, we consider the region $|y| > 1/M $. 
Define \[\wt{V}(y,s): =\frac{(1+y^2)^{1/2}}{y} W_{yy}(y,s).\]
From \eqref{EP2_2D2},  $\wt{V}(y,s)$ satisfies 
\begin{equation}\label{Eq_2nd}
	\partial_s\wt{V} +D_2^{\wt V} \wt{V}  + U^W \wt{V}_y = F_2^{\wt V},
\end{equation}
where 
\begin{subequations}
\begin{align*}
D_2^{\wt V}(y,s) &   := \frac{5}{2}+\frac{3W_y}{1-\dot{\tau}}+\frac{2e^{s/2}Z_y}{1-\dot{\tau}}+\frac{1}{y(1+y^2)}\left(\frac{W}{1-\dot{\tau}}+\frac{3}{2}y + \sigma \right),\\
F_2^{\wt V}(y,s) & :=-\frac{2e^s\partial_y^3\Phi}{1-\dot{\tau}}\frac{(1+y^2)^{1/2}}{y}-\frac{e^{s/2}Z_{yy}W_y}{1-\dot{\tau}}\frac{(1+y^2)^{1/2}}{y}.
\end{align*}
\end{subequations}

\textit{Lower bound of $D_2^{\wt V}$}:  Rearranging the terms in  $D_2^{\wt V}$ with a simple trick $1/(1-\dot\tau) = 1 + \dot\tau / (1-\dot\tau)$,  we have   
\begin{multline}\label{D2Vt} 
D_2^{\wt V} (y,s) =  \left( \frac{5}{2}+3\overline{W}'+\frac{1}{y(1+y^2)}\left(\frac{3y}{2}+\overline{W}\right)   \right)	+ \frac{3}{1-\dot{\tau}}(\wt{W}_y+\dot{\tau}\overline{W}')
\\
+\frac{2e^{s/2}Z_y}{1-\dot{\tau}}+\frac{1}{y(1+y^2)}\left( \frac{\wt{W}+\dot{\tau}\overline{W}}{1-\dot{\tau}}+\sigma\right).
\end{multline}
Note that 
 by \eqref{0605_m_5}, we see that the first grouped terms in \eqref{D2Vt} satisfy    
	\begin{equation*}
		\frac{5}{2}+3\overline{W}' +\frac{1}{y(1+y^2)}\left(\frac{3y}{2}+\overline{W}\right)\geq \frac{2y^2}{3(1+y^2)},  \qquad y\in\mathbb{R}.
	\end{equation*}
The remaining terms in \eqref{D2Vt},  
\[ \frac{3}{1-\dot{\tau}}(\wt{W}_y+\dot{\tau}\overline{W}')
+\frac{2e^{s/2}Z_y}{1-\dot{\tau}}+\frac{1}{y(1+y^2)}\left( \frac{\wt{W}+\dot{\tau}\overline{W}}{1-\dot{\tau}}+\sigma\right)\]
are estimated as follow. 
From \eqref{Boot_1}, we have 
\[ \left|\frac{3}{1-\dot{\tau}}\wt{W}_y \right| \le \frac{ y^2}{3(1+ y^2)}.\] 
From \eqref{Boot_3} and \eqref{dottau}, 
\[ \left| \frac{3\dot{\tau}\overline{W}'}{1-\dot{\tau}}   +\frac{2e^{s/2}Z_y}{1-\dot{\tau}} \right| \le C e^{-s} \le C\ve.  \]
Noting from \eqref{Boot_1} that 
\[ | \wt{W} (y,s) | \le \left|  \int_0^{y}   \wt{W}_y (y',s)   \; ds' \right|  \le \int_0^{|y|} \frac{(y')^2}{10 ( 1+ (y')^2 )} \; dy' \le \frac{|y|^3}{30}, \]
we estimate using \eqref{tau}, 
\[ \left| \frac{1}{y(1+y^2)}  \frac{\wt{W} }{1-\dot{\tau}} \right| \le \left( \frac{1}{30} + \ve \right)  \frac{y^2}{ ( 1+y^2)} \le   \frac{y^2}{ 30 ( 1+y^2)} + \ve. \]
By \eqref{tau}, \eqref{Wbd1} and \eqref{temp_2},   we have for $|y| \ge 1/M$, 
\begin{equation*}
    \begin{split}
 \left| \frac{ 1 }{ y ( 1+y^2) }  \left( \sigma + \frac{\dot\tau}{1-\dot\tau } \overline{W} \right) \right|
 &\le
 \frac{ 2e^{-s/3} + 2 |y | e^{-s} + | \dot\tau| | \overline{W} | }{|y|(1+y^2)} 
      \le 2\ve^{1/3}M + C\ve. 
 \end{split}
 \end{equation*}
Combining all, we see that there is a sufficiently small $\veps_0$ such that if $\veps \in (0,\veps_0)$, then it holds that 
 \begin{equation}\label{DWyy}
	\begin{split}
		D_2^{\wt V} (y,s) 
		& \geq \frac{2y^2}{3(1+y^2)} -\frac{1}{3} \frac{y^2}{1+y^2} - \frac{y^2}{30(1+y^2)} - 2\veps^{1/3} M   - O(\veps)  
  \\
  & \geq  \frac{y^2}{4(1+y^2)}   \geq \frac{1}{4(M ^2+1)}
	\end{split}
\end{equation}
for all  $|y|\geq 1/M$. 

\textit{Estimate of $F_2^{\wt V}$}: By \eqref{tau}, \eqref{Uy1}, \eqref{Phiy34} and \eqref{Zyy_56},  it is straightforward to see that for sufficiently small $\ve_0>0$ if $\ve\in(0,\ve_0)$, then  there exists a positive constant $C=C(M,\|(w_0,z_0)\|_{L^2},P_{\pm})$ such that
\begin{equation*}
    \begin{split}
       | F_2^{\wt V}(y,s) |  &  = \left| \frac{(1+y^2)^{1/2}}{y} \left(  \frac{2e^s\partial_y^3\Phi}{1-\dot{\tau}} +
       \frac{e^{s/2}Z_{yy}W_y}{1-\dot{\tau}}
       \right) \right| 
       \\ 
       & \le C (  e^{ -5s/2} + e^{-s/3} )
        \le Ce^{-s/3}  
    \end{split}
\end{equation*}
for all $|y|\ge 1/M$.
 Thus we have
\begin{equation}\label{FWyy} 
		\|F_2^{\wt V} (\cdot,s)\|_{L^{\infty}(|y|\geq 1/M )} \leq Ce^{-s/3}. 
\end{equation} 
 
\textit{Asymptotic behavior of $\wt{V}$ at $|y|=\infty$}: We note that \eqref{1D3'} implies 
\begin{equation}\label{wtV_init}
	\|\wt{V}(\cdot,s_0)\|_{L^{\infty}(\mathbb{R})}\leq 7.
\end{equation}   
We claim that 
\begin{equation}\label{wtV_init1}
\limsup_{|y|\rightarrow \infty}|\wt{V} (y, s) |< 14.
\end{equation}
 Thanks to \eqref{DWyy} and \eqref{FWyy}, 
one has   
\begin{equation*}
	\inf_{|y|\geq 1/M }D_2^{\wt V} (y,s)\geq \lambda^{\wt{V}}_D(M ), \qquad  \|F_2^{\wt V} (\cdot,s)\|_{L^{\infty}(|y|\geq 1/M )}\leq F^{\wt{V}}_0(M ) e^{-\lambda_F^{ \wt{V}} s},
\end{equation*}
with   
\[   \lambda^{\wt{V}}_D:= \frac{1}{4(M ^2+1)}, \quad \lambda_F^{\wt{V}}:= \frac13 \quad  \text{ and}\quad F_0^{\wt{V} }  := C. \] 
We see that $ \lambda^{\wt{V}}_D < \lambda_F^{\wt{V}}$ for any $M\ge1$.  Then, together with \eqref{UW_far}, we apply Lemma~\ref{rmk2} to obtain that 
\begin{equation*}
\begin{split}
	\limsup_{|y|\rightarrow \infty}|\wt{V}(y,s)|
 &	\le  \limsup_{|y|\rightarrow \infty}|\wt{V}(y,s_0)| + C e^{-s_0/3}
   \le 7 + C \ve^{1/3} <14.
\end{split}
\end{equation*}
The second inequality holds thanks to  \eqref{wtV_init} and the last one holds by choosing $\veps_0$ sufficiently small. Thus  \eqref{wtV_init1} is proved.

Then we apply Lemma~\ref{max_2} with \eqref{W_yy1}, \eqref{DWyy}, \eqref{FWyy}, \eqref{wtV_init} and \eqref{wtV_init1}. This  completes the proof of Lemma~\ref{Wy2_lem}. 
\end{proof}

\subsubsection{The third derivative of $W$}

Now we close \eqref{EP2_1D5}.
\begin{lemma}\label{Wy3_lem}
	If the same assumptions as in Proposition~\ref{mainprop} hold, then we have 
	\begin{equation*}
		\|\partial_y^3W(\cdot,s)\|_{L^{\infty}} \leq \frac{M^{5/6}}{2}
	\end{equation*}
	 for $s\in[s_0,\sigma_1]$.
\end{lemma}

\begin{proof}
Let us first consider  the region near $y=0$. Expanding $\partial_y^3W$ at $y=0$, and then using \eqref{EP2_1D2} and \eqref{EP2_1D4}, we have  
\[
|\partial_y^3W(y,s)| \leq |\partial_y^3W(0,s)| +  |y|M \le 7+\frac{M_3}{8}  <  \frac{M_3}{4}
\]
for all $|y|\leq  \frac{M_3}{8M } = \frac18{M^{-1/6}}$ with sufficiently large $M_3 = M^{5/6} >56$.
 
We recall the equation for $\partial_y^3W$ \eqref{EP2_2D3}: 
\[ 
\partial_s \partial_y^3W + D^W_3 \partial_y^3W+U^W \partial_y^4W=F^W_3,
\]
where $U^W$ is defined in \eqref{UW} and
	\begin{subequations}
		\begin{align*}
			D^W_3&:=4 +\frac{4W_y}{1-\dot{\tau}}+\frac{3e^{s/2}Z_y}{1-\dot{\tau}},
			\\
			F^W_3&:=-\frac{2e^s\partial_y^4\Phi}{1-\dot{\tau}}-\frac{e^{s/2}}{1-\dot{\tau}}(\partial_y^3ZW_y+3Z_{yy}W_{yy})-\frac{3W_{yy}^2}{1-\dot{\tau}}.
		\end{align*}
	\end{subequations}
	
\textit{Estimate of $D^W_3$}: We  obtain a lower bound of $D^W_3$. Note that by \eqref{y-w-y2}, we have
\begin{equation*}
		1+\overline{W}'\geq 1-\frac{1}{1+\frac{3y^2}{(3y^2+1)^{2/3}}}
		\geq \frac{y^2}{3(1+y^2)}.
\end{equation*}
Using this and \eqref{Boot_3}, \eqref{Boot_1}, \eqref{Uy1},  \eqref{dottau},   we see that there is a constant $C>0$ such that 
	\begin{equation*}
		\begin{split}
			D^W_3&=4(1+\overline{W}')+4\wt{W}_y+\frac{4\dot{\tau}W_y}{1-\dot{\tau}}+\frac{3e^{s/2}Z_y}{1-\dot{\tau}}
			\\
			&\geq 4(1+\overline{W}')-4|\wt{W}_y| - C\veps
			\\
			&\geq \frac{4y^2}{3(1+y^2) }-\frac{4y^2}{10(1+y^2)} - C\veps .
		\end{split}
	\end{equation*} 
	Hence, there is a small number $\veps_0(M)>0$ such that for all $\veps\in(0,\veps_0)$,  it holds that
	\begin{equation}\label{Wy3D}
		\begin{split}
			D^W_3&\geq \frac{y^2}{2(1+y^2)}
    \ge   \frac{M_3^2}{2(64M ^2+M_3^2)} 
		\end{split}
	\end{equation}
	as long as $|y|> \frac{M_3}{8M}  >0$. 

\textit{Estimate of $F^W_3$}: 	
	By \eqref{tau}, \eqref{EP2_1D3}, \eqref{Uy1},  \eqref{Phiy34} and  Lemma~\ref{difZn}, we obtain that for all $\ve\in(0,\ve_0)$  with  sufficiently small $\veps_0>0$,
	\begin{equation}\label{Wy3F}
		\begin{split}
			|F^W_3|&=\left|-\frac{2e^s\partial_y^4\Phi}{1-\dot{\tau}}-\frac{e^{s/2}}{1-\dot{\tau}}(\partial_y^3ZW_y+3Z_{yy}W_{yy})-\frac{3W_{yy}^2}{1-\dot{\tau}}\right|
			\\
			&\leq (1+4\veps ) ( 2\|e^s\partial_y^4\Phi\|_{L^{\infty}}+\|e^{s/2}\partial_y^3Z\|_{L^{\infty}}\|W_y\|_{L^{\infty}}
  \\
  & \qquad \qquad \qquad \qquad \qquad 
+3\|e^{s/2}Z_{yy}\|_{L^{\infty}}\|W_{yy}\|_{L^{\infty}}+3\|W_{yy}\|_{L^{\infty}}^2 )
			\\
			&\leq (1+4\veps )
   ( 2e^{-5s/2}+e^{-s/3}+45e^{-s/3}+3\cdot 15^2 ) 
   \\
			&\leq 4\cdot 15^2.
		\end{split}
	\end{equation}

\textit{Asymptotic behavior of $\partial_y^3W$ at $|y|=\infty$}: We claim that 
\begin{equation*}
\limsup_{|y|\rightarrow \infty}|\partial_y^3W|\leq \frac{M_3}{4}.
\end{equation*}
From \eqref{Wy3D}, we see that $D^W_3$ satisfies 
	\begin{subequations}
		\begin{align*}
			&\inf_{|y|\geq 1}D^W_3(y,s) \geq \frac{1}{2}.
		\end{align*}
	\end{subequations}
	Therefore, with $\lambda_D=1/2$, $\lambda_F=0$ and $F_0=4\cdot 15^2$, Lemma~\ref{rmk2}  implies that
	\begin{equation*}
	\begin{split}
		\limsup_{|y| \rightarrow \infty}|\partial_y^3W(y,s)| 
		& \leq \limsup_{|y|\rightarrow \infty}|\partial_y^3W(y,s_0)|e^{-0.5(s-s_0)}+ 8\cdot 15^2 \\
		&  \leq  \frac{M_3}{4}
		\end{split}
	\end{equation*}
for sufficiently large $M_3>0$. Here, we have used \eqref{init_24} in the second inequality.

To finish the proof, we apply Lemma~\ref{max_2} with $\Omega =\{y: |y| \leq \frac{M_3}{8M} \}$,  $m_0=\frac{M_3}{4}$, $\lambda_D = \frac{M_3^2}{2(64M ^2+M_3^2)}$, $F_0 = 4\cdot 15^2$, $\delta=0$. Recalling  $M_3:= M^{5/6}$, we see that  for sufficiently large number $M>0$, 
	\begin{equation*}
m_0\lambda_D=	\frac{M_3}{4}\frac{M_3^2}{2(64M ^2+M_3^2)}>  2\cdot 15^2 = F_0/2.
	\end{equation*} 
Combining the above estimates, we finish the proof.
\end{proof}

\subsubsection{The fourth derivative of $W$}
The following lemma closes the bootstrap assumption \eqref{EP2_1D4}.
\begin{lemma}\label{Wy4_lem}
	If the same assumptions as in Proposition~\ref{mainprop} hold, then we have 
	\begin{equation*}
		\|\partial_y^4W(\cdot,s)\|_{L^{\infty}}\leq \frac{M }{2},
	\end{equation*}
	for $s\in[s_0,\sigma_1]$.
\end{lemma}

\begin{proof}	
	Recalling \eqref{EP2_2D4}, we see that $\partial_y^4W$ satisfies
	\begin{equation*}
		\partial_s \partial_y^4W + D^W_4\partial_y^4W +U^W \partial_y^5W = F^W_4,
	\end{equation*}
	where 
	\begin{subequations}
		\begin{align*}
			D^W_4&:= \frac{11}{2} +\frac{5W_y}{1-\dot{\tau}}+\frac{4e^{s/2}Z_y}{1-\dot{\tau}},
			\\
			F^W_4&:= -\frac{10W_{yy}\partial_y^3W}{1-\dot{\tau}} -\frac{2e^s\partial_y^5\Phi}{1-\dot{\tau}} -\frac{e^{s/2}}{1-\dot{\tau}}(\partial_y^4ZW_y+4\partial_y^3ZW_{yy}+6Z_{yy}\partial_y^3W).
		\end{align*}
	\end{subequations} 
	By \eqref{tau}, \eqref{Boot_3} and \eqref{Uy1}, we have for sufficiently small $\veps>0$,
	\begin{equation}\label{DW4}
		\begin{split}
			D^W_4&\geq \frac{11}{2}-\frac{5|W_y|}{1-\dot{\tau}}-\frac{4e^{s/2}|Z_y|}{1-\dot{\tau}} 
			\geq   \frac{11}{2} - 5 + O(\veps)
			 \geq \frac{1}{4}.
		\end{split}
	\end{equation} 
Using Lemma~\ref{difZn}, \eqref{tau}, \eqref{EP2_1D3}, \eqref{EP2_1D5} and \eqref{Uy1},  it is straightforward to check that for all sufficiently small $\veps$,  
	\begin{equation}\label{FW4} 
			 | F^W_4 |  \leq 200 M_3 + 1.
	\end{equation} 

Integrating the above equation for $\partial_y^4W$ along the flow associated with $U^W$, and then using  \eqref{DW4} and \eqref{FW4}, we get 
	\begin{equation*}
 \begin{split} 
\|\partial_y^4W\|_{L^{\infty}}
& \leq 	\|\partial_y^4W(\cdot,s_0)\|_{L^{\infty}}e^{-\frac{1}{4}(s- s_0)} +(200M_3 + 1) \int^s_{s_0}e^{-\frac{1}{4}(s-s')}
\\
& \leq 1 + 4(200M_3 + 1) \leq \frac{M }{2}
\end{split} 
	\end{equation*} 
for sufficiently large $M$ (recall $M_3 = M^{5/6}$). 
Here, we have used \eqref{1D4}. This proves Lemma~\ref{Wy4_lem}.
\end{proof}

\section{Closure of Bootstrap II}\label{sec 4}
In this section, we prove Proposition~\ref{decay_lem}.  We first remark that the assumption \eqref{Boot_L} implies that 
			\begin{equation}
				\|(y^{2/3}+8)\Phi_{yy}\|_{L^{\infty}}\leq Ce^{-2s}\label{Phiyy_decay}
			\end{equation}
			for some positive constant $C$. To see this, we consider the change of variable  $\tilde{y} = e^{-3s/2}y$. Then,  we have from \eqref{EP2_3} that
			\[
			-\wt{\Phi}_{\tilde{y}\tilde{y}} = \tilde{f}(\tilde{y},s) + 1- e^{\wt{\Phi}},
			\]
			where $\wt{\Phi}(\tilde{y},s) = \Phi(y,s)$ and $\tilde{f}(\tilde{y},s) = e^{\frac{e^{-s/2}W + \kappa - Z}{2\sqrt{K}}} -  1$. 
Thanks to \eqref{Boot_L}, we also see that  
  	\begin{equation*}
			|(1+\tilde{y}^{2/3}) \tilde{f}(\tilde{y},s)| \le 	\|(1+e^{-s}y^{2/3})f(y,s)\|_{L^{\infty}_y(\mathbb{R})} \le   C_{\tilde{f}}:= \frac{1}{14}.
			\end{equation*}
Let $I(y)$ be a continuous function defined by 
\begin{equation}
    \label{B_I_0}
  I(y):=  (y^{2/3} +1) \int_{-\infty}^\infty \frac{e^{-|y-y'|}}{ 1+|y'|^{2/3}}
    dy', \quad y\in\mathbb{R}.
    \end{equation}
    Then from the fact that $\sup I \le 3.5$(see below), we have   
    $$C_{\tilde f} \leq \min\left\{\frac{1}{4\sup_{y\in\mathbb{R}} I}, \frac{ 2e^{-1/4}}{(\sup_{y\in\mathbb{R}} I)^2}\right\}.$$
%
Applying Lemma~\ref{Lem_Pois}, we obtain \eqref{Phiyy_decay}, i.e., 
			\begin{equation*}
				1 \gtrsim |(\tilde{y}^{2/3}+1)\wt{\Phi}_{\tilde{y}\tilde{y}}| = |(y^{2/3}+e^s)e^{2s}\Phi_{yy}|\geq |(y^{2/3}+8)e^{2s}\Phi_{yy}|.
			\end{equation*} 
			We can check that 
			\begin{equation}\label{supI}
				\sup_{y\in\mathbb{R}} I \leq 3.5,
			\end{equation}
			in the following way.
			Since $I(y)$ is an even function, without loss of generality, we assume $y\geq 0$. First, we split $I(y)$ into
			\begin{equation*}
				\begin{split}
					I(y)&=(y^{2/3}+1)\int^{y/2}_{-\infty} \frac{e^{-|y-y'|}}{(y')^{2/3}+1}\,dy' + (y^{2/3}+1)\int^{\infty}_{y/2}\frac{e^{-|y-y'|}}{(y')^{2/3}+1}\,dy'
					\\
					&=:I_1(y)+I_2(y).
				\end{split}
			\end{equation*} 
			By the facts 
			\begin{equation*}
				I_1(y)\leq (y^{2/3}+1)e^{-y}\int^{y/2}_{-\infty}e^{y'}\,dy' = (y^{2/3}+1)e^{-y/2}
			\end{equation*}
			and
			\begin{equation*}
				I_2(y)\leq \frac{y^{2/3}+1}{(y/2)^{2/3}+1}\int^{\infty}_{y/2}e^{-|y-y'|}\,dy' \leq \frac{y^{2/3}+1}{(y/2)^{2/3}+1}(2-e^{-y/2}),
			\end{equation*}
			we see that $I(y)$ has an upper bound as 
			\begin{equation*}
				I(y)\leq (y^{2/3}+1)e^{-y/2}+\frac{y^{2/3}+1}{(y/2)^{2/3}+1}(2-e^{-y/2}) \leq 3.5.
			\end{equation*}
			Here, we can check the last inequality by a straightforward calculation. This proves \eqref{supI}.

The inequality \eqref{Phiyy_decay} will be used to prove  the estimates for $(y^{2/3} + 8)\wt{W}_y$ in Lemma~\ref{mainprop_2}--\ref{mainprop_1}.
	
\subsection{Weighted estimate of $\rho-1$}	
		The following lemma proves \eqref{pi}.
		\begin{lemma}\label{Ptilde_lem}
			If the same assumptions as in Proposition~\ref{decay_lem} hold, then we have  
			\begin{equation*}
			\sup_{s\in[s_0,\sigma_1], \; y\in \mathbb{R} } (1+e^{-s}y^{2/3})\left|e^{\frac{e^{-s/2}W+\kappa-Z}{2\sqrt{K}}}-1\right|\le \frac{3}{56}.
			\end{equation*} 
		\end{lemma}

		\begin{proof}
	Let $\wt{P}:=(1+e^{-s}y^{2/3})(e^{P}-1)$. From \eqref{Peq}, we have
		\begin{multline}\label{weightP}
			\partial_s\wt{P} + U^Z\wt{P}_y  
			=
			-\frac{2W_y}{1-\dot{\tau}}(1+e^{-s}y^{2/3}) 
			\\
			-  \left( \frac{2W_y}{1-\dot{\tau}}-\frac{2}{3}\frac{e^{-s}}{y^{1/3}(1+e^{-s}y^{2/3})}\left( \frac{W}{1-\dot{\tau}} + \sigma - \frac{4\sqrt{K}e^{s/2}}{1-\dot{\tau}} \right)\right)\wt{P} =: F_P.
		\end{multline}
		Let us first claim that
		\begin{equation}\label{tP_F_fin}
				|F_P(y,s)|\leq Ce^{-s/3} \quad \text{as long as } |y|\geq e^s.
		\end{equation}
  From \eqref{Utildey_M} and \eqref{y-w-y}, we have 
  \begin{equation}\label{w-y-y23} |W_y(y,s) | \le \frac{1}{y^{2/3} +8} + | \overline W'(y) | \le Cy^{-2/3}.
  \end{equation}
  This together with $W(0,s) =0$ yields  
  \[ | W(y,s) | \le C |y |^{-1/3}. \] 
  Then, 
  it holds that 
		\begin{equation}\label{temp_P_1}
			|W_y(y,s)| \leq  C e^{-2s/3}, \quad \left|\frac{W(y,s)}{y}\right|\leq C e^{-2s/3}, \qquad |y|\geq e^s.
		\end{equation}		
By \eqref{w-y-y23} and \eqref{temp_P_1},   we see that there is a constant $C>0$ such that 
		\begin{equation*}
			\begin{split}
				\left|\frac{2W_y}{1-\dot{\tau}}(1+e^{-s}y^{2/3}) \right|
    & \le \left|\frac{2W_y}{1-\dot{\tau}}  \right| + \left|\frac{2 y^{2/3} W_y}{1-\dot{\tau}}  e^{-s} \right|
				\\
				&\leq Ce^{-2s/3}, \qquad |y|\geq e^s.
			\end{split}
		\end{equation*}
		Next, using \eqref{Boot_L} and \eqref{temp_P_1}, we get  as long as  $|y| \geq e^s$, 
\begin{equation*}
\left|\frac{2W_y}{1-\dot{\tau}} - \frac{2}{3}\frac{e^{-s}}{y^{1/3}(1+e^{-s}y^{2/3})}\frac{W}{1-\dot{\tau}} \right| |\wt{P}| \leq  \left| \frac{2W_y}{1-\dot{\tau}}\right| |\wt{P}| +\left| \frac{4}{3}\frac{W}{y}\right|  |\wt{P}| \leq C e^{-2s/3}.
\end{equation*}
Using \eqref{Boot_L} and \eqref{temp_2}, we obtain
\begin{equation*}
\begin{split}
\left| \frac{2}{3}\frac{e^{-s}}{y^{1/3}(1+e^{-s}y^{2/3})}\left( \sigma  - \frac{4\sqrt{K}e^{s/2}}{1-\dot{\tau}} \right) \wt{P} \right| &\leq \left| \frac{2}{3y}\left( \sigma  - \frac{4\sqrt{K}e^{s/2}}{1-\dot{\tau}} \right) \wt{P} \right|
\\
&  \leq C \left(  \frac{1+|y|}{|y|} e^{-s/3} + \frac{e^{s/2}}{|y|} \right)  \\
& \leq Ce^{-s/3},
\end{split}
\end{equation*}
where we have used $|y|\geq e^s$ for the last inequality. Combining all together, we obtain \eqref{tP_F_fin} for all $\ve<\ve_0$, where $\veps_0$ is sufficiently small.  

Let  $\psi(y,s)$  be the solution to  $ {\partial_s}\psi(y,s)=U^Z(\psi,s)$ with $\psi(y,s_0)=y$. Recalling the definition of $U^Z$ in \eqref{UW}, and using   \eqref{temp_2} and \eqref{temp_P_1}, we deduce   that as long as $\psi(y,s) \geq e^s$,
		\begin{equation}\label{uz-rest}
			\begin{split} 
				\partial_s \psi(y,s) = U^Z(\psi(y,s),s)
				& = \frac{3}{2}\psi   + \frac{W(\psi,s)}{1-\dot{\tau}} +\sigma(\psi,s) - \frac{4\sqrt{K}e^{s/2}}{1-\dot{\tau}} \\
				&\geq \left(\frac{3}{2}- Ce^{-2s/3}\right)\psi - C (1+\psi)e^{-s/3} - C\sqrt{K}e^{s/2}
				\\
				& > e^{s}.
			\end{split}
		\end{equation} 
Using \eqref{uz-rest}, we claim that  if 	 $\psi(y,s_0)\geq e^{s_0}$,  then $\psi(y,s) \geq e^s$ for all $s \geq s_0$. To see this, it is enough to show that $\partial_s\psi(y,s) \geq e^s$ for $s  \geq s_0$. If not, then by the \textit{strict} inequality \eqref{uz-rest} evaluated at $s=s_0$ and continuity, there exist $s_\ast$ and $s_\ast'$  such that
\begin{equation}\label{A3}
\partial_s\psi(y,s) \geq e^s \quad \text{for } s \in [s_0,s_\ast],  \quad \partial_s\psi(y,s) < e^s \quad \text{for } s \in (s_\ast,s_\ast'].
\end{equation}
 On the other hand, by integrating $\partial_s\psi(y,s)$ from $s_0$ to $s_\ast$, we have
\[
\psi(y,s_\ast) \geq \psi(y,s_0) + e^{s_\ast} - e^{s_0} \geq  e^{s_\ast},
\]
which implies $\partial_s \psi(y,s_\ast)> e^{s_\ast}$ by \eqref{uz-rest}. This  contradicts to \eqref{A3} by continuity of $\partial_s \psi(y,s)$ in $s$. Similarly,  we see that $U^Z(\psi(y,s),s) <-e^s$ 		as long as $\psi(y,s) \leq -e^s$, and using the above argument,   we deduce that  $|\psi(y,s)|\geq e^s$ for all $s\geq s_0$ if $|\psi(y,s_0)|\geq e^{s_0}$. 

Now we integrate \eqref{weightP} along $\psi$  and  use \eqref{tP_F_fin} to obtain that  for $|y|\geq e^{s_0}$,
		\begin{equation}\label{tildeP_Meq}
			\begin{split}
				|\wt{P}(\psi(y,s),s)| 
				&  \leq |\wt{P}(\psi(y,s_0),s_0)|+C\int^s_{s_0}e^{-s'/3}\,ds' \leq \frac{1}{28}+3Ce^{-s_0/3} \leq \frac{3}{56}
			\end{split}
		\end{equation}  
		for all  $\veps\le \ve_0$ sufficiently small. Here, we have used  \eqref{tildeP_init}. 

It is remained to consider  the case $|y|< e^{s_0}$. By continuity, we consider two cases: (i) $|\psi(y,s)|< e^s$ for all $s\geq s_0$, or  (ii) there is $s_{\ast\ast}$ such that  $|\psi(y,s)|< e^s$ for $s\in[s_0,s_{\ast\ast})$  and $|\psi(y,s)| = e^s$ at $s=s_{\ast\ast}$. For the  case (i),  since $\left|e^{P}-1\right|\leq 2/5$ by \eqref{wz3M1}, we have
		\begin{equation}\label{Psmall_1}
		\begin{split}
			|\wt{P}(\psi(y,s),s)|
			&  =|\psi(y,s)|^{2/3}\left|e^{P}-1\right|e^{-s} \\
			& \leq \frac{2}{5}e^{-s/3} \leq \frac{1}{28}
			\end{split}
		\end{equation}
for all  $\veps\le \ve_0$ sufficiently small. For the case (ii),  we can check that \eqref{Psmall_1} holds for $s\in[s_0,s_{\ast\ast})$ by the same analysis for the case (i). Applying the same argument to the case $|y=\psi(y,s_0)| \geq e^{s_0}$,  we see that $|\psi(y,s)| \geq e^s$ for $s \geq s_{\ast\ast}$. Now, using  \eqref{Psmall_1} at $s=s_{\ast\ast}$, one can obtain the same inequality as \eqref{tildeP_Meq}. This finishes the proof of Lemma~\ref{Ptilde_lem}.
\end{proof}

		\subsection{Weighted estimates of $W_y$}
Our main goal is to prove Lemma~\ref{mainprop_1}. We first prove the following auxiliary lemma.

	\begin{lemma}\label{mainprop_2}
		If the same assumptions as in Proposition~\ref{decay_lem} hold, then we have 
		\begin{equation}\label{claim2_nudec}
			\limsup_{|y|\rightarrow \infty}\left| (y^{2/3}+8) ( W_y(y,s) -\overline{W}'(y)) \right| \leq  \frac{3}{4}.
		\end{equation}   
	\end{lemma}
	\begin{proof}
	Since
	\begin{equation}
	\begin{split}
	& \limsup_{|y|\rightarrow \infty}|(y^{2/3}+8) ( W_y(y,s) -\overline{W}'(y))|
	\\
 & \leq \limsup_{|y|\rightarrow \infty}|(y^{2/3}+8)W_y(y,s)| + \limsup_{|y|\rightarrow \infty} |(y^{2/3}+8)\overline{W}'(y)| \\
	& = \limsup_{|y|\rightarrow \infty}|(y^{2/3}+8)W_y(y,s)| + \frac{1}{3},
	\end{split}
	\end{equation}
	 it is enough to show that 
	 \begin{equation}\label{3.00}
	 \limsup_{|y|\rightarrow \infty}\left|  (y^{2/3}+8)  W_y(y,s) \right| \leq \frac{5}{12}.
	 \end{equation}
	 Let $\mu:=(y^{2/3}+8)W_y$. From \eqref{EP2_2D1}, we obtain  
		\begin{equation}\label{decay_mu}
			\partial_s \mu+D^\mu \mu+U^W \mu_y=F^\mu,
		\end{equation}
		where $U^W$ is defined in \eqref{UW}, and 
		\begin{subequations}
			\begin{align*}
				D^\mu(y,s)&:=1+\wt{W}_y+\overline{W}'-\frac{2y^{2/3}}{3(y^{2/3}+8)}\left(\frac{3}{2}+\frac{\wt{W}+\overline{W}}{y}\right),
				\\
				F^\mu(y,s)&:=-\frac{2e^{s}\Phi_{yy}(y^{2/3}+8)}{1-\dot{\tau}}+\left(-\frac{\dot{\tau}W_y}{1-\dot{\tau}}-\frac{e^{s/2}Z_y}{1-\dot{\tau}}+\frac{2}{3}\frac{y^{2/3}}{y^{2/3}+8}\left(\frac{\dot{\tau}}{1-\dot{\tau}}\frac{W}{y}+\frac{\sigma}{y} \right)\right)\mu.
			\end{align*}
		\end{subequations}
		Using \eqref{Utildey_M} and the fundamental theorem of calculus with $\wt{W}(0,s)=0$, we obtain that 
		\begin{equation}\label{Dmu}
			\begin{split}
				D^\mu(y,s)&\geq 1-\frac{1}{y^{2/3}+8}+\overline{W}'-\frac{2y^{2/3}}{3(y^{2/3}+8)}\left(\frac{3}{2}+\frac{\overline W}{y}+\frac{1}{y}\int_{0}^{y}{\frac{dy'}{{y'}^{2/3}+8}}\right).
			\end{split}
		\end{equation}
From \eqref{0605_m_7'}, the lower bound of $D^\mu$ in \eqref{Dmu} is strictly positive for $|y| \geq 3$.   Also, we have
		\begin{equation}\label{mu9}
			|\mu(y,s)|\leq |(y^{2/3}+8)\wt{W}_y|+|(y^{2/3}+8)\overline{W}'| \leq   C
		\end{equation}
		from \eqref{Utildey_M} and \eqref{y-w-y}.  Using \eqref{Boot_3}, \eqref{Uy1}, \eqref{Wbd1}, \eqref{temp_2}, \eqref{dottau} and \eqref{Phiyy_decay},  it is easy to check that 
\begin{equation}\label{Fmu}
\begin{split}
\|F^\mu(y,s)\|_{L^{\infty}(|y|\geq 3)} 
& \leq C (e^{-s} +  e^{-s/3}\|\mu(\cdot, s)\|_{L^\infty(\mathbb{R}) })   \leq C e^{-s/3},
\end{split}
\end{equation}		
where we have used \eqref{mu9} for the last inequality. Applying Lemma~\ref{rmk2}, \eqref{UW_far}, \eqref{Dmu} and \eqref{Fmu}, we obtain 
		\begin{equation}\label{3.108}
			\limsup_{|y|\rightarrow \infty}|\mu(y,s)|\leq \limsup_{|y|\rightarrow \infty}|\mu(y,s_0)|+ C\veps^{1/3}.
		\end{equation}
On the other hand, from \eqref{wx-exp-inf} and  \eqref{4.3a}, we see that 
		\begin{equation}\label{3.1081}
  \begin{split} 
			\limsup_{|y|\rightarrow \infty}|\mu(y,s_0)|
   &\leq \limsup_{|y|\rightarrow \infty}|(y^{2/3}+8)\wt{W}_y(y,s_0)|+\lim_{|y|\rightarrow \infty}|(y^{2/3}+8)\overline{W}'(y)|
   \\
   & \leq \frac{1}{24}+\frac{1}{3}=\frac{9}{24}.
   \end{split} 
		\end{equation} 
Hence, combining \eqref{3.108} and \eqref{3.1081}, we conclude that   \eqref{3.00} holds for all $\ve<\ve_0$ with $\veps_0>0$ sufficiently small.
	\end{proof}

We finally close the bootstrap assumption  \eqref{Utildey_M} in the following lemma. 
\begin{lemma}\label{mainprop_1}
	If the same assumptions as in Proposition~\ref{decay_lem} hold, then we have 
\begin{equation*}
| (y^{2/3}+8)  (W_y(y,s) -\overline{W}'(y) ) | \le \frac{24}{25},
\end{equation*}
for $s\in[s_0,\sigma_1]$.
\end{lemma}

\begin{proof}
Let  $\nu (y,s):=(y^{2/3}+8)\wt{W}_y(y,s)$ where $\wt{W}:=W-\overline{W}$. From \eqref{Eq_diff}, we find the equation for $\nu$:
\begin{equation}\label{nu_eq1}
	\partial_s \nu + U^W(y,s)\partial_y \nu + D^\nu (y,s)\nu= F^\nu_1(y,s) +F^\nu_2(y,s) +\int_{\mathbb{R}}{\nu(y',s) K^\nu(y,s;y') \,dy'},
\end{equation}
where $U^W$ is defined in \eqref{UW}, 
\begin{subequations}
	\begin{align*}
		D^\nu &:=1+\wt{W_y}+2\overline{W}'-\frac{2y^{2/3}}{3(y^{2/3}+8)}\left(\frac{3}{2}+\frac{\wt{W}+\overline{W}}{y}\right),
		\\
		F^\nu_1&:= -\left(\frac{\dot{\tau}\overline{W}}{1-\dot{\tau}}+ \sigma\right)(y^{2/3}+8)\overline{W}'' -\left(\frac{\dot{\tau}\overline{W}'}{1-\dot{\tau}}+\frac{e^{s/2}Z_y}{1-\dot{\tau}}\right)(y^{2/3}+8)\overline{W}', \\
		F^\nu_2&:= -\frac{2e^s}{1-\dot{\tau}}(y^{2/3}+8)\Phi_{yy} -\left(\frac{\dot{\tau}(\wt{W}_y+2\overline{W}')}{1-\dot{\tau}}+\frac{e^{s/2}Z_y}{1-\dot{\tau}}-\frac{2}{3}\frac{y^{2/3}}{y^{2/3}+8}\left(\frac{\dot{\tau}W}{(1-\dot{\tau})y}+\frac{\sigma}{y}  \right)\right)\nu,
	\end{align*}
\end{subequations}
 and 
 \begin{equation*}
 	K^\nu(y,s;y'):= -\frac{1}{1-\dot{\tau}}(y^{2/3}+8)\overline{W}''(y)\mathbb{I}_{[0,y]}(y')\frac{1}{(y')^{2/3}+8}.
 \end{equation*} 
%
From \eqref{Utildey_M}, we have 
\begin{equation}\label{nu_D}
	\begin{split}
		D^\nu(y,s)&=1+\wt{W}_y+2\overline{W}'-\frac{2y^{2/3}}{3(y^{2/3}+8)}\left(\frac{3}{2}+\frac{\wt{W}+\overline{W}}{y}\right)
		\\
		&\geq 1-\frac{1}{y^{2/3}+8}+2\overline{W}'-\frac{2y^{2/3}}{3(y^{2/3}+8)}\left(\frac{3}{2}+\frac{\overline W}{y}+\frac{1}{y}\int_{0}^{y}{\frac{dy'}{{y'}^{2/3}+8}}\right) 
		\\
  & =:   D_{-}(y),
	\end{split}
\end{equation}
and by \eqref{tau},   
\begin{equation}\label{nu_K}
	\int_{\mathbb{R}}{|K^\nu(y,s;y')|\,dy'}\leq(1+4\veps)(y^{2/3}+8)|\overline{W}''|\int^{|y|}_{0}{\frac{dy'}{{y'}^{2/3}+8}} =: K_{+}(y).
\end{equation}
Thanks to \eqref{0605_m_7}, we have 
 \begin{equation}\label{A2}
 K_{+}(y)\leq D_{-}(y),  \qquad  |y|\geq 3
 \end{equation} 
for all $\ve<\ve_0$  sufficiently small.
Combining \eqref{nu_D}--\eqref{A2}, we get
 \begin{equation}\label{D-K}  
 D^\nu (y,s)\geq   \int_{\mathbb{R}}|K^\nu(y,s; y')|\,dy', \qquad |y|\ge 3.
 \end{equation}
 Next, we estimate $F^\nu_1$ and $F^\nu_2$. 
 It is straightforward to check that 
\begin{equation}\label{F1}
	\|F^\nu_1(\cdot,s)\|_{L^{\infty}(|y|\geq 3)}\leq Ce^{-s/3}.
\end{equation}
Here we have by \eqref{dottau}, \eqref{temp_2} and \eqref{Wyybar_bdd} that 
\begin{equation*}
\begin{split}
  &\left| \left(\frac{\dot{\tau}\overline{W}}{1-\dot{\tau}}+ \sigma\right)(y^{2/3}+8)\overline{W}'' \right| \le Ce^{-s/3},
  \end{split}
\end{equation*}
and by \eqref{Boot_3}, \eqref{dottau} and \eqref{y-w-y} that
\begin{equation*}
\begin{split}
  &\left| \left(\frac{\dot{\tau}\overline{W}'}{1-\dot{\tau}}+\frac{e^{s/2}Z_y}{1-\dot{\tau}}\right)(y^{2/3}+8)\overline{W}' \right| \le C e^{-s}.   
\end{split}
\end{equation*}
Similar to \eqref{Fmu}, 
we have 
\begin{equation}\label{F2}
\begin{split}
\|F^\nu_2(\cdot,s)\|_{L^{\infty}(|y|\geq 3)} 
&  \leq C ( e^{-s} + e^{-s/3}\|\nu(\cdot, s) \|_{L^\infty})   \leq C e^{-s/3},
\end{split}
\end{equation}
where we have 
used \eqref{Utildey_M} for  the last inequality.   
To see this,  each term in $F_2^\nu$ can be examined as follows. Thanks to \eqref{Phiyy_decay}, we have
\[ \left| \frac{2 e^s }{1-\dot\tau} (y^{2/3} +8) \Phi_{yy} \right| \le C e^{-s}, \]
and 
by  \eqref{Boot_1}, \eqref{Wbd1}, \eqref{temp_2}, \eqref{dottau} and $|\overline{W}'|\leq 1$ from \eqref{Wbar_2}, 
we have for $|y|\ge3$, 
\begin{equation*}
\begin{split}
\left| \frac{\dot{\tau}(\wt{W}_y+2\overline{W}')}{1-\dot{\tau}}+\frac{e^{s/2}Z_y}{1-\dot{\tau}}-\frac{2}{3}\frac{y^{2/3}}{y^{2/3}+8}\left(\frac{\dot{\tau}W}{(1-\dot{\tau})y}+\frac{\sigma}{y}  \right)\right|
\le C e^{-s/3}. 
\end{split} 
\end{equation*}
To finish the proof, we claim that
\begin{equation}\label{claim_reg}
	\|\nu(\cdot,s)\|_{L^{\infty}} < \frac{24}{25}, \qquad s\in[s_0, \sigma_1].
\end{equation}
Suppose to the contrary that \eqref{claim_reg} fails. 
Since $\nu \in C([s_0, \sigma_1]; L^\infty(\mathbb{R}))$,  $s_2:= \min \{ s \in [s_0, \sigma_1]: \|\nu(\cdot,s)\|_{L^{\infty}} = 24/25 \}$ is well-defined, and there exists $s_1\in (s_0, s_2)$  such that
for all $s\in(s_1,s_2)$,
\begin{equation}\label{34}
\frac{23}{25} = \|\nu(\cdot,s_1)\|_{L^{\infty}}  \le \|\nu(\cdot,s)\|_{L^{\infty}} < \frac{24}{25} = \|\nu(\cdot,s_2)\|_{L^{\infty}}. 
\end{equation}
Then, for each $s\in[s_1,s_2]$, thanks to the smoothness of $\nu$ and the decay property \eqref{claim2_nudec},   we can choose a point $y_*(s)$ such that  
  \begin{equation}\label{claim2_nudec0}
  \|\nu(\cdot, s) \|_{L^{\infty}} = |\nu(y_*(s), s)|.
  \end{equation} 
In view of \eqref{34},  we see that $\nu(y_*(s), s) \neq 0$ for all $s \in [s_1,s_2]$. Moreover, $ \partial_y \nu(y_\ast(s),s) =0 $ due to \eqref{claim2_nudec0}.

Thanks to \eqref{Boot_1}, we see that
\begin{equation}\label{A1}
\|\nu(\cdot,s)\|_{L^{\infty}(|y|\leq 3)} \leq \frac{3^2(3^{2/3} + 8)}{10(1+3^2)}  < \frac{23}{25}.
\end{equation}   
Combining \eqref{34} and \eqref{A1}, it follows that  $|y_*(s)|\geq 3.$ Hence, by \eqref{nu_D}--\eqref{A2} and \eqref{claim2_nudec0}, if $\nu(y_*(s),s)> 0$, then  we have
\begin{equation}
	\begin{split}\label{nu_DK}
		D^\nu(y_*(s),s)\nu(y_*(s),s)&\geq D_{-}(y_*(s))\|\nu(\cdot,s)\|_{L^{\infty}}
		\\
		&\geq K_{+}(y_*(s))\|\nu(\cdot,s)\|_{L^{\infty}}\geq \left|\int_{\mathbb{R}}K^\nu(y_*(s),s;y')\nu(y',s)\,dy'\right|.
	\end{split}
\end{equation}
Similarly, if  $\nu(y_*(s),s) < 0$, we have 
\begin{equation}
	\begin{split}\label{nu_DK2}
		D^\nu(y_*(s),s)\nu(y_*(s),s)  \le - \left|\int_{\mathbb{R}}K^\nu(y_*(s),s;y')\nu(y',s)\,dy'\right|.
	\end{split}
\end{equation}

We evaluate  \eqref{nu_eq1} at $y=y_\ast(s)$. Then, using   \eqref{F1}, \eqref{F2}, \eqref{nu_DK}, \eqref{nu_DK2}  and the fact that $\partial_y \nu(y_*(s), s) =0$,  we deduce that  
\begin{equation}\label{nu_temp}
	\begin{split}
  \partial_s \nu(y, s)|_{y=y_*(s)} & \leq  Ce^{-s/3}+\int_{\mathbb{R}} \nu(y',s)K^\nu(y_*(s),s;y')\,dy'-D^\nu(y_*(s),s)\nu(y_*(s),s)
  \\
  &
		\leq Ce^{-s/3}
	\end{split}
\end{equation}
if   $\nu(y_*(s),s)> 0$, 
and similarly that  
\begin{equation}\label{nu_temp-}
	\begin{split}
		\partial_s \nu(y, s)|_{y=y_*(s)} & \ge - C e^{-s/3}
	\end{split}
\end{equation} 
if $\nu(y_*(s),s)< 0$. 

For fixed $s$, by the definition of $y_*$, it holds  that \[ \|\nu(\cdot,s-h)\|_{L^\infty} \geq |\nu(y_*(s) - h U^W(y_*(s),s), s - h)| 
\] 
for any small $h>0$. 
Then, it is straightforward to check that 
\begin{equation}\label{AP_R1} 
\begin{split}
\lim_{h \to 0^+} \frac{\|\nu(\cdot,s-h)\|_{L^\infty} - \|\nu(\cdot,s)\|_{L^\infty}}{-h} 
& \leq (\partial_s + U^W(y_*(s),s)\partial_y)|\nu|(y,s)|_{y=y_*(s)}, 
\end{split}
\end{equation} 
provided that the limit on the LHS of \eqref{AP_R1} exists. Note that by Rademacher's theorem, 
$ \| \nu(\cdot, s)\|_{L^{\infty}}$, being Lipschitz continuous in $s$, is differentiable at almost all $s\in[s_1, s_2]$. Thus, since  $\partial_y \nu(y_*(s), s) =0$, we deduce from \eqref{AP_R1} that 
\begin{equation*} \label{L-thm}
\begin{split}
\frac{d}{ds} \| \nu(\cdot, s)\|_{L^{\infty}} 
& \leq  \left\{ \begin{array}{l l}
\partial_s \nu(y, s)|_{y=y_*(s)} & \text{if } \nu(y_*(s),s)>0, \\
-\partial_s \nu(y, s)|_{y=y_*(s)} & \text{if } \nu(y_*(s),s)<0
\end{array}
\right.  
\end{split}
\end{equation*} 
for almost all $s\in[s_1, s_2]$.  Combining  with  \eqref{nu_temp} and \eqref{nu_temp-}, we have 
\begin{equation*}
\begin{split} 
 \| \nu(\cdot, s_2) \|_{L^{\infty}} 
&=  \| \nu(\cdot, s_1) \|_{L^{\infty}} +  \int_{s_1}^{s_2} \frac{d}{ds} \| \nu(\cdot, s)\|_{L^{\infty}} \; ds  \\ 
& \le \| \nu(\cdot, s_1) \|_{L^{\infty}} + \int_{s_1}^{s_2} C e^{-s/3} \; ds \\
& \le  \| \nu(\cdot, s_1) \|_{L^{\infty}} +  Ce^{-s_1/3}\\
&  \leq  \| \nu(\cdot, s_1) \|_{L^{\infty}} + C \veps^{1/3}.
\end{split} 
\end{equation*}
This together with \eqref{34} leads to a contradiction  for sufficiently small $\ve$,  which proves \eqref{claim_reg}, in turn Lemma~\ref{mainprop_1}.   
\end{proof}

\section{Blow-up for the Isentropic Euler-Poisson system}\label{Isen}
In this section, we present a similar blow-up result as that of Theorem~\ref{mainthm1} for
  the isentropic Euler-Poisson system: 
		\begin{subequations}\label{EP-Isen}
		\begin{align}
			& \rho_t +  (\rho u)_x = 0,\label{EP-Isen1}\\ 
			&\rho( u_t  + u u_x ) + P_\gamma(\rho)_x 
   = -  \phi_x, \label{EP-Isen2}\\
			& - \phi_{xx} = \rho - e^\phi,\label{EP-Isen3}
		\end{align}
	\end{subequations} 
	where $\rho>0$, $u$ and $\phi$ represent the ion density, the fluid velocity for ions, and  the electric potential, respectively, and the pressure $P_\gamma(\rho)$ is given by $P_\gamma(\rho) := {\rho^\gamma}/{\gamma}$, $\gamma>1$.  
 %
 %
 %
 %

 We  introduce the Riemann functions associated with \eqref{EP-Isen} as 
	\begin{equation}\label{RI-isen}
		w=u+\frac{1}{\alpha}\rho^{\alpha},\quad z=u-\frac{1}{\alpha}\rho^{\alpha},
	\end{equation} 
 	corresponding to the eigenvalues 
	\begin{equation*} 
\lambda_{+}=u+\rho^{\alpha}=\frac{1+\alpha}{2}w+\frac{1-\alpha}{2}z, \quad \lambda_{-}=u-\rho^\alpha=\frac{1-\alpha}{2}w+\frac{1+\alpha}{2}z, 
	\end{equation*} 
	where $\alpha := {(\gamma-1)}/{2}.$
	In terms of   $w$ and $z$, \eqref{EP-Isen} can be rewritten as 	\begin{subequations}\label{EP-Isen'}
		\begin{align}
			&w_t+\left(w+\frac{1-\alpha}{1+\alpha}z\right)w_x=-\frac{2}{1+\alpha}\phi_x,\label{EP-Isen'1}\\
			&z_t+\left(z+\frac{1-\alpha}{1+\alpha}w\right)z_x=-\frac{2}{1+\alpha}\phi_x,\label{EP-Isen'2}\\
			&-\phi_{xx}=\left(\frac{\alpha}{2}(w-z)\right)^{1/\alpha}-e^{\phi}.\label{EP-Isen'3}
		\end{align} 
	\end{subequations} 
 We give a list of the initial conditions for the isentropic case, which are similar to those for the isothermal case, i.e.,  \eqref{init_w_3}--\eqref{tildeP_init}  but  the last one. 
	\begin{subequations}\label{init_gen_0}
		\begin{align}
			& w_0(0)=\kappa_0 > 1, \quad \partial_x w_0 (0)  =-\veps^{-1}, \quad \partial_x^2 w_0 (0)=0, \quad \partial_x^3w_0(0)=6\veps^{-4},\label{init_gen_1}
			\\
			&\| \partial_x w_0  \|_{L^{\infty}}\leq \veps^{-1},\quad \|\partial_x^2w_0 \|_{L^{\infty}}\leq \veps^{-5/2},\quad \|\partial_x^3w_0 \|_{L^{\infty}}\leq 7\veps^{-4},\label{init_gen_2}
   \\
   &  \|\partial_x^4  w_0  \|_{L^{\infty}}\leq \veps^{-11/2}, \nonumber 
			\\
			&\|z_0\|_{C^4}\leq 1/2,\label{init_gen_3}
			\\
			&	\sup_{x\in \mathbb{R} }\left((x^{2/3}+8\veps)\left|\rho_0(x)-1\right|\right) \label{init_gen_4}
   \\
   & \qquad = \sup_{x\in 	\mathbb{R} }\left((x^{2/3}+8\veps)\left|\left(\frac{\alpha}{2}(w_0(x)-z_0(x))\right)^{1/\alpha}-1\right|\right)\leq 1/16, \nonumber
			\\
			&		\left|\varepsilon(\partial_x w_0)\left( x\right)-\overline{W}'\left( \frac{x}{\varepsilon^{3/2}}\right)\right|\leq \min\left\{\frac{(\frac{x}{\varepsilon^{3/2 }})^2}{40(1+(\frac{x}{\varepsilon^{3/2}})^2)}, \frac{1}{24(8+(\frac{x}{\veps^{3/2 }})^{2/3 })} \right\},\label{init_gen_5} 
			\\
			& \inf_{ x\in \mathbb{R}} ( w_0(x)-z_0(x) ) =:P_- >0. \label{init_gen_6}
		\end{align}
	\end{subequations}
%
%
%
Then we state our result. 
	\begin{theorem}\label{main-isen}
There is a constant $\veps_0=\veps_0(\alpha, \|(\rho_0-1,u_0)\|_{L^2},\inf_{x\in\mathbb{R}}\rho_0, \sup_{x\in\mathbb{R}}\rho_0,\kappa_0)>0$ such that for each $\veps\in(0,\veps_0)$, if the initial data   $(\rho_0,u_0)$ satisfies $\rho_0>0$, $(\rho_0-1,u_0)\in H^k(\mathbb{R})$, where $k\geq 5$, and  satisfies \eqref{init_gen_0}, then there is a unique smooth solution $(\rho,u)\in C\left([-\veps,T_\ast); C^4(\mathbb{R})\right)$ to \eqref{EP-Isen},  where the maximal existence time $T_\ast>-\ve$ is finite and  $T_*=O(\veps)$. Furthermore, it holds that 
	\begin{enumerate}[(i)]
	\item
	$ \sup_{t<T_*} \left[ (\rho, u) (\cdot, t) \right]_{C^\beta}<\infty \; \text{ for } \beta \leq 1/3$;
	\item 
	$\lim_{t\nearrow T_*} \left[ \rho (\cdot, t) \right]_{C^\beta} =\infty$ and $\lim_{t\nearrow T_*} \left[ u (\cdot, t) \right]_{C^\beta} = \infty$ for $\beta>1/3$;
	\item   for $\beta>1/3$, 
	the temporal blow-up rate is obtained as
	\ \begin{equation*} \left[ \rho (\cdot, t) \right]_{C^\beta},  \left[ u(\cdot, t) \right]_{C^\beta} \sim (T_*-t)^{-\frac{3\beta-1}{2}} \; \text{ for } \beta > 1/3
	\end{equation*} 
		for all $t$ sufficiently close to $T_*$; 
	 \item $\inf_{x\in \mathbb{R}, t<T_*} \rho(x,t) \ge \rho_*$ and $\sup_{t<T_\ast} \| ( \rho, u)(\cdot, t) \|_{L^\infty} \le M_*$ for some $\rho_*, M_*>0$. 
	\end{enumerate}
	\end{theorem}
	The proof of Theorem~\ref{main-isen} is quite parallel to that  of Theorem~\ref{mainthm1}, but there are several noteworthy  differences in the analysis. 
	In what follows, we outline the proof with   emphasis  on such  differences. 
	%
	%
	%
	%
	%
	%
	%
	%
	%
	%

	 Similarly as \eqref{modulation-new}, 
  we define three dynamic modulation functions $\tau, \kappa, \xi:[-\ve, \infty) \to \mathbb{R}$ satisfying  {
	 \begin{subequations}\label{modulation_gen}
	 	\begin{align}
	 		\dot \tau & = \frac{1-\alpha}{1+\alpha}( \tau(t) - t) \partial_x z(\xi(t) , t) - \frac{2}{1+\alpha} ( \tau(t) - t)^2 \partial_x^2 \phi (\xi(t) , t), \label{modulation_gen_1}
	 		\\
	 		\dot{\kappa}&= \frac{-2( \tau(t) - t)^{-1} \partial_x^3 \phi  (\xi(t) , t)+ (1-\alpha) ( \tau(t) - t)^{-2} \partial_x^2 z (\xi(t) , t) }{ (1+\alpha) \partial_x^3 w (\xi(t) , t)} \label{modulation_gen_2}
    - \frac{2 \partial_x \phi (\xi(t) , t) }{1+\alpha},   
	 		\\
	 		\dot{\xi}&= \frac{1-\alpha}{1+\alpha} z(\xi(t) , t)  + \frac{ 2 \partial_x^3 \phi(\xi(t) , t) - (1-\alpha) (\tau(t) - t)^{-1} \partial_x^2 z(\xi(t) , t) }{ (1+\alpha) \partial_x^3 w(\xi(t) , t) } + \kappa(t),\label{modulation_gen_3}
	 	\end{align}
	 \end{subequations}
 }
with the same initial values as  \eqref{Modul_init}. 
	Defining new functions $W, Z$ and $\Phi$ in self-similar variables  (see \eqref{change_var}--\eqref{WZPhi}),
	  we obtain the equations  for $W,Z$ and $\Phi$ as  
\begin{subequations}\label{EP2_gen}
		\begin{align}
			& \partial_s W-\frac{1}{2}W+U^W_\alpha W_y=-\frac{\dot{\kappa}e^{-s/2}}{1-\dot{\tau}}-\frac{1}{1-\dot{\tau}}\frac{2e^s\Phi_y }{1+\alpha},\label{EP2_1_gen} \\ 
			& \partial_s Z+U^Z_\alpha Z_y=-\frac{1}{1-\dot{\tau}}\frac{2e^{s/2}\Phi_y}{1+\alpha}, \label{EP2_2_gen} \\
			& -\Phi_{yy}e^{3s} = \left(\frac{\alpha}{2}\left(e^{-s/2}W+\kappa-Z\right)\right)^{1/\alpha} - e^{\Phi}, \label{EP2_3_gen} 
		\end{align}
	\end{subequations} 
	where 
	\begin{equation}\label{UW_gen}
		\begin{split}
			&U^W_\alpha:= \frac{e^{s/2}}{1-\dot{\tau}}\left(\kappa-\dot{\xi}+\frac{1-\alpha}{1+\alpha}Z\right) +\frac{3}{2}y+\frac{W}{1-\dot{\tau}}, 
   \\
   & 
   U^Z_\alpha:= \frac{e^{s/2}}{1-\dot{\tau}}\left(\frac{1-\alpha}{1+\alpha}\kappa-\dot{\xi}\right)+\frac{e^{s/2}Z}{1-\dot{\tau}} +\frac{3}{2}y+\frac{1-\alpha}{1+\alpha}\frac{W}{1-\dot{\tau}}.
			\end{split}
	\end{equation}
%
%
%
%
%
%
	
	Now, similarly as in the isothermal case, we impose  the bootstrap assumptions as follows:
	\begin{subequations}\label{boots_isen}
		\begin{align}
			&|W_y(y,s)-\overline W'(y)| \leq \frac{y^2}{10(1+y^2)},\label{Uy1_gen}
			\\
			&|W_y(y,s)-\overline{W}'(y)| \le \frac{1}{y^{2/3}+8}, \label{W-y-decay-isen}
			\\
			&|W_{yy}(y,s)| \leq \frac{15|y|}{(1+y^2)^{1/2}},\label{EP2_1D3-isen}
			\\
			&|\partial_y ^3 W(0,s)-6| \leq 1,\label{Wy30_isen} 
			\\
			&\|\partial_y ^3 W(\cdot, s) \|_{L^{\infty}} \leq  M^{5/6} , \label{Wy3_isen}
			\\
			&\|\partial_y ^4 W(\cdot, s) \|_{L^{\infty}} \leq M, 
			\\
			&\|Z_y\|_{L^{\infty}}\leq e^{-(1/2+\delta) s}, \label{difz32-isen} \\
			& | \dot{\tau}| \leq 2e^{-s}, \label{dottau_isen}
		\end{align}
	\end{subequations}
	where $M>0$ is a sufficiently large constant and $\delta>0$ is a   number such that $\delta\in(0,  (1- |1-\alpha|(1+\alpha )^{-1})/4)$. We notice that the bootstrap assumption \eqref{difz32-isen} for $Z_y$ is slightly different from the isothermal case (see \eqref{Boot_3}). This is due to the fact that  the isentropic model \eqref{EP-Isen} has a different form of the Riemann functions from those of the isothermal case. 
	
	In addition, 
	we impose the decaying condition for $\rho-1=     ( {\alpha}  (e^{-s/2}W+\kappa-Z ) /2 )^{1/\alpha}-1 $ as 
	\begin{equation}\label{isen_rho-1}
		e^{-s}(y^{2/3}+8)    \left|     \left(\frac{\alpha}{2}\left(e^{-s/2}W+\kappa-Z\right)\right)^{1/\alpha}-1\right|      \leq \frac{1}{8}.
	\end{equation}
	We remark that under the initial conditions \eqref{init_gen_0}, one can show that a local-in-time solutions $(\rho, u, \phi)$ exist and the corresponding solutions $W, Z$ and $\Phi$  satisfy    the bootstrap assumptions  \eqref{boots_isen}--\eqref{isen_rho-1} at least for a local time interval. 
	
	%
	%
	%
	%
	%
	%
	%
	%
	%
	%
	%
	%
	%
	%

	In the course of our analysis,  $L^\infty$ norms of $\phi$ and $\phi_x$, which are similar to Lemma~\ref{phi_lem}, will be crucially used. 
	We introduce the conserved energy for \eqref{EP-Isen}  as 
	\begin{equation*}
		H_\gamma (t):=\int_{\mathbb{R}} 
		\frac{1}{2}\rho u^2
		+\frac{1}{\gamma-1} \mathcal{P}_\gamma(\rho)
		+\frac{1}{2}|\partial_x\phi|^2+(\phi-1)e^\phi+1 \,dx 
	\end{equation*}
	where $\mathcal{P}_\gamma(\rho) := {\gamma}^{-1}(\rho^\gamma -1)-(\rho-1)$. Similarly as \eqref{H-L2} for the isentropic case, one can show that
	\begin{equation}\label{H-L2-gen}
		|H_{\gamma}(t)|\leq C\|(\rho-1, u)\|_{L^2}^2\leq C\|(w,z)\|_{L^2}^2,
	\end{equation}
	where $C>0$ is a constant depending only on $\inf_{x\in\mathbb{R}}\rho$ and $\sup_{x\in\mathbb{R}}\rho$. 
%

\subsection{Estimates of $\Phi$}
The following lemma corresponds to Lemma~\ref{phi_lem}. We remark that, contrast to the isothermal case, the condition concerning the amplitude of $\rho$ is not required for the uniform bound of $\phi_x$ for the isentropic case.
	\begin{lemma}
		\label{phi_lem_gen}
		Let $(\rho, u)$ be a smooth solution to \eqref{EP-Isen}   satisfying $(\rho, u) \to (1,0)$ as $|x| \to \infty$. Then, it holds that 
		\begin{subequations}
			\begin{align}
				&\|\phi\|_{L^{\infty}}=\|\Phi\|_{L^{\infty}}\leq M_1
				,\label{Phi_0_M_gen}\\
				&\|\phi_x\|_{L^{\infty}}=e^{3s/2}\|\Phi_y\|_{L^{\infty}}<M_2 \label{Phi_1_M_gen}
			\end{align}
		\end{subequations}
		for some 
		$M_1>0$ and $M_2>0$ depending only on  $\|(\rho_0-1, u_0)\|_{L^2(\mathbb{R})}$, $\inf_{x\in\mathbb{R}}\rho_0$ and $\sup_{x\in\mathbb{R}}\rho_0$.
		Here,  $M_1$ and $M_2$  tend to zero as $\|(w_0 , z_0)\|_{L^2(\mathbb{R})}$ tends to zero.
	\end{lemma}
	\begin{proof}
	We loosely follow the proofs of Lemma 2.1 and Lemma 2.2 in \cite{BCK}. Similarly as in the proof of Lemma 2.1 in \cite{BCK} (see also the proof of \eqref{Phi_0_M-1}), one can obtain the inequality \eqref{V-phi} with $H$ replaced by $H_\gamma$. Then using  \eqref{H-L2-gen}, we have  \eqref{Phi_0_M_gen}.

		Next, we prove \eqref{Phi_1_M_gen}. 
		Multiplying \eqref{EP-Isen3} by $-\phi_x,$ and then integrating in $x$,
		\begin{equation}\label{phix2}
			\begin{split}
				\frac{\phi_x^2}{2}&=\int^{x}_{-\infty}-\phi_x(\rho-1)\,dx'+\int^x_{-\infty}e^{\phi}\phi_x-\phi_x \,dx'\\
				&\leq \left(\int^{x}_{-\infty}\phi_x^2\, dx'\right)^{1/2}\left(\int^{x}_{-\infty}(\rho-1)^2dx'\right)^{1/2}+e^{\phi}-\phi-1\\
				&\leq \sqrt{2H_\gamma(-\ve)}\left(\int^{x}_{-\infty}(\rho-1)^2 \,dx'\right)^{1/2}+e^{M_1}-M_1-1.
			\end{split}
		\end{equation}
		Here we have used the fact 	 from the definition of $H_\gamma(t)$ that $$\frac{1}{2}\int_{\mathbb{R}}|\phi_x|^2dx'\leq   H_\gamma(t)= H_\gamma(-\ve)$$
		and the fact  from \eqref{Phi_0_M_gen} that 
		$$e^{\phi}-\phi-1 \leq e^{M_1}-M_1-1.$$
		Here, we notice that 
		\begin{equation}\label{rho-c}
			\int_{\mathbb{R}} |\rho-1|^2\,dx \lesssim \int_{\mathbb{R}} \mathcal{P}_{\gamma}(\rho)\lesssim H_{\gamma}(t),
		\end{equation} 
which follows from the fact that  $\mathcal{F}(\rho):=\frac{1}{\gamma}(\rho^\gamma-1)-(\rho-1)-c(\rho-1)^2 \geq 0 $ on $\rho \in(0,+\infty)$, for any fixed $0<c<\min\{(\gamma-1)/2,(\gamma-1)/\gamma\}$.  Hence, combining \eqref{H-L2-gen}, \eqref{phix2} and \eqref{rho-c}, we obtain
		\[ |\phi_x | \le M_2 (  \| (w_0, z_0 ) \|_{L^2(\mathbb{R})} )  =: M_2.\] We finish the proof.
		
	\end{proof}

Using Lemma~\ref{phi_lem_gen}, one can obtain from \eqref{EP-Isen'} and \eqref{EP2_3_gen} that 
		\begin{equation}\label{phi-yy-e}
			\begin{split}
			|\Phi_{yy}e^{3s}|  \lesssim 1.
			\end{split}
		\end{equation}

\subsection{Uniform lower bound of $\rho$}
For the isothermal case,  from the relation $\rho=e^{(w-z)/(2\sqrt{K})}$, Proposition~\ref{rho-1_prop}  implies that 
	$\rho$ has a uniform positive lower bound, provided that $\|w_0-z_0\|_{L^{\infty}}$ is small enough. 
	%
	%
	 However, for the isentropic case, the situation is more delicate. More precisely,  since $\rho=\left(\tfrac{\alpha}{2}(w-z)\right)^{1/\alpha}$, the smallness of $w-z$ does not ensure the lower bound of $\rho$.  
	 %
	 %
	 To resolve this issue, from \eqref{init_gen_6} and the bootstrap assumptions, we  show that   $w-z$  has a uniform positive lower bound.  
	\begin{lemma}\label{vacuum_isen}
		Suppose that the bootstrap assumptions \eqref{boots_isen}--\eqref{isen_rho-1} and the initial conditions \eqref{init_gen_0} hold. 
		Then we have 
		\begin{equation*}
			\inf_{x\in\mathbb{R}}\rho(x,t) \geq P_- e^{-\veps^{1/2}}>0.
		\end{equation*} 
	\end{lemma}
	\begin{proof}
		Subtracting \eqref{EP-Isen'1} from \eqref{EP-Isen'2}, we have
		\begin{equation}\label{w-z-isen}
			(w-z)_t+\left(\frac{1-\alpha}{1+\alpha}z+w\right)(w-z)_x+\frac{2\alpha}{1+\alpha}z_x(w-z)=0.
		\end{equation}
		Let  $\zeta(t;x)$ be the characteristic curve defined by the initial value problem: 
		$$\frac{d}{dt} \zeta=\frac{1-\alpha}{1+\alpha}z(\zeta,t)+w(\zeta,t)$$ with $\zeta(-\veps;x)=x$. Integrating \eqref{w-z-isen} along $\zeta$ over $[-\ve, t]$, we get
		\begin{equation}\label{w-z''}
			(w-z)(\zeta(t;x),t)=(w_0(x)-z_0(x))e^{-\int^t_{-\veps}\frac{2\alpha}{1+\alpha}z_x(\zeta(t';x),t')\,dt'}.
		\end{equation}
		On the other hand, from \eqref{difz32-isen} and \eqref{dottau_isen}, we see that 
		\begin{equation}\label{w-z_damp'}
  \begin{split}
			-\int^t_{-\veps}\frac{2\alpha}{1+\alpha}z_x ( \zeta(t';x),t') \,dt' 
   & \ge -  \frac{2\alpha}{1+\alpha} \int^s_{s_0}\frac{e^{s'/2} \| Z_y (\cdot, s') \|_{L^\infty} }{1-\dot{\tau}}\,ds' 
   \\ 
   & \geq -\frac{4\alpha}{1+\alpha}\int^s_{s_0}  e^{-s'}\,ds'\geq -\veps^{1/2}.
   \end{split}
		\end{equation}
		Using \eqref{init_gen_6} and \eqref{w-z_damp'} for \eqref{w-z''}, we get
		\begin{equation}\label{rho-low}
			(w-z)(\zeta(t;x),t)\geq P_- e^{-\veps^{1/2}}.
		\end{equation} 
		Here note that  $\zeta(t;\cdot): \mathbb{R} \to \mathbb{R}$ is a bijective function of $x$  for each $t$ as long as the solution exists. 
		%
		%
		 Thus, together with the relation $\rho=\left(\alpha(w-z)/2\right)^{1/\alpha}$, \eqref{rho-low} gives the desired bound. We finish the proof.
	\end{proof}
	Using Lemma~\ref{phi_lem_gen} and Lemma~\ref{vacuum_isen}, we can  obtain  bounds for higher derivatives of $\Phi$, which is corresponding to Lemma~\ref{phi_y1y2_lem} and Lemma~\ref{difZn} for the isothermal case. In fact we have the following result.
\begin{lemma}\label{high-order-phi-y}
Under the same assumptions as in Lemma~\ref{vacuum_isen}, as long as the smooth solution exists, it holds that 
	\begin{equation}\label{high-phi-isen}
		\| \partial_y^n \Phi(\cdot,s) \|_{L^{\infty}} \le C(n)e^{-3s}, \qquad n=2,3,4,5, 
	\end{equation}
	for some  positive constant $C(n)$ depending only on $n$.
\end{lemma} 
We can readily prove this by following the proof of Lemma~\ref{phi_y1y2_lem}, so we omit the details here. 
In fact, the bounds \eqref{Phi_0_M_gen}--\eqref{Phi_1_M_gen} in Lemma~\ref{phi_lem_gen} and \eqref{high-phi-isen} are crucially used to  close bootstrap assumptions, especially to close \eqref{W-y-decay-isen}.
So far, we have discussed some differences in the analysis between the isothermal and isentropic cases. 
	In summary, under our bootstrap assumptions \eqref{boots_isen}--\eqref{isen_rho-1} and the initial conditions \eqref{init_gen_0}, one can show  that the following  bounds hold and close the bootstrap. 
	\begin{subequations}\label{boots_isen'}
		\begin{align}
			&|W_y(y,s)-\overline W'(y)| \leq \frac{y^2}{20(1+y^2)},
			\\
			&|W_y(y,s)-\overline{W}'(y)| \le \frac{24}{25}\frac{1}{y^{2/3}+8}, \label{decay_Wy_isen}
			\\
			&|W_{yy}(y,s)| \leq \frac{14|y|}{(1+y^2)^{1/2}},
			\\
			&|\partial_y ^3 W(0,s)-6| \leq \veps^{1/4}, 
			\\
			&\|\partial_y ^3 W(\cdot, s) \|_{L^{\infty}} \leq  \frac{M^{5/6}}{2} , 
			\\
			&\|\partial_y ^4 W(\cdot, s) \|_{L^{\infty}} \leq \frac{M}{2}, 
			\\
			&\|Z_y\|_{L^{\infty}}\leq \frac{1}{2} e^{-(1/2+\delta) s}, \label{111}
			\\
			& |\dot{\tau}| \leq e^{-s} \\
			&e^{-s}(y^{2/3}+8)\left|\left(\frac{\alpha}{2}\left(e^{-s/2}W+\kappa-Z\right)\right)^{1/\alpha}-1\right|\leq \frac{3}{32}.
		\end{align}
	\end{subequations}
	Since the    analysis to close the bootstrap assumptions is fairly similar  to that  of the isothermal case except \eqref{111}, we omit the details here and refer to the counterparts for the isothermal case. 
 	Here, we introduce the proof of \eqref{111}.
  \begin{lemma}\label{pos_damp_isen}
 	Under the assumptions \eqref{boots_isen}, it holds that $ \|e^{s/2}Z_y\|_{L^{\infty}} \leq  e^{-\delta s}/2$. 
 \end{lemma}
 \begin{proof}
 	By \eqref{EP2_2_gen}, we have the equation of $e^{s/2}Z_y$ that
 	\begin{equation}\label{eZy}
 		\begin{split}
 			&	\left( \partial_s + 1 +\frac{1-\alpha}{1+\alpha}W_y \right)(e^{s/2}Z_y) + U^Z_\alpha (e^{s/2}Z_y)_{y}
 			\\
 			& =-\frac{2}{1+\alpha}\frac{e^{s}\Phi_{yy}}{1-\dot{\tau}}-\frac{1-\alpha}{1+\alpha}\frac{\dot{\tau}}{1-\dot{\tau}}(e^{s/2}Z_y)W_y-\frac{(e^{s/2}Z_y)^2}{1-\dot{\tau}} =: F_\alpha. 
 		\end{split}
 	\end{equation} 
 	For the damping $D_\alpha : =1 +\frac{1-\alpha}{1+\alpha}W_y$, we have from $\alpha>0$ and $|W_y|\leq 1$ thanks to \eqref{Uy1_gen} that 
 	\begin{equation}\label{alpha-damping} 
 		D_\alpha \geq 1-\left|\frac{1-\alpha}{1+\alpha}\right|> \delta.
 	\end{equation} 
 	(Recall  $\delta\in(0,  (1- |1-\alpha|(1+\alpha )^{-1})/4)$.)    
 	By \eqref{difz32-isen}, \eqref{dottau_isen} and  \eqref{phi-yy-e}, we have  $\|F_\alpha\|_{L^{\infty}} \leq e^{-2\delta s}. $ 
 	By integrating the equation \eqref{eZy} along the characteristic $U^Z_\alpha$, we have the desired result.
 \end{proof}

 Now,   following the proof of Theorem~\ref{mainthm1} presented in Section~\ref{global-conti}, we can prove  the assertions of Theorem~\ref{main-isen}. 
\subsection{Proof of Theorem~\ref{main-isen}}\label{subsec51} 
To parallel subsection~\ref{global-conti}, we can show that $[w]_{C^{1/3}}<\infty$ by using \eqref{decay_Wy_isen}. To conclude that $\rho_x$ and $u_x$ blows up at the same time, we also need to prove that $z_x$ is finite for $t_0\leq t\leq T_*$, i.e.,
\begin{equation}\label{222}
	|z_x(x,t)|<\infty \quad \text{for all } (x,t)\in \mathbb{R}\times [t_0,T_*].
\end{equation} 
%
%
\begin{remark}
In fact, one can prove $\sup_{-\veps\leq t <T_\ast}\|z_x(\cdot,t)\|_{L^{\infty}}$ in a similar way as in the isothermal case, which is stronger result than \eqref{222}. More precisely, following the proof of  Lemma~\ref{Zder_lem}, one can close the bootstrap assumptions \eqref{boots_isen} with $\|Z_y\|_{L^{\infty}}\leq e^{-(1/2+\delta)s}$ replaced by $\|Z_y\|_{L^{\infty}}\leq Ce^{-3s/2}$ if we additionally assume that $\inf_{x\in \mathbb{R}}(w_0-z_0) \geq \|w_0-z_0\|_{L^\infty}e^{-\frac{\alpha}{|1-\alpha|}}$. We omit the details.
\end{remark}

\begin{lemma}
Under the same assumptions as in Lemma~\ref{vacuum_isen}, \eqref{222} holds.
\end{lemma}

\begin{proof}

By taking $\partial_x$ of \eqref{EP-Isen'2} and then integrating the resulting equation along the flow $\eta=\eta(t;x)$ defined as
\begin{equation}\label{eta}
\partial_t \eta(t) =  z(\eta,t)+ \frac{1-\alpha}{1+\alpha} w(\eta,t), \quad \eta(-\veps;x_0)=x_0,
\end{equation}
 we obtain 
\begin{equation*}
\partial_x z(\eta,t) = \partial_x z_0(x_0) e^{-\int_{-\veps}^t ( \partial_x z + \tfrac{1-\alpha}{1+\alpha} \partial_x w)\,dt'} - \frac{2}{1+\alpha} \int_{-\veps}^t e^{-\int_{t'}^t ( \partial_x z + \tfrac{1-\alpha}{1+\alpha} \partial_x w)\,dt''}  \partial_x^2 \phi \,dt'.
\end{equation*}
By \eqref{111} and \eqref{phi-yy-e}, we see that $\textstyle\int_{-\veps}^t  |\partial_x z (\eta,\cdot)| \,dt' \leq \int_{s_0}^s  e^{-\delta s'} \,ds' \lesssim 1$ and $\|\partial_x^2 \phi\|_{L^\infty} \lesssim 1$ uniformly in  $t \in [-\veps,  T_\ast)$. 

We claim that for any fixed $x_0\in\mathbb{R}$ and for all $t$ sufficiently close to $T_*$, there exists constant $C_{x_0}>0$ such that
\begin{equation}\label{Ax3}
\int_t^{T_\ast} |\partial_x w(\eta(t';x_0),t')|\,dt' <C_{x_0}.
\end{equation}
In view of \eqref{dx-dy} and \eqref{dottau_isen}, we use the change of variables $e^{-s}\,ds =  (1-\dot\tau(t) ) dt$ and the fact that $|W_y|\leq cy^{-2/3}$, which comes from \eqref{W-y-decay-isen} and \eqref{y-w-y}, to see that 
\begin{equation}\label{Ax2}
	\begin{split}
		\int^{T_*}_{t}|\partial_xw(\eta,t')|\,dt'
		& \le  c \int^{\infty}_{s}|W_y((\eta-\xi)e^{3s'/2}, s')|\,ds' \\
		& \leq \int^{\infty}_{s}\frac{c}{(\eta-\xi)^{2/3}}e^{-s'}\,ds' \le \int_{t}^{T_\ast}\frac{c}{(\eta-\xi)^{2/3}}\,dt',
	\end{split}
\end{equation}
where $\xi$ is defined in \eqref{modulation_gen}.

 We split into two cases: (i) $\lim_{t \to T_\ast}\eta(t) \neq  \lim_{t \to T_\ast} \xi(t)$ and (ii)  $\lim_{t \to T_\ast}\eta(t) =  \lim_{t \to T_\ast} \xi(t)$. 
 
 For the case (i),  by  continuity, there is $c>0$ such that $|\eta(t)-\xi(t)| >c$ for all $t$ sufficiently close to $T_\ast$, for which \eqref{Ax3} immediately follows. 

Let us consider the case (ii), i.e., $\lim_{t \to T_\ast}\eta(t) =  \lim_{t \to T_\ast} \xi(t)$. 
We estimate the lower bound of $|\eta-\xi|$. From \eqref{modulation_gen_3} and \eqref{eta}, we have
\begin{equation*}
\begin{split}
\eta(t) - \xi(t)  
& =  \frac{2\alpha}{1+\alpha} \int_t^{T_\ast} \kappa dt' - \int_t^{T_\ast}\left( z - \frac{1-\alpha}{1+\alpha}z(0,\cdot) \right)\,dt' \\
&\quad  + \int_s^\infty \frac{2e^{s'/2}\partial_y^3\Phi(0,\cdot)-(1-\alpha)\partial_y^2 Z(0,\cdot)}{(1+\alpha)\partial_y^3 W(0,\cdot)} e^{-s'} \,ds' \\
& \quad - \frac{1-\alpha}{1+\alpha} \int_s^\infty e^{-3s'/2}W\,ds'
\\
&=:I_1+I_2+I_3+I_4.
\end{split}
\end{equation*} 
We  notice that by \eqref{dottau_isen}, it holds that 
	\begin{equation*}
		T_*=\tau(T_*)=\int^{T_*}_{t_0}\dot{\tau}\,dt\leq 2\int^{\infty}_{s_0}e^{-2s}\,ds\leq \veps.
	\end{equation*}

\textit{Estimate of $I_1$}:  We take $\partial_y^2$ of \eqref{EP2_2_gen}, and we have
\begin{equation*}
	\partial_sZ_{yy}+D^Z_2Z_{yy}+U^Z_{\alpha}Z_{yy}=F^Z_2,
\end{equation*}
where $D^Z_2:=3+\frac{1-\alpha}{1+\alpha}\frac{2W_y}{1-\dot{\tau}}+\frac{3e^{s/2}Z_y}{1-\dot{\tau}}$ and $F^Z_2:=-\frac{2}{1+\alpha}\frac{e^{s/2}\partial_y^3\Phi}{1-\dot{\tau}}-\frac{1-\alpha}{1+\alpha}\frac{W_{yy}Z_y}{1-\dot{\tau}}$.
Since $D^Z_2\geq 1$ and $\|F^Z_2\|_{L^{\infty}}\leq 15e^{-(\frac{1}{2}+\delta)s}$ by \eqref{dottau_isen}, \eqref{high-phi-isen}, \eqref{EP2_1D3-isen} and \eqref{difz32-isen}, we get 
\begin{equation}\label{Zy2_isen}
	\|Z_{yy}\|_{L^{\infty}}\leq e^{-(\frac{1}{2}+\delta)s}.
\end{equation}

Integrating \eqref{modulation_gen_2} in $t$, and then using \eqref{Wy30_isen}, \eqref{Phi_1_M_gen}, \eqref{phi-yy-e}, and \eqref{Zy2_isen}, we obtain 
\begin{equation}\label{kappa_1}
|\kappa(t) - \kappa_0| \leq C\veps^{1/2+\delta}.
\end{equation}
Hence, by choosing sufficiently small $\veps(\kappa_0)>0$, we have
\begin{equation}\label{I1}
I_1 \geq \frac{\alpha}{1+\alpha} \kappa_0(T_\ast-t).
\end{equation}

\textit{Estimate of $I_2$}: Integrating \eqref{EP-Isen'1} and \eqref{EP-Isen'2} along each characteristics, and then using \eqref{Phi_1_M_gen}, we obtain the uniform bounds for $w$ and $z$: 
\begin{equation}\label{zw_upp}
\|z\|_{L^\infty} \leq \|z_0\|_{L^\infty} + \veps M_2, \quad \|w\|_{L^\infty} \leq \|w_0\|_{L^\infty} + \veps M_2.
\end{equation}
Using \eqref{init_gen_3} and \eqref{zw_upp}, we see that 
\begin{equation}\label{I2}
\begin{split}
|I_2|   \leq \left( \frac{\alpha}{1+\alpha} + \veps M_2\right)(T_\ast - t).
\end{split}
\end{equation} 

\textit{Estimate of $I_3$}: Using \eqref{Wy30_isen}, \eqref{high-phi-isen}, and \eqref{Zy2_isen}, we obtain 
\begin{equation}\label{I3}
	|I_3| \leq C\veps^{\frac{1}{2}+\delta}(T_\ast-t).
\end{equation}

\textit{Estimate of $I_4$}: First of all, using \eqref{W-y-decay-isen} and \eqref{y-w-y}, it is easy to obtain $|W(y,s)|\leq C|y|^{1/3}$. Applying this to $I_4$, we obtain
	\begin{equation}\label{I4'}
		|I_4| \lesssim \int^{\infty}_{s}e^{-3s'/2}|W((\eta-\xi)e^{3s'/2},s')|\,ds' \leq \int^{\infty}_s e^{-s'}|\eta-\xi|^{1/3}\,ds'.
	\end{equation}
	Now, we obtain a proper upper bound for $|\eta-\xi|$. Using \eqref{Wy30_isen},  \eqref{high-phi-isen}, \eqref{Zy2_isen}, \eqref{kappa_1}, and \eqref{zw_upp}, we obtain from \eqref{modulation_gen_3} that $|\dot{\xi}|\leq C$. Hence, using \eqref{zw_upp}, we obtain
\begin{equation}\label{Ax1}
	|\eta(t)-\xi(t)|=\left|\int^{T_*}_{t}(z+\frac{1-\alpha}{1+\alpha}w)(t')-\dot{\xi}(t')\,dt'\right| =  O(T_*-t) = O(\veps).
\end{equation} 
From \eqref{I4'} and \eqref{Ax1}, we get
\begin{equation}\label{I4}
	|I_4|\leq C\veps^{1/3}(T_*-t).
\end{equation}

Combining \eqref{I1}, \eqref{I2}, \eqref{I3} and \eqref{I4}, we conclude that
\begin{equation*}
	\eta(t)-\xi(t)\geq \left(\frac{\alpha}{1+\alpha}\kappa_0-\left(\frac{\alpha}{1+\alpha}+\veps M_2\right)-C\veps^{\frac{1}{2}+\delta}-C\veps^{1/3}\right)(T_*-t) \geq c(T_*-t)
\end{equation*}
for some positive constant $c$, provided that $\kappa_0>1$ and $\veps$ sufficiently small depending on $\kappa_0$. 
 Then, this together with \eqref{Ax2} implies \eqref{Ax3}.
This completes the proof.
\end{proof}

\section{Appendix}\label{appendix}

\subsection{Rescaling initial data}\label{scaling-IC} 
%

In this subsection, we discuss how the initial condition \eqref{init_w_3} can be relaxed. Suppose that at some time $t_\ast$ and point $x_\ast$, it holds that
\begin{equation}\label{init_w_3'}
	\partial_xw(t_\ast,x_*)= - \veps^{-1}, \quad \partial_x^2w(t_*,x_*)=0, \quad \partial_x^3w(t_*,x_*)=6\mu^{-2}\veps^{-4}
\end{equation}
for sufficiently small $\veps>0$ and some constant $\mu>0$ (independent of $\veps$). In other words, $-\partial_x w(t_\ast,x_\ast)$ is sufficiently large, $-\partial_x w(t_\ast,x)$ attains its local maximum at $x=x_\ast$, and   $\partial_x^3 w(t_\ast,x_\ast) \sim \veps^{-4}$. By time and space translation $t\mapsto t-t_*-\veps$ and $x\mapsto x-x_*$, we see that \eqref{init_w_3'} holds with $t_\ast=-\veps$ and $x_\ast=0$. 

We introduce new variables $t'= t$, $x'=\mu x$, and 
\begin{equation*}
	w(t,x)=\frac{1}{\mu}\wt{w}(t',x'), \quad z(t,x)=\frac{1}{\mu}\wt{z}(t',x'), \quad \phi(t,x)=\wt{\phi}(t',x'). 
\end{equation*}
Then, from \eqref{iso_EP},  we get
\begin{equation}\label{wt_eq}
	\begin{split}
		&\wt{w}_{t'}+(\wt{w}+\wt{z}+2\sqrt{K'})\wt{w}_{x'}=-2\mu^2\wt{\phi}_{x'},
		\\
		&\wt{z}_{t'}+(\wt{z}+\wt{w}-2\sqrt{K'})\wt{z}_{x'}=-2\mu^2\wt{\phi}_{x'},
		\\
		&-\mu^2\wt{\phi}_{x'x'}=e^{\frac{\wt{w}-\wt{z}}{2\sqrt{K'}}}-e^{\wt{\phi}},
	\end{split}
\end{equation}
where $K':=K\mu^2$. We notice that $\partial_x w_0(0)=\partial_{x'}\wt{w}_0(0)$ and  $\partial_{x'}^3\wt{w}_0(0)= \mu^2\partial_x^3w_0(0)= 6\veps^{-4}$.  The coefficient $\mu$ of the system \eqref{wt_eq} does not affect our analysis as long as $\mu$ is independent of $\veps$. In conclusion, we can relax \eqref{init_w_3} to any $w_0$ satisfying \eqref{init_w_3'}. We remark that for the one-dimensional Euler-Poisson system (\eqref{wt_eq} without the forcing term $\phi$), the condition for $\partial_x^3w_0$ in  \eqref{init_w_3'} can be relaxed to $\partial_x^3w_0>0$.

\subsection{Useful inequalities}\label{inequalities}

In this subsection,  we give a set of   inequalities that are used in the course of our analysis. 
We note that 
$\overline{W}$ of \eqref{Burgers_SS} is an implicit solution of 
\begin{equation}\label{Wbar_1}
	y=-\overline{W}-\overline{W}^3.
\end{equation}
Differentiating \eqref{Wbar_1} in $y$, we obtain
\begin{equation}\label{Wbar_2}
	\overline{W}'=-\frac{1}{1+3\overline{W}^2}
\end{equation}
and
\begin{equation}\label{W''_eq}
	\overline{W}'' = \frac{6\overline{W}^2}{1+3\overline{W}^2}\frac{(\overline{W}')^2}{-\overline{W}}.
\end{equation}
From \eqref{Wbar_2}, we have 
\begin{equation}\label{Wbar_sign}
	\overline{W}'\leq 0, \qquad y\in\mathbb{R}.
\end{equation}
Combining with the fact that $\overline{W}(0)=0$, we   see that
\begin{equation}\label{sign_W}
	\overline{W}(y)< 0 \quad \text{for}\quad y > 0, \qquad \overline{W}(y) > 0 \quad \text{for}\quad y < 0.
\end{equation}
By Young's inequality, we have
\begin{equation*}
	-\overline{W}=-\frac{\overline{W}}{y^{2/3}}\cdot y^{2/3}\leq -\frac{\overline{W}^3}{3y^2}+\frac{2y}{3}, \qquad y > 0.
\end{equation*}
This with \eqref{Wbar_1} gives
\begin{equation}\label{Wbar_3}
	-\overline{W}^3\geq \frac{y^3}{3y^2+1},  \qquad y\geq 0.
\end{equation}
Similarly, we have 
\begin{equation}\label{Wbar_4}
	-\frac{y^3}{1+3y^2}\leq \overline{W}^3, \qquad y\leq 0.
\end{equation}
From \eqref{Wbar_3} and \eqref{Wbar_4}, we get 
\begin{equation}\label{Wbarsq}
	\overline{W}^2\geq \frac{y^2}{(1+3y^2)^{2/3}}, \qquad y \in\mathbb{R}.
\end{equation}
Combining \eqref{Wbar_2} and  \eqref{Wbarsq}, we obtain a useful bound for small $|y|$, i.e.,
\begin{equation} \label{y-w-y2}
 -\frac{1}{1+\frac{3y^2}{(3y^2+1)^{2/3}}}\leq  \overline{W}'\leq 0.
\end{equation}

Next we shall obtain sharper bounds of $\overline{W}$ for large $|y|$. For any   $y\ne0$, setting $X:=-\overline{W}{y^{-1/3}}$, we have from \eqref{Wbar_1} that   $	f(X):=X^3+{y^{-2/3}}X-1=0$.  Let $a$ and $b$ be any positive numbers satisfying 
\begin{equation}\label{ab}
	b^3+ba^{-2/3}-1\le 0.
\end{equation} 
For $|y|\geq a>0$, we see that $f(1)=y^{-2/3} > 0$ and $f(b)\leq b^3+b a^{- 2/3}-1\le 0$. Combining with the fact that $f'(X)=3X^2+y^{-2/3}>0$, 
 we see that $0<b<1$ and that $X$ must be in the interval $(b,1)$. This implies that 
\begin{equation}\label{far_W}
	b|y|^{1/3}\leq |\overline{W}|\leq |y|^{1/3}, \qquad |y|\geq a.
\end{equation}

\begin{lemma}[Decay rate for $\overline{W}'$ and $\overline{W}''$] There exists a constant $C>0$ such that for all $y\in\mathbb{R}$ the following hold:
\begin{subequations}
		\begin{align}
			&|\overline{W}(y)|\leq  |y|,  \label{W<y}
			\\
			&|\overline{W}'(y)|\leq  C(1+y^2)^{-1/3}, \label{y-w-y} \\
			& |\overline{W}''(y)|\leq C(1+y^2)^{-5/6}. \label{Wyybar_bdd}
		\end{align}
	\end{subequations}
\end{lemma}
\begin{proof}
Thanks to \eqref{Wbar_2}, which implies $|\overline{W}'|\leq 1$, we obtain \eqref{W<y} since
	\begin{equation*}
		|\overline{W}(y)|\leq \int^{|y|}_0|\overline{W}'(y')|\,dy' \leq |y|.
	\end{equation*}
	
For $a=1$, there exists $b=b(1)>0$ satisfying \eqref{ab}. From this and \eqref{far_W}, we get
	\begin{equation}\label{ab_1}
		b|y|^{1/3}\leq |\overline{W}| \leq |y|^{1/3}, \qquad |y|\geq 1.
	\end{equation} 
	From \eqref{Wbar_2} and \eqref{ab_1}, we see that 
	\begin{equation*}
		|\overline{W}'(y)|\leq \frac{1}{1+3b^2y^{2/3}}, \qquad |y| \geq 1,
	\end{equation*}
which implies \eqref{y-w-y} since $|\overline{W}'| \leq 1$ for all $y\in\mathbb{R}$.
	
From  \eqref{W''_eq}, \eqref{y-w-y} and \eqref{ab_1}, we see that there is a constant $C>0$ such that 
	\begin{equation*}
		|\overline{W}''|=  \frac{6\overline{W}^2}{1+3\overline{W}^2}\frac{(\overline{W}')^2}{|\overline{W}|} \leq Cy^{-5/3}, \qquad |y|\ge 1.
	\end{equation*} 
	On the other hand, from  \eqref{W''_eq} and $|\overline{W}'|\le 1$,  
	 one can easily check that 
	\begin{equation*}
		\begin{split}
			|  \overline{W}''  | & =  \frac{6|\overline{W}|}{1+3\overline{W}^2} (\overline{W}')^2    \le  C, \qquad  |y|\le 1.
		\end{split}
	\end{equation*} 
Combining the above two inequalities, we obtain  \eqref{Wyybar_bdd}.
\end{proof}

\begin{lemma}
	It holds that for all $y\in\mathbb{R}$,
	\begin{subequations}
		\begin{align}
			&1+2\overline{W}'+\frac{2}{1+y^2}\left(\frac{3}{2}+\frac{\overline W}{y}\right)\geq \frac{y^2}{5(1+y^2)}+\frac{16 y^2}{5(1+8y^2)},  \label{0605_m_3}
			\\
			&\frac{5}{2}+3\overline{W}' +\frac{1}{1+y^2}\left(\frac{3}{2}+\frac{\overline{W}}{y}\right)
			\geq \frac{2y^2}{3(1+y^2)}. \label{0605_m_5}
		\end{align}
	\end{subequations}
\end{lemma}
\begin{proof}
	By \eqref{y-w-y2} and \eqref{W<y}, we have 
	\begin{equation*}
		\begin{split}
			1+2\overline{W}'+\frac{2}{1+y^2}\left(\frac{3}{2}+\frac{\overline W}{y}\right) 
			&\geq 1-\frac{2}{1+\left(\frac{3y^2}{1+3y^2}\right)^{2/3}}+\frac{1}{1+y^2}.
		\end{split}
	\end{equation*} 
It is straightforward to check that for $|y| \leq 0.7$,
	\begin{equation*}
		1-\frac{2}{1+\left(\frac{3y^2}{1+3y^2}\right)^{2/3}}+\frac{1}{1+y^2} \geq \frac{y^2}{5(1+y^2)}+\frac{16y^2}{5(1+8y^2)}.
	\end{equation*}
Combining the above inequalities, we see that  \eqref{0605_m_3} holds for $|y| \leq 0.7$.

Let us consider the region $|y|\geq 0.7$. We let $a=0.7$ and $b=0.6$ so that \eqref{ab} is satisfied. Then, from \eqref{far_W}, we see that  $0.6|y|^{1/3}\leq |\overline{W}| \leq |y|^{1/3}$ for $ |y|\geq 0.7$. Using this and \eqref{Wbar_2}, we have	 
	 \begin{equation*}
	 	\overline{W}'\geq -\frac{1}{1.08y^{2/3}+1}.
	 \end{equation*} 
	 Thus, using the above inequalities, we see that for $|y| \geq 0.7$,
	\begin{equation*}
		\begin{split}
			1+2\overline{W}'+\frac{2}{1+y^2}\left(\frac{3}{2}+\frac{\overline W}{y}\right) 
			&\geq 1-\frac{2}{1.08y^{2/3}+1}+\frac{2}{1+y^2}\left(\frac{3}{2}-\min\left\{1,\frac{1}{y^{2/3}}\right\}\right)
			\\
			&\geq \frac{y^2}{5(1+y^2)}+\frac{16y^2}{5(1+8y^2)},
		\end{split}
	\end{equation*}
	where one can check the last inequality by straightforward calculation. We finish the proof of \eqref{0605_m_3}.

	Now, we prove \eqref{0605_m_5}. From  \eqref{y-w-y2}, we see that  
	\begin{equation*}
		1+\overline{W}'\geq 1-\frac{1}{1+\frac{3y^2}{(3y^2+1)^{2/3}}}
			\geq \frac{y^2}{3(1+y^2)}.
	\end{equation*} 
	 Together with \eqref{0605_m_3}, this implies that for $y\in \mathbb{R}$,
	\begin{equation*}
		\begin{split}
			\frac{5}{2}+3\overline{W}' +\frac{1}{1+y^2}\left(\frac{3}{2}+\frac{\overline{W}}{y}\right)
			&= 2(1+\overline{W}')+\frac{1}{2}\left(1+2\overline{W}'+\frac{2}{1+y^2}\left(\frac{3}{2}+\frac{\overline{W}}{y}\right)\right) 
			\\
			&\geq \frac{2y^2}{3(1+y^2)}+\frac{1}{2}\left(\frac{y^2}{5(1+y^2)}+\frac{16y^2}{5(1+8y^2)}\right) 
			\\
			&\geq \frac{2y^2}{3(1+y^2)}.
		\end{split}
	\end{equation*}
	This completes the proof.
\end{proof}


	\begin{lemma}
		There exists a constant $\lambda>1$ such that for all $y\in\mathbb{R}$,
		\begin{equation}\label{0605_m_4}
			\lambda |\overline{W}''|  \frac{ 1+y^2}{y^2} \int^{|y|}_0{\frac{(y')^2}{1+(y')^2 }\,dy'}\leq\frac{3y^2}{1+8y^2}+\frac{y^2}{30(1+y^2)}.
		\end{equation}
	\end{lemma}
\begin{proof}
	
	We split the domain $\mathbb{R}$ into two regimes: $|y|\leq 3/2$ and  $|y| \geq 3/2$. 
	We first consider the region $|y| \geq 3/2$. Since ${y^2(1+y^2)^{-1}}$ is increasing for $y>0$, we have   
	\begin{equation}\label{m_4_int'}
		\frac{1+y^2}{y^2}\int^{|y|}_0 \frac{(y')^2}{1+(y')^2}\,dy' \le  |y|.
	\end{equation}
By choosing  $a=3/2$ and $b=3/4$, from \eqref{ab} and \eqref{far_W},   we see that $\frac{3}{4}|y|^{1/3}\leq |\overline{W}|\leq |y|^{1/3}$ for $ |y|\geq 3/2$. 	 Using this fact with \eqref{Wbar_2} and  \eqref{W''_eq},  we obtain 
	\begin{equation}\label{m_4_large}
		\begin{split}
		|\overline{W}''|
		& =\left|\frac{6\overline{W}^2}{1+3\overline{W}^2}\frac{(\overline{W}')^2}{-\overline{W}}\right|   \leq \frac{6|y|^{2/3}}{1+3|y|^{2/3}} \left|\frac{(\overline{W}')^2}{\overline{W}} \right|\\
		& \leq \frac{6}{3/4}\frac{|y|^{1/3}}{1+3y^{2/3}}\frac{1}{(1+3\cdot (3/4)^2 y^{2/3})^2}, \qquad |y|\geq 3/2.
		\end{split}
	\end{equation}  
	Combining \eqref{m_4_int'} and \eqref{m_4_large}, we obtain 
	\begin{equation*}
		\begin{split}
			\lambda |\overline{W}''|  \frac{ 1+y^2}{y^2} 	\int^{|y|}_0{\frac{(y')^2}{1+(y')^2 }\,dy'}&\leq \lambda \frac{6}{3/4}\frac{|y|^{4/3}}{1+3y^{2/3}}\frac{1}{(1+3\cdot (3/4)^2 y^{2/3})^2}.
		\end{split}
	\end{equation*}
	Now we claim that  for some $\lambda>1$,
\[
\lambda\frac{6}{3/4}\frac{|y|^{4/3}}{1+3y^{2/3}}\frac{1}{(1+3\cdot (3/4)^2 y^{2/3})^2} \leq \frac{3y^2}{1+8y^2}+\frac{y^2}{30(1+y^2)}, \qquad |y|\geq 3/2,
\]
which is equivalent to $240 \lambda \mathcal{F}_1(y)\mathcal{F}_2(y)\leq 1$ for  $ |y|\geq 3/2$, where $\mathcal{F}_1(y)=\tfrac{1+y^2}{91+98y^2}$ and $\mathcal{F}_2(y)=\tfrac{1}{(1+3\cdot (3/4)^2y^{2/3})^2}\tfrac{1+8y^2}{y^{2/3}(1+3 y^{2/3})}$. Here, $\mathcal{F}_1$ and $\mathcal{F}_2$ are decreasing for $y\geq 0$. Hence, we have for any $\lambda \in (1,6/5)$,  $240 \lambda \mathcal{F}_1(y)\mathcal{F}_2(y) \leq 1$	for  $|y| \geq 3/2$.
We finish the proof of the claim.

	Next, we consider the region $|y|\leq 3/2$. 
	By \eqref{Wbar_2}, \eqref{W''_eq}, \eqref{y-w-y2} and \eqref{W<y},  we have 
	\begin{equation}\label{m_4_small}
		\begin{split} 
		|\overline{W}''|=6|\overline{W}||\overline{W}'|^3 \leq 6|y|\left(1+\frac{3y^2}{(3y^2+1)^{2/3}}\right)^{-3}.
		\end{split} 
	\end{equation} 
	Using the fact that $\textstyle\int^{|y|}_0 \frac{(y')^2}{1+(y')^2}\,dy' = |y|-\tan^{-1}|y|$ 	and \eqref{m_4_small}, we obtain a bound for the left side of \eqref{0605_m_4}, i.e.,
	\begin{equation*}
		\begin{split}
			\lambda |\overline{W}''|  \frac{ 1+y^2}{y^2} \int^{|y|}_0{\frac{(y')^2}{1+(y')^2 }\,dy'}&\leq 6\lambda |y|\left(1+\frac{3y^2}{(3y^2+1)^{2/3}}\right)^{-3}\frac{1+y^2}{y^2}(|y|-\tan^{-1}|y|).
		\end{split}
	\end{equation*}
	Now, we claim that for some $\lambda>1$,
	\begin{equation*}
		6\lambda |y|\left(1+\frac{3y^2}{(3y^2+1)^{2/3}}\right)^{-3}\frac{1+y^2}{y^2}(|y|-\tan^{-1}|y|) \leq \frac{3y^2}{1+8y^2}+\frac{y^2}{30(1+y^2)}, \quad |y|\leq 3/2,
	\end{equation*} 
	equivalently,
	\begin{equation*}
		180\lambda \underbrace{(1+8y^2)\left(1+\frac{3y^2}{(3y^2+1)^{2/3}}\right)^{-3}}_{=:\mathcal{G}(y)}\cdot\underbrace{\frac{(1+y^2)}{98y^2+91}}_{=:\mathcal{G}_1(y)}\cdot\underbrace{\frac{1+y^2}{|y|^3}(|y|-\tan^{-1}|y|)}_{=:\mathcal{G}_2(y)}\leq 1, \quad |y|\leq 3/2.
	\end{equation*}
	We first check that $\mathcal{G}(y)$ has its maximum at $y=0$, in the domain $|y|\leq 3/2$. Note that 
\begin{equation*}
	\frac{d}{dy}\mathcal{G}(y)=\frac{2y(3y^2+1)}{(3y^2+(3y^2+1)^{2/3})^4}\cdot\underbrace{(8(3y^2+1)^{5/3}-57y^2-9)}_{=:\mathcal{G}_0(y)},
\end{equation*}
where $\mathcal{G}_0$ satisfies $\mathcal{G}_0(0)<0$ and $\mathcal{G}_0(3/2)>0$. Hence, $\mathcal{G}(y)$ decreases on $[0,\alpha]$ and increases on $[\alpha,3/2]$, where $\alpha>0$ is a unique zero of  $\mathcal{G}_0(y)$. Comparing the values $\mathcal{G}(0)$ and $\mathcal{G}(3/2)$, we deduce that $\mathcal{G}(y)\leq \mathcal{G}(0)=1 $ for $|y| \leq 3/2$. 

On the other hand, one can easily check that  $\mathcal{G}_1(y) \leq 1/91$ and $\mathcal{G}_2(y) \leq  \mathcal{G}_2(3/2)$ for $|y|\leq 3/2$  since $\mathcal{G}_1$ decreases and $\mathcal{G}_2$ increases on $y>0$. Combining all, we conclude that for any $1<\lambda<1.01$ and $|y| \leq 3/2$, $180\lambda \mathcal{G}(y)\mathcal{G}_1(y)\mathcal{G}_2(y) \leq 1$. This proves our claim, and we complete the proof of \eqref{0605_m_4}.

\end{proof}

\begin{lemma}
There exists $\lambda>1$ such that for $|y|\geq 3$,  
	\begin{subequations}
		\begin{align}
			\begin{split}
				&\lambda |\overline{W}''| (y^{2/3}+8) \int^{|y|}_{0}{\frac{dy'}{{y'}^{2/3}+8}}\leq
				\\
				&\qquad \qquad \left(1-\frac{1}{y^{2/3}+8}+2\overline{W}'-\frac{2y^{2/3}}{3(y^{2/3}+8)}\left(\frac{3}{2}+\frac{\overline{W}}{y}+\frac{1}{y}\int^y_0 \frac{dy'}{y'^{2/3}+8}\right)\right),\label{0605_m_7}
			\end{split}
			\\
			&1-\frac{1}{y^{2/3}+8}+\overline{W}'-\frac{2y^{2/3}}{3(y^{2/3}+8)}\left(\frac{3}{2}+\frac{\overline W}{y}+\frac{1}{y}\int_{0}^{y}{\frac{dy'}{{y'}^{2/3}+8}}\right)>0. \label{0605_m_7'}
		\end{align}
	\end{subequations}
\end{lemma}
\begin{proof}
By letting $a=3$ and $b=0.84$, we obtain from  \eqref{ab} and  \eqref{far_W} that 
\begin{equation}\label{Wbar_far}
	0.84|y|^{1/3}\leq |\overline{W}| \leq |y|^{1/3}, \qquad |y|\geq 3. 
\end{equation}
From \eqref{Wbar_2}, \eqref{W''_eq} and \eqref{Wbar_far}, we have 
\begin{equation}\label{Wbar''_far}
	|\overline{W}''|=\frac{6|\overline{W}|}{(1+3\overline{W}^2)^3}\leq \frac{6|y|^{1/3}}{(1+3\cdot(0.84)^2\cdot y^{2/3})^3}, \qquad |y|\geq 3.
\end{equation}
Thanks to \eqref{Wbar''_far}, the left hand side  of \eqref{0605_m_7} satisfies
\begin{equation*}
	\begin{split}
		|\overline{W}''|(y^{2/3}+8)\int^{|y|}_0 \frac{dy'}{(y')^{2/3}+8} 
		&\leq \frac{6|y|^{2/3}(y^{2/3}+8)}{(1+3\cdot (0.84)^2y^{2/3})^3}\min\left\{3,\frac{y^{2/3}}{8}\right\}
		\\
		&=:L, \qquad |y|\geq 3.
	\end{split}
\end{equation*}
Here we have used the fact that 
\begin{equation}\label{int-y13} \frac{1}{|y|^{1/3}} \int_0^{|y|} \frac{dy'}{ (y')^{2/3} +8 }\le \min\left\{3,\frac{y^{2/3}}{8}\right\}. 
\end{equation}
On the other hand, from \eqref{Wbar_2} and \eqref{Wbar_far}, we obtain
\begin{equation}\label{barw-3}
	\overline{W}'=-\frac{1}{1+3\overline{W}^2}\geq -\frac{1}{1+3\cdot(0.84)^2\cdot |y|^{2/3}}, \qquad |y|\geq 3.
\end{equation}
Thanks to \eqref{Wbar_far}, \eqref{int-y13} and \eqref{barw-3}, the right hand side  of \eqref{0605_m_7} satisfies for $|y|\geq 3$,
\begin{equation*}
	\begin{split}
	&	1-\frac{1}{y^{2/3}+8}+2\overline{W}'-\frac{2y^{2/3}}{3(y^{2/3}+8)}\left(\frac{3}{2}+\frac{\overline{W}}{y}+\frac{1}{y}\int^y_0 \frac{dy'}{y'^{2/3}+8}\right)
		\\
		&=\frac{7}{y^{2/3}+8}+2\overline{W}'-\frac{2}{3(y^{2/3}+8)}\left(\frac{\overline{W}}{y^{1/3}}+\frac{1}{y^{1/3}}\int^y_0 \frac{dy'}{y'^{2/3}+8}\right)
		\\
		&\geq \frac{7}{y^{2/3}+8}-\frac{2}{1+3\cdot (0.84)^2 y^{2/3}}-\frac{2}{3(y^{2/3}+8)}\left(-0.84+\min\left\{3,\frac{y^{2/3}}{8}\right\}\right)
		 =:R.
	\end{split}
\end{equation*}
For any $1<\lambda<1.1$,
it is straightforward to check that  $R-\lambda L >0$ for all $|y|\geq 3$. This proves \eqref{0605_m_7}.
Then, \eqref{0605_m_7'} immediately follows from \eqref{0605_m_7} and  $\overline{W}'<0$. 
\end{proof}


%
%
%
%
%
%
%
%
%
%
%

%

\subsection{Poisson equation}
We define a continuous function 
\begin{equation}\label{B_I}
	I(y):=(y^{2/3}+1)\int_{-\infty}^\infty \frac{e^{-|y-y'|}}{1+ |y'|^{2/3}}  \,dy', \quad y\in\mathbb{R}.
\end{equation}
By using  L'Hospital's Rule, one can check that the supremum of $I(y)$ over $y\in\mathbb{R}$ is finite. 
\begin{lemma}\label{Lem_Pois}
Let $f(y,s)$ be a continuous function such that
\begin{equation}\label{09}
\sup_{(y,s)\in\mathbb{R}\times [0,\infty)}\; (1+y^{2/3} ) |f(y,s)| \leq C_f
\end{equation}
for some constant $C_f>0$. Suppose that $C_f>0$ is a  constant such that 
	\begin{subequations}
		\begin{align}
			& C_f\sup_{y\in \mathbb{R}}I(y) \leq 1/4, \label{B_CCon1} 
			\\ 
			& e^{1/4}C_f\left(\sup_{y\in \mathbb{R}}I(y)\right)^2 \leq 2 \label{B_CCon2}
		\end{align}
	\end{subequations}
	hold, where $I(y)$ is defined in \eqref{B_I}. Then, the solution to the Poisson equation
	\begin{equation}\label{B_Pois}
		- \partial_{yy}\Phi = f(y,s) + 1 - e^\Phi
	\end{equation}
	satisfies for   $n=0,1,2$,
	\begin{equation}\label{B_Pois2} 
		 \sup_{(y,s)\in\mathbb{R}\times[0,\infty)}|(y^{2/3} + 1)\partial_y^n\Phi(y,s)| \lesssim C_f.
	\end{equation}
\end{lemma}

\begin{proof}

We define a sequence of functions as follows:
\[
\Phi_{n+1}   = \frac{1}{2} \int_{-\infty}^\infty e^{-|y-y'|}(f - (e^{\Phi_{n}} - 1 - \Phi_{n}))(y',s)\,dy' \quad \text{for } n \in \mathbb{N},
\]
where $\Phi_1$ is the solution to the linear inhomogeneous equation $(1 - \partial_{yy})\Phi_1 = f$. By induction, we claim that 
\begin{equation}\label{allPhi_14}
	\|(y^{2/3}+1)\Phi_n(y,s)\|_{L^{\infty}}\leq C_f\sup_{y\in \mathbb{R}}I(y)  \quad \text{ for all } n\in\mathbb{N}.
\end{equation} 
For $n=1$, \eqref{allPhi_14} holds since 
\begin{equation}\label{BPois1}
	\begin{split}
		|(y^{2/3}+1)\Phi_1|
		& = \left| \frac{(y^{2/3}+1)}{2}\int_{-\infty}^\infty e^{-|y-y'|} f(y',s) \,dy' \right| 
		\\
		& \leq  \frac{C_f}{2} I(y)
	\end{split}
\end{equation}
by  \eqref{09}.  Assume that \eqref{allPhi_14} holds for some $n=k$.  Then, using \eqref{B_CCon1}, \eqref{allPhi_14} and \eqref{BPois1}, we have
\begin{equation}\label{BPoisn}
	\begin{split}
		& (y^{2/3}+1)|\Phi_{k+1}(y,s)|
  \\ 
		& = (y^{2/3}+1)\left|\Phi_1(y,s)-\frac{1}{2}\int^{\infty}_{-\infty}e^{-|y-y'|}(e^{\Phi_k}-1-\Phi_k)\,dy' \right| 
		\\
		& \leq (y^{2/3}+1) \left(|\Phi_1| + \frac{1}{4} \int^{\infty}_{-\infty}e^{-|y-y'|}|\Phi_k|^2 e^{|\Phi_k|}\,dy' \right)   
		\\
		& \leq \frac{C_f}{2}\sup_{y\in \mathbb{R}}I  + \frac{e^{1/4}}{4}\big(C_f\sup_{y\in \mathbb{R}}I\big)^2 (y^{2/3}+1) \int^{\infty}_{-\infty}e^{-|y-y'|}(y'^{2/3}+1)^{-2}\,dy' 
		\\ 
		& \leq C_f\sup_{y\in \mathbb{R}}I \left( \frac{1}{2} + \frac{e^{1/4}}{4}C_f\big(\sup_{y\in \mathbb{R}}I\big)^2 \right),
	\end{split}
\end{equation}
where we have used
\[
|e^{\Phi_n} - 1 - \Phi_n| \leq  \frac{|\Phi_n|^2}{2}\sum_{j=2}^\infty \frac{|\Phi_n|^{j-2}}{j!/2} \leq \frac{|\Phi_n|^2}{2}e^{|\Phi_n|}
\]
in the second line. Together with \eqref{B_CCon2}, \eqref{BPoisn} implies that \eqref{allPhi_14} holds true for $n=k+1$. 

Now we show that $\{\Phi_n\}$ is Cauchy in $C_b(\mathbb{R}\times [0,\infty))$.  Since  $|\Phi_n| \leq 1/4$ by \eqref{B_CCon1} and \eqref{allPhi_14}, it holds that
\begin{equation*} 
	2(1-e^{-1/2}) e^{-1/4}  \leq 2(1-e^{-1/2}) e^{\Phi_n} \leq \frac{1-e^{-(\Phi_n-\Phi_{n-1})}}{\Phi_n-\Phi_{n-1}}e^{\Phi_n} \leq  -2(1-e^{1/2})e^{1/4} 
\end{equation*}
since $x \mapsto  \frac{1-e^{-x}}{x}$ is a positive decreasing function on $\mathbb{R}$. Hence, we see that
\begin{equation}\label{nu_n+11}
	\begin{split}
	&	\left| \frac{(e^{\Phi_n}-\Phi_n)-(e^{\Phi_{n-1}}-\Phi_{n-1})}{\Phi_n-\Phi_{n-1}}  \right|
		 = \left|1-e^{\Phi_n}\frac{1-e^{-(\Phi_n-\Phi_{n-1})}}{\Phi_n-\Phi_{n-1}} \right| 
		\\
		& \leq \max\left\{ 1-2e^{-1/4}(1-e^{-1/2}), -2e^{1/4}(1-e^{1/2})-1\right\} 
		  =:c_1 < 1.
	\end{split}
\end{equation}
Since $c_1<1$ and 
\begin{equation}\label{nu_n+1}
	\begin{split}
		|\Phi_{n+1}-\Phi_n|
		& \leq \frac{1}{2}\int^{\infty}_{-\infty}e^{-|y-y'|}|(e^{\Phi_n}-\Phi_n)-(e^{\Phi_{n-1}}-\Phi_{n-1})|\,dy' 
		\\ 
		& \leq  c_1\sup_{(y,s)\in\mathbb{R}\times[0,\infty)}|\Phi_{n}-\Phi_{n-1}|,
	\end{split}
\end{equation}  
we see  that $\{\Phi_n\}$ be a Cauchy sequence in $C_b(\mathbb{R}\times[0,\infty))$. Therefore,  $\lim_{n\rightarrow \infty}\Phi_n=\Phi$ exists uniformly in $(y,s)$. Furthermore, using \eqref{nu_n+11}, we obtain that
\begin{equation}\label{BPois2}
	\begin{split}
		\Phi 
		& =\lim_{n\rightarrow \infty}\Phi_{n+1}
		\\
		& =\lim_{n\rightarrow \infty} \frac{1}{2}\int^{\infty}_{-\infty}e^{-|y-y'|}( f- e^{\Phi_n}+1+\Phi_n)(s,y')\,dy' 
		\\
		&=\frac{1}{2}\int^{\infty}_{-\infty}e^{-|y-y'|}( f- e^{\Phi}+1+\Phi)(s,y')\,dy',
	\end{split}
\end{equation}
which means that $\Phi$ is the (unique) solution of Poisson equation.

By taking the limit of \eqref{allPhi_14}, we see that \eqref{B_Pois2} holds true for $n=0$. By taking the derivatives of \eqref{BPois2}, we get
\begin{subequations}\label{B_Pois3}
	\begin{align}
 		\partial_y \Phi  & =  \frac{1}{2}\int_{-\infty}^\infty -\frac{y-y'}{|y-y'|} e^{-|y-y'|}  (f - e^\Phi + 1 + \Phi)\,dy', 
 		\\
		\partial_{yy} \Phi  & =  \frac{1}{2}\int_{-\infty}^\infty  e^{-|y-y'|}  (f - e^\Phi + 1 + \Phi)\,dy'  - (f - e^\Phi + 1  + \Phi).
	\end{align}
\end{subequations} 
By applying \eqref{B_Pois2} for $n=0$ to \eqref{B_Pois3}, it is straightforward to see  that \eqref{B_Pois2} holds for $n=1,2$. 
\end{proof}

\subsection{Maximum Principles}
We present a     maximum principle, which is   modified  from the one developed in  \cite{BSV}, to apply to our analysis.
We consider the initial value problem: 
\begin{equation}\label{IVP-f}
	\begin{split} 
		& \partial_s f(y,s)+D(y,s) f(y,s)+U(y,s)\partial_yf(y,s)=F(y,s)+\int_{\mathbb{R}}f(y',s)K(y,s;y')\,dy',
		\\
		& f(y,s_0) = f_0(y), \quad 
		s\in[s_0,\infty), \quad y\in \mathbb{R}. 
	\end{split} 
\end{equation}

	\begin{lemma}
		\label{max_2}
		Let $f$ be a classical solution to IVP \eqref{IVP-f}. Let $\Omega \subseteq \mathbb{R}$ be any compact set. 
		Suppose that the following hold:  
		\begin{subequations}
			\begin{align}
				& \|f(\cdot,s)\|_{L^{\infty}(\Omega)}\leq m_0, \label{max_2_1}
				\\
				&  \|f(\cdot,s_0)\|_{L^{\infty}(\mathbb{R})}\leq m_0, \label{max_2_1'}
				\\
				& \int_{\mathbb{R}}|K(y,s;y')|\,dy'\leq \delta D(y,s)\quad \text{for}\quad (y,s)\in \Omega^c\times [s_0,\infty), \label{max_2_2}
				\\
				& \inf_{(y,s)\in \Omega^c\times [s_0,\infty)}D(y,s)\geq \lambda_D >0, \label{max_2_3}
				\\
				& \|F(\cdot,s)\|_{L^{\infty}(\Omega^c)}\leq F_0, \label{max_2_4}
				\\ 
				& \limsup_{|y|\rightarrow \infty}|f(y,s)| <  2m_0 \label{max_2_6}
			\end{align}
		\end{subequations} 
		for some $m_0, F_0, \lambda_D >0$ and $\delta<1$. If $m_0\lambda_D> F_0/(2-2\delta)$, 
		then $\|f(\cdot,s)\|_{L^{\infty}(\mathbb{R})}\leq 2m_0$.
	\end{lemma}

\begin{proof}
	Suppose to the contrary that $\|f(\cdot,s_1)\|_{L^{\infty}(\mathbb{R})} > 2m_0$ for some $s_1>s_0$. Then, by the continuity of $f(y,s)$ and \eqref{max_2_1'}, there exists $s_*\in(s_0, s_1)$  such that $\|f(\cdot,s)\|_{L^{\infty}(\mathbb{R})} \geq \|f(\cdot,s_*)\|_{L^{\infty}(\mathbb{R})}  =2m_0$ for all $s\in[s_*,s_1]$. Moreover, thanks to  \eqref{max_2_1} and \eqref{max_2_6}, there exists  $y_* \in \Omega^c$ such that $| f(y_*, s_*) | = \|f(\cdot,s_*)\|_{L^{\infty}(\mathbb{R})}  =2m_0$.
	This implies $\partial_yf(y_*,s_*)=0$.

	On the other hand, it holds that $\|f(\cdot,s_*-h)\|_{L^\infty} \geq |f(y_* - h U(y_*,s_*), s_* - h)| $ for any small $h>0$,  by the definition of $y_*$. 
	Thus, we obtain 
	\begin{equation}\label{AP_R1'} 
		\begin{split}
			\lim_{h \to 0^+} \frac{\|f(\cdot,s_*-h)\|_{L^\infty} - \|f(\cdot,s_*)\|_{L^\infty}}{-h} 
			& \leq (\partial_s + U(y_*,s_*)\partial_y)|f|(y_*,s_*),
		\end{split}
	\end{equation} 
	provided that the limit on the LHS of \eqref{AP_R1'} exists. Note that by Rademacher's theorem, 
	$ \| f(\cdot, s)\|_{L^{\infty}}$, being Lipschitz, is differentiable at almost all $s\in[s_0, s_1]$, and from \eqref{AP_R1'} and the fact that $\partial_y f(y_*, s_*) =0$, we have 
	\begin{equation} \label{L-thm'}
		\begin{split}
			\frac{d}{ds} \| f(\cdot, s)\|_{L^{\infty}} |_{s=s_*} 
			& \leq  \left\{ \begin{array}{l l}
				\partial_s f(y, s)|_{(y,s)=(y_*,s_*)} & \text{if } f(y_*,s_*)>0, \\
				-\partial_s f(y, s)|_{(y,s)=(y_*,s_*)} & \text{if } f(y_*,s_*)<0.
			\end{array}
			\right.  
		\end{split}
	\end{equation} 
	Since $\|f(\cdot,s)\|_{L^{\infty}(\mathbb{R})} \geq \|f(\cdot,s_*)\|_{L^{\infty}(\mathbb{R})}$ for all $s\in[s_*,s_1]$, we have $\frac{d}{ds}\|f(\cdot,s_*)\|_{L^{\infty}}\geq 0.$  
	Combining with \eqref{L-thm'}, 
	\begin{equation}\label{max_f'}
		\begin{split}
			\partial_s f(y, s)|_{(y,s)=(y_*,s_*)} \geq 0 & \quad \text{ if } f(y_*,s_*)>0, \\
			\partial_s f(y, s)|_{(y,s)=(y_*,s_*)} \leq 0 & \quad \text{ if } f(y_*,s_*)<0
		\end{split}
	\end{equation} 
	holds.
	
	Let us check the case $f(y_*,s_*)>0$, first. 
	By \eqref{max_2_2}, we obtain
	\begin{equation*}
		\left|\int_{\mathbb{R}} f(y',s_*)K(y_*,s_*;y')\,dy'\right|\leq \delta\|f(\cdot,s_*)\|_{L^{\infty}(\mathbb{R})}D(y_*,s_*)=\delta f(y_*,s_*)D(y_*,s_*).
	\end{equation*}
	On the other hand, $D(y_*,s_*)f(y_*,s_*)\geq 2m_0\lambda_D $, by the setting $f(y_*,s_*)=\|f(\cdot,s_*)\|_{L^{\infty}}=2m_0$ and \eqref{max_2_3}. 
	Thus, \eqref{max_2_3}, \eqref{max_2_4} imply
	\begin{equation*}
		\begin{split}
			(\partial_sf)(y_*,s_*)&\leq |F(y_*,s_*)|+\left(\left|\int_{\mathbb{R}} f(y',s_*)K(y_*,s_*;y')\,dy'\right|-D(y_*,s_*)f(y_*,s_*)\right)
			\\
			&\leq |F(y_*,s_*)|-(1-\delta)D(y_*,s_*)f(y_*,s_*)
			\\
			&\leq F_0-2(1-\delta)m_0\lambda_D <0.
		\end{split}
	\end{equation*}
	The last inequality holds by $m_0\lambda_D> F_0/(2-2\delta)$. This contradicts to \eqref{max_f'}.
	
	In case $f(y_*,s_*)<0$, similar computations hold. We have
	\begin{equation*}
		\left|\int_{\mathbb{R}} f(y',s_*)K(y_*,s_*;y')\,dy'\right|\leq \delta\|f(\cdot,s_*)\|_{L^{\infty}(\mathbb{R})}D(y_*,s_*)=-\delta f(y_*,s_*)D(y_*,s_*)
	\end{equation*}
	from \eqref{max_2_2}, and $D(y_*,s_*)f(y_*,s_*)\leq -2m_0\lambda_D $ holds by \eqref{max_2_3}. Combining with \eqref{max_2_3} and \eqref{max_2_4}, we obtain
	\begin{equation*}
		\begin{split}
			(\partial_sf)(y_*,s_*)&\geq -|F(y_*,s_*)|-\left(\left|\int_{\mathbb{R}} f(y',s_*)K(y_*,s_*;y')\,dy'\right|+D(y_*,s_*)f(y_*,s_*)\right)
			\\
			&\geq -|F(y_*,s_*)|-(1-\delta)D(y_*,s_*)f(y_*,s_*)
			\\
			&\geq -F_0+2(1-\delta)m_0\lambda_D >0.
		\end{split}
	\end{equation*}
	This contradicts to \eqref{max_f'}. 
\end{proof}

	Next, we present Lemma~\ref{rmk2}, yielding the decaying properties of the solutions to transport type equations  under proper assumptions.  This will be importantly used in several proofs in subsection~\ref{ch3.4}.

\begin{lemma}\label{rmk2}(Decaying properties as $|y|\rightarrow \infty$)
Let  $f$ be a smooth solution to the equation 
\begin{equation*}
	\partial_s f(y,s) + U(y,s)\partial_y f(y,s) +D(y,s)f(y,s) =F(y,s),
\end{equation*}  
where
$U$, $D$ and $F$ are smooth functions satisfying
\begin{subequations}
	\begin{align} 
		& \inf_{\{|y| \geq N ,\, s\in[s_0,\infty)\}}  U(y,s) \frac{y}{|y|}  > 0,\label{rmk2_ass0}
		\\
		& \inf_{\{|y| \geq N,\, s\in[s_0,\infty)\}}D(y,s)\geq \lambda_D,\label{rmk2_ass1}
		\\
		& \|F (\cdot, s) \|_{L^{\infty}(|y| \geq N)}\leq F_0e^{-s\lambda_F}\label{rmk2_ass2}
	\end{align}
\end{subequations}
for some {$ \lambda_D, \lambda_F, N  \geq  0$}.   Then it holds that  
\begin{equation*}
\begin{array}{l l}
\limsup_{| y |\rightarrow \infty}|f(y,s)|\leq \limsup_{|y|\rightarrow \infty}{|f(y,s_0)|}e^{-\lambda_D(s-s_0)}+\frac{F_0}{\lambda_D-\lambda_F}e^{-s\lambda_F} & \quad \text{if } \lambda_D>\lambda_F, \\
\limsup_{| y |\rightarrow \infty}|f( y ,s)|\leq \limsup_{|y|\rightarrow \infty}{|f(y,s_0)|}e^{-\lambda_D(s-s_0)}+\frac{F_0e^{-s_0\lambda_F}}{\lambda_F-\lambda_D}e^{-\lambda_D(s-s_0)} & \quad \text{if } \lambda_F>\lambda_D.
\end{array}
\end{equation*}
\end{lemma}

\begin{proof}
Let $\psi(y,s)$ be the solution to the equation $\partial_s\psi=U(\psi,s)$ with $\psi(y,s_0)=y$. We claim that
\begin{equation}\label{rmk2_claim1}
\begin{array}{l l}
\psi(y,s) > y & \quad \text{for }  y >N,  \;s>s_0,  \\
\psi(y,s) < y & \quad \text{for }  y<-N, \;s>s_0.
\end{array} 
\end{equation}
We only consider the case $y >N$ since the same argument applies to the case $y < -N$.  Suppose to the contrary that $\psi(\overline{y},\overline{s}) \leq \overline{y}$ for some $\overline{y}>N$ and $\overline{s}>s_0$. Then, there exists  $\overline{s}_\ast>s_0$ satisfying  $\psi(\overline{y},\overline{s}_\ast)= \overline{y}$ and $\psi(\overline{y},s) >  \overline{y}$ for $s\in(s_0,\overline{s}_\ast)$ since $\psi(\overline{y},s_0)=\overline{y}$ and 
\[
\partial_s\psi(\overline{y},s_0)=U(\psi(\overline{y},s_0),s_0) = U(\overline{y},s_0)>0
\]
due to \eqref{rmk2_ass0}. Now, the mean value theorem implies that there is $\overline{s}_{\ast\ast} \in (s_0,\overline{s}_{\ast})$ such that
\[
0 = \partial_s \psi(\overline{y},\overline{s}_{\ast\ast})  = U(\psi(\overline{y},\overline{s}_{\ast\ast}), \overline{s}_{\ast\ast}) >0,
\] 
 which is a contradiction. Here, the strict inequality holds due to \eqref{rmk2_ass0} and $\psi(\overline{y},\overline{s}_{\ast\ast}) > \overline{y} > N$. This proves \eqref{rmk2_claim1}.

	By integrating $f$ along $\psi$, we have
	\begin{equation*}
		f(\psi(y,s),s)=f(y,s_0)e^{-\int^s_{s_0}(D\circ\psi)\,ds'}+\int^s_{s_0}e^{-\int^s_{s'}(D\circ\psi)\,ds''}(F\circ\psi)\,ds'.
	\end{equation*}
	Since \eqref{rmk2_claim1} implies $|\psi(y,s)|>N$ for all $|y|>N$, 
	we have by \eqref{rmk2_ass1} and \eqref{rmk2_ass2} that for $|y|>N$ and $s \geq s_0$,
	\begin{equation*}
	\begin{array}{l l}
	|f(\psi(y,s),s)|\leq |f(y,s_0)|e^{-\lambda_D(s-s_0)}+\frac{F_0}{\lambda_D-\lambda_F}e^{-s\lambda_F} & \quad \text{if }\lambda_F<\lambda_D, \\
	|f(\psi(y,s),s)|\leq |f(y,s_0)|e^{-\lambda_D(s-s_0)}+\frac{F_0e^{-s_0\lambda_F}}{\lambda_F-\lambda_D}e^{-\lambda_D(s-s_0)} & \quad \text{if } \lambda_F>\lambda_D.
	\end{array}
	\end{equation*}
	
To finish the proof, we show that for each $s \geq s_0$, 
\begin{equation}\label{13}
\limsup_{|y|\rightarrow \infty}|f(\psi(y,s),s)| = \limsup_{|y|\rightarrow \infty}|f(y,s)|
\end{equation}
provided that $\limsup_{|y|\rightarrow \infty}|f(y,s)|$ is finite. We let $\psi_s(y)=\psi(y,s)$. For each $s \geq s_0$, the mapping $y \in \mathbb{R}\to \psi_s(y)\in \mathbb{R}$ is (strictly) increasing and onto.  
Using those properties of $\psi_s$, we claim that for any $n \geq 0$,
\begin{equation}\label{12}
\sup_{y\in\{|y| \geq n\}} | f(\psi_s(y),s) | = \sup_{z\in \{z\geq \psi_s(n)\} \cup \{z\leq \psi_s(-n)\} }|f(z,s)|.
\end{equation}
Indeed, since $\psi_s(y)$ is increasing in $y$, we have $|f(\psi_s(y),s) | \leq \sup_{ \{z\geq \psi_s(n)\} \cup \{z\leq \psi_s(-n)\} }|f(z,s)|$ for all $|y| \geq n$. 
On the other hand, for each $z \geq \psi_s(n)$, there is $\overline{y}$ such that $\psi_s(\overline{y})=z$ since $\psi_s$ is   onto, and thus, we have $\overline{y} \geq n$ since $\psi$ is increasing. Similarly, for each $z \leq \psi_s(-n)$, there is $\overline{y}$ such that $\psi_s(\overline{y})=z$ and $\overline{y} \leq -n$. Hence, we have 
\[
|f(z,s) |\leq \sup_{|y| \geq n} | f(\psi_s(y),\cdot) | \quad \text{for all } z\in \{z\geq \psi_s(n)\} \cup \{z\leq \psi_s(-n)\}.
\] 
This proves \eqref{12}.

From \eqref{rmk2_claim1}, it is clear that $a:= \limsup_{|y|\rightarrow \infty}|f(\psi_s(y),s)| \leq  \limsup_{|y|\rightarrow \infty}|f(y,s)| =:b$ holds. To show $a=b$, we suppose that $a <  b$. We choose sufficiently small $\veps>0$ satisfying $a+\veps < b-\veps$ such that for all sufficiently large positive numbers $N_1$ and $N_2$, the following holds:
	\begin{equation}\label{N1N2}
		\begin{split}
			\sup_{|y|\geq n} |f(\psi_s(y),s)| <a+\veps \qquad \text{if}\quad n\geq N_1,
			\\
			\sup_{|y|\geq n} |f(y,s)| >b-\veps \qquad \text{if}\quad n\geq N_2.
		\end{split}
	\end{equation} 
We choose sufficiently large $N_1, N_2$ satisfying $N_2 > \max\{\psi_s(N_1), -\psi_s(-N_1)\}$ so that 
	\begin{equation}\label{11}
		\sup_{|y|\geq N_1}|f(\psi_s(y),s)|= \sup_{ \{z\geq \psi_s(N_1)\} \cup \{z\leq \psi_s(-N_1)\} }|f(z,s)|\geq \sup_{|y|\geq N_2}|f(y,s)|
	\end{equation}
	holds, where the equality is from \eqref{12} and the inequality is due to \eqref{rmk2_claim1}. Combining \eqref{N1N2} and \eqref{11}, we obtain $a+\veps> b-\veps$, which is a contradiction. This proves \eqref{13}, and we finish the proof of Lemma~\ref{rmk2}.
\end{proof}

  \section*{Acknowledgments.}
B.K. was supported by Basic Science Research Program through the National Research Foundation of Korea (NRF) funded by the Ministry of science, ICT and future planning (NRF-2020R1A2C1A01009184).  J.B. was supported by the National Research Foundation of Korea grant funded by the Ministry of Science and ICT (NRF-2022R1C1C2005658).

 \end{document}